\documentclass[11pt,reqno]{amsart}
\usepackage{hyperref}
\usepackage{mathtools}
\usepackage{amssymb}
\usepackage{color}
\usepackage{physics}
\usepackage[left=3cm,right=3cm,top=3cm,bottom=3cm]{geometry}
\usepackage{enumitem}

\newtheorem{theorem}{Theorem}
\newtheorem{lemma}{Lemma}[section]
\newtheorem{proposition}{Proposition}[section]

\newtheorem{corollary}{Corollary}[section]
\newtheorem{definition}{Definition}[section]
\theoremstyle{remark}
\newtheorem{remark}{Remark}[section]
\numberwithin{equation}{section}

\newcommand{\R}{\mathbb{R}}

\newcommand{\N}{\mathbb{N}}
\newcommand{\tsum}{\textstyle\sum}
\numberwithin{equation}{section}
\def\bell{{\boldsymbol{\ell}}}
\def\by{{\boldsymbol{y}}}
\def\psl#1#2{\left(#1,#2 \right)_{L^2}}
\def\pse#1#2{\left(#1,#2 \right)_{\mathcal E}}
\def\psh#1#2{\left(#1,#2 \right)_{\dot H^1}}
\def\pshb#1#2{\left(#1,#2 \right)_{\dot H^1_\ell}}

\def\nol#1{\left\|#1 \right\|_{L^2}}
\def\noh#1{\left\|#1 \right\|_{\dot H^1}}
\def\nohb#1{\left\|#1 \right\|_{\dot H^1_\ell}}

\def\Nint{\mathcal N_{\Omega}}

\newcommand{\bg}{\mathbf{g}}
\newcommand{\bh}{\mathbf{h}}
\newcommand{\bbf}{\mathbf{f}}
\newcommand{\br}{\mathbf{r}}
\newcommand{\WW}{\mathbf{W}}
\newcommand{\XX}{\mathbf{X}}
\newcommand{\trans}{t}
\newcommand{\cC}{\mathcal{C}}
\newcommand{\cE}{\mathcal{E}}

\newcommand{\cG}{\mathcal{G}}
\newcommand{\cH}{\mathcal{H}}
\newcommand{\cK}{\mathcal{K}}

\newcommand{\mm}{m}
\newcommand{\pp}{p}
\newcommand{\pun}{\partial_1}

\newcommand{\cN}{\mathcal{N}}
\newcommand{\MS}{S}
\newcommand{\ba}{\boldsymbol{\alpha}}
\newcommand{\bs}{\boldsymbol\sigma}
\newcommand{\bR}{\mathbf R}
\newcommand{\bW}{\mathbf W}
\newcommand{\bX}{\mathbf X}

\begin{document}
\title[Classification of multi-solitons for critical wave equation]
{Classification of multi-solitons in one sense of time for the 5D energy-critical wave equation}
\author[Y. Martel]{Yvan Martel}
\address{Laboratoire de mathématiques de Versailles, UVSQ, CNRS
and Institut Universitaire de France,
45 avenue des Etats-Unis, 78035 Versailles}
\email{yvan.martel@uvsq.fr}
\author[F. Merle]{Frank Merle}
\address{Universit\'e de Cergy Pontoise and Institut des Hautes \'Etudes Scientifiques, AGM CNRS UMR8088, 95302 Cergy-Pontoise, France}
\email{merle@math.u-cergy.fr}
\begin{abstract}
We classify the multi-solitons, in the positive sense of time, of the focusing energy-critical wave equation in space dimension $5$, more precisely the solutions $u(t,x)$ of the equation
\[
\partial_t^2 u - \Delta u - |u|^{\frac 4{3}} u = 0, \quad (t,x)\in [T_0,\infty)\times \R^5,
\]
which satisfy the estimate
\[
\bigl\|\nabla_{t,x} \bigl(u(t) - \tsum_{k} W_k(t) \bigr) \bigr\|_{L^2}
\lesssim t^{-\frac12^+} \quad \hbox{for all $t\gg 1$}.
\]
Here, $\{W_k\}_k$ is any given finite family of traveling waves, based on the explicit ground state solution
\[
W(x) = \left(1+ \frac{|x|^2}{15}\right)^{-\frac32}
\]
with collinear two-by-two different velocities.

The proof combines the strategy introduced in \cite{Co} to prove a similar classification result for the supercritical generalized Korteweg-de Vries equation, with tools from \cite{DMwave,MMwave1,MMwave2} 
(mainly energy estimates) which are specific to the energy critical wave equation.
In particular, the construction of a finite family of multi-solitons, corresponding to the instability direction of each soliton, is a prerequisite of the classification result.
\end{abstract}

\maketitle

\section{Introduction}
\subsection{Main results}
We consider the focusing energy-critical nonlinear wave equation in dimension $5$
\begin{equation}\label{wave}
\left\{ \begin{aligned}
&\partial_t^2 u - \Delta u - |u|^{\frac 4{3}} u = 0, \quad (t,x)\in [0,\infty)\times \R^5,\\
& u_{|t=0} = u_0\in \dot H^1(\R^5),\\ 
& \partial_t u_{|t=0} = u_1\in L^2(\R^5),
\end{aligned}\right.
\end{equation}
which is written equivalently as a system for $\vec u = (u,\partial_t u) = (u,v)$,
\begin{equation*}
\begin{cases}
\partial_t u = v\\
\partial_t v = \Delta u + |u|^{\frac 4{3}} u 
\end{cases}
\end{equation*}
with $\vec u_{|t=0}=(u_0,u_1)$.

\medskip

We recall that the Cauchy problem for equation \eqref{wave} is locally well-posed in the energy space $\cE:=\dot H^1(\R^5)\times L^2(\R^5)$. 
See \emph{e.g.} \cite{BCLPZ,GiSoVe92,Kapitanski94, KM,LiSo95,Pecher84,Sogge95,ShSt94,ShSt98}.
The nonlinear wave equation~\eqref{wave} is energy-critical since it is invariant by the $\dot H^1$ scaling:
let $u$ be a solution of \eqref{wave} and $\lambda>0$, then the function $u_\lambda$ defined by
\[
u_\lambda(t,x)=\frac{1}{\lambda^{\frac32}}u\left(\frac{t}{\lambda},\frac{x}{\lambda}\right)
\]
is also a solution of \eqref{wave} and $\|u_\lambda(t)\|_{\dot H^ 1}=\|u(t/\lambda)\|_{\dot H^ 1}$.

\medskip

For any $\dot H^1\times L^2$ solution $\vec u=(u,\partial_t u)$ of \eqref{wave}, the energy 
$E(u(t),\partial_t u(t))$ and the momentum $M(u(t),\partial_t u(t))$ are conserved, where
\begin{equation}\label{eq:91}
\begin{aligned}
E(u,v) & = \frac 12 \int v^2 + \frac 12 \int |\nabla u|^2
- \frac {3}{10} \int |u|^{\frac {10}{3}},\\
M(u,v) & = \int v\nabla u.
\end{aligned}
\end{equation}
Recall that the function $W$ defined by
\begin{equation}\label{defW}
W(x) = \left( 1+ \frac {|x|^2}{15}\right)^{-\frac{3}2}
\end{equation}
satisfies
\begin{equation}
\Delta W + W^{\frac 73}=0 , \quad x\in \R^5,
\end{equation}
and so it is a stationary solution, called \emph{standing soliton}, of \eqref{wave}.
Using the Lorentz transformation on the function $W$, we obtain \emph{traveling solitons}:
for any $\boldsymbol{\ell}\in \R^5$, with $|\boldsymbol{\ell}|< 1$, let
\begin{equation}\label{defWbb}
W_{\bell}(x)=W\left(\left(\frac{1}{\sqrt{1-|\boldsymbol{\ell}|^2}}-1\right) \frac{\boldsymbol{\ell}(\boldsymbol{\ell}\cdot x)}{|\boldsymbol\ell|^2} +x\right),
\end{equation}
then $u(t,x)=\pm W_{\boldsymbol{\ell}}(x-{\boldsymbol{\ell}} t)$ is a solution of \eqref{wave}.

\medskip

In this article, we only consider the special case $\bell=\ell e_1$, where $e_1=(1,0,0,0,0)\in \R^5$, $\ell\in (-1,1)$, and we use the simplified notation
\[
W_\ell(x) = W \left(\frac{x_1}{\sqrt{1-\ell^2}}, \overline x \right), \quad \overline x = (x_2,\ldots,x_5).
\]
See Remark \ref{rk:ll} for a discussion on this restriction.

\medskip

First, we recall from \cite{MMwave1} the following result of existence of multi-solitons.

\begin{theorem}[Existence of multi-solitons {\cite[Theorem 1]{MMwave1}}]\label{th:0}
Let $K\geq 2$. For $k\in \{1,\ldots,K\}$, let $\lambda_k^\infty>0$, $\by_k^\infty\in \R^5$, 
$\epsilon_k=\pm 1$ and $\ell_k\in (-1,1)$ with $\ell_k\neq \ell_{m}$ for $k\neq m$.
Let
\[
W_k^\infty(t,x)= \frac {\epsilon_k}{(\lambda_k^\infty)^{\frac 32}} 
W_{\ell_k} \left( \frac{x - \ell_k t e_1 -{\by}^\infty_k}{\lambda_k^\infty}\right).
\]
Then, there exists $T_0>0$ and a solution $\MS$ of \eqref{wave} on $[T_0,+\infty)$
in the energy space such that
\begin{equation*}
\left\|\nabla_{t,x} \left(\MS - \sum_{k=1}^K W_k^\infty \right) (t)\right\|_{L^2} \lesssim t^{-1}.
\end{equation*}
\end{theorem}

Theorem \ref{th:0} is a basic existence result where the decay $t^{-1}$ 
is optimal when comparing the multi-soliton $S$ to a sum of exact traveling waves.
In Definition \ref{SMS} and Proposition \ref{pr:vp}, 
we will recall from \cite{MMwave1,MMwave2} 
more refined properties of the multi-soliton $\MS$, in particular its regularity and
its proximity at order $t^{-2}$ to a sum of suitably modulated solitons.

\medskip

Recall that \cite[Theorem 1]{MMwave1} was proved after a series of other constructions of 
multi-solitons for several nonlinear dispersive equations
(\cite{CMM,CMkg,KMR,Ma,MMnls,Me}).
We point out \cite{CMkg} where the strategy of \cite{CMM,Ma,MMnls} has been extended to wave-type equations (namely, for the Klein-Gordon equation).The articles \cite{MMwave1,MMwave2} had to deal with the additional 
difficulty of the low decay rate of the ground state \eqref{defW}.

\medskip

Inspired by the  article \cite{Co} devoted to the classification of multi-solitons of the supercritical generalized KdV equation, 
our goal is to make an advance on the classification of the multi-solitons of equation \eqref{wave}
 in the context of Theorem \ref{th:0}.

\medskip

We start with the following multi-existence result, the parameters of the solitons being fixed.

\begin{theorem}[Multi-existence of multi-solitons]\label{th:1}
Under the assumptions of Theorem \ref{th:0}, 
there exists a $K$-parameter family $\{\MS_{\bf A}\}_{{\bf A}\in \R^K}$
of solutions of \eqref{wave} such that for all ${\bf A}\in \R^K$,
for $t$ large enough, 
\begin{equation*}
\left\|\nabla_{t,x} \left(\MS_{\bf A} - \sum_{k=1}^K W_k^\infty \right) (t)\right\|_{L^2} \lesssim t^{-1}
\end{equation*}
and if $\tilde {\bf A}\neq {\bf A}$ then 
$\MS_{\tilde {\bf A}}\neq \MS_{\bf A}$.
\end{theorem}
We point out that as in \cite{Co}, such a multi-existence result is a consequence of an instability direction (backwards in time) attached to the dynamics of each soliton. 
We also refer to \cite{DMwave} for the case of the single soliton of \eqref{wave}.
The solutions constructed in Theorem \ref{th:1} are shown to be strong multi-solitons
in the sense of Definition \ref{SMS}.

\medskip

Second, we prove the following classification result, under a mild decay assumption in time.
\begin{theorem}[Classification of multi-solitons]\label{th:2} 
Under the assumptions of Theorem \ref{th:0}, if 
$u(t)$ is a solution of \eqref{wave} on a certain time interval $[T,+\infty)$ satisfying that for some $\delta>0$,
and for all $t\geq T$,
\begin{equation}\label{eq:hy}
\left\|\nabla_{t,x} \left(u - \sum_{k=1}^K W_k^\infty \right) (t)\right\|_{L^2} \lesssim t^{-\frac12-\delta},
\end{equation}
then there exists ${\bf A}\in \R^K$ such that $u\equiv \MS_{\bf A}$.
\end{theorem}

The proofs of Theorems \ref{th:1} and \ref{th:2} consist of adapting the general strategy introduced
in \cite{Co} to the context of 
the energy methods developed in \cite{DMwave,MMwave1,MMwave2} for the energy-critical wave equation.
See Section \ref{S:1.2} for the organisation of the proofs.

\begin{remark}
We comment on the mild decay assumption in \eqref{eq:hy}.
For equation \eqref{wave}, it remains an open question to establish the classification result under the weaker assumption
\[
\lim_{t\to+\infty} \left\|\nabla_{t,x} \left(u -\sum_{k=1}^K W_k^\infty \right) (t)\right\|_{L^2} =0.
\]
For the generalized KdV equation,
in the uniqueness result of \cite{Ma} (unconditional uniqueness in the subcritical and
critical case) and in the classification result proved in \cite{Co} (supercritical case),
there is no need of supposing an explicit decay rate as in \eqref{eq:hy},
and convergence to zero is sufficient.
This is due to specific monotonicity properties derived in the case of the generalized KdV equation.
For the nonlinear Klein-Gordon equation in $3$D,
\cite[Theorem 1.9]{ChJe} provides an unconditionnal classification result 
of the family of (pure) multi-solitons.
The method in \cite{ChJe} is based on Strichartz estimates for a wave operator with moving
potential developed in \cite{ChJ1}, which also allows to prove a general asymptotic stability result
on multi-soliton solutions.
Note that the above results also rely on the exponential decay property of the solitons.

\medskip

We refer to \cite{CoFr,RoSZ} for other results of conditional uniqueness of multi-solitons,
in the context of nonlinear Schrödinger equations.
\end{remark}

\begin{remark}
The classification in Theorem \ref{th:2} combined with the inelasticity result
of \cite{MMwave2} implies that a multi-soliton satisfying \eqref{eq:hy} 
(in the positive sense of time) cannot be 
a multi-soliton in the other sense of time.
See Corollary \ref{co:in}.
Such a global classification result is expected to play a role in the soliton resolution
conjecture, see for example \cite{DKM1,DKM2,JL18} for such questions.
We also refer to the review \cite{CB} on the soliton resolution for energy critical wave equations.

\medskip

For the critical wave equation in dimensions larger than~$6$, we refer to \cite{Ja19,JL18}
for global results on radial multi-bubbles.
\end{remark} 

\begin{remark}\label{rk:ll}
Being based on the techniques introduced in the previous articles~\cite{MMwave1,MMwave2}
(energy estimates specific to the nonlinear wave equation around multi-solitons)
the results of the present paper are restricted to the case of collinear speeds.
By rotation invariance, we assume without loss of generality that all the speeds are along the direction $e_1$.

\medskip

For a general multi-soliton existence result, we refer to our recent article \cite{MMwave3}.
We expect that the new energy functional introduced in \cite{MMwave3} in order to address the case of general
speeds can also be used to prove a classification result.
However, this functional is not defined in the energy space and 
requires additional decay in space on the solution.
This is not problematic for a construction result, but 
its use would certainly restrict any classification result to a space strictly included in the energy space.
\end{remark}

\subsection{Organisation of the article}\label{S:1.2}

In Section \ref{S:2}, we quickly recall some notation and basic information on solitons and multi-solitons from \cite{MMwave1,MMwave2}.

\medskip

In Section \ref{S:3}, we introduce the notion of \emph{strong multi-soliton}, inspired by the properties of the multi-solitons constructed in \cite{MMwave2}. We also recall the definition of the \emph{refined approximate multi-soliton} $\WW$ introduced in \cite{MMwave2} in order to prove that 
any multi-soliton satisfying \eqref{eq:hy} is a strong multi-soliton (pre-classification result Proposition \ref{pr:2}).
The proof follows rather directly from the techniques introduced in \cite{MMwave2}. This result allows to reduce the classification of multi-solitons
satisfying \eqref{eq:hy} to the classification of strong multi-solitons.

\medskip

In Section \ref{S:4}, we start by fixing a strong multi-soliton $S$ and
we introduce a refined decomposition of any multi-soliton $u$ around $S$. Indeed, to construct the whole family of multi-soliton (Theorem \ref{th:1})
and to classify strong multi-solitons (Theorem \ref{th:2}), it is essential to rely on an exact strong multi-soliton $S$ and not on an approximate
multi-soliton (the approximate multi-soliton $\WW$ introduced in Section \ref{S:3} would be easier to handle but it satisfies the equation 
only up to some power of $t^{-1}$, which is not sharp enough to distinguish multi-solitons).

\noindent  Section \ref{S:4} also includes the introduction of
an energy functional $\cH$ required to control the difference between two multi-solitons (this requires some minimal regularity and decay properties on one of the two multi-solitons, thus motivating the notion of strong multi-soliton).

\medskip

In Section \ref{S:5}, we construct a family of multi-solitons (Theorem \ref{th:1})
inspired by \cite{Co} and using the framework of Section \ref{S:4}.
A subtlety here lies on the fact that the various multi-solitons differ among them by a factor which is exponentially small in time, while they are all at a distance $t^{-2}$ of a sum of (modulated) solitons.

\medskip

Finally, in Section \ref{S:6}, we classify all the strong multi-solitons of \eqref{wave}, following the strategy 
introduced in \cite{Co}. This is
the most delicate step of this article from the technical point of view: to pass from an estimate on the difference of two multi-solitons in powers of 
$t^{-1}$ to an exponential estimate, one needs sharp energy estimates (this is why the decomposition of Section \ref{S:4} is somehow involved) as well as
technical lemmas for \emph{perturbed hyperbolic configurations} (see Lemmas \ref{le:td}, \ref{le:te}).

\section{Preliminaries}\label{S:2}
The content of this section is mainly taken from \cite{MMwave1,MMwave2} for the reader's convenience.
The proofs are omitted.

\subsection{Notation}
We denote
\[
\psl g {\tilde g} =\int g \tilde g,\quad \nol g^2 = \int |g|^2,\quad \psh g {\tilde g} =\int \nabla g \cdot \nabla {\tilde g} ,
\quad \noh g^2=\int |\nabla g|^2.
\]
For 
\[
\vec g = \begin{pmatrix} g \\h\end{pmatrix} ,
\ \vec {\tilde g} = \begin{pmatrix}\tilde g \\\tilde h\end{pmatrix},
\]
set
\begin{align*}
 \psl {\vec g} {\vec {\tilde g}} = \psl g {\tilde g} + \psl h{\tilde h},\quad
 \pse {\vec g} {\vec {\tilde g}} = \psh g {\tilde g} + \psl h{\tilde h},\quad 
\|\vec g\|_\cE^2 = \noh g^2 + \nol h^2.
\end{align*}
We will also use the notation $\partial_j=\partial_{x_j}$, for $j=1,\ldots,5$
and for $\beta=(\beta_1,\ldots,\beta_5)\in \N^5$, the notation
$\partial^\beta=\partial_1^{\beta_1}\cdots\partial_5^{\beta_5}$,
$|\beta|=\sum_{j=1}^5 \beta_j$.
When there is no risk of confusion, we write 
$\sum_j=\sum_{j=1}^5$.

Since $x_1$ is a specific coordinate, we denote
\[
\overline x = (x_2,\ldots, x_5),
\quad \overline \nabla g = (\partial_{x_2} g, \ldots, \partial_{x_5} g),
\quad \overline \Delta g = \sum_{j=2}^5 \partial_j^2 g.
\]
For $-1<\ell<1$, let
\[
\pshb g{\tilde g} =(1-\ell^2)\int \pun g \pun \tilde g+ \int \overline \nabla g \cdot \overline \nabla \tilde g, 
\quad \nohb g^2= \pshb gg,
\]
and
\[ 
A_\ell=\partial_t +\ell \partial_1,\quad
\Delta_\ell = (1-\ell^2) \partial_1^2+ \overline \Delta.
\]
Let $\Lambda$ and $\widetilde \Lambda$ be the $\dot H^1$ and $L^2$ scaling operators defined as follows
\begin{equation}
\label{aL}
\Lambda g = \frac 32 g+ x \cdot \nabla g,
\quad \widetilde \Lambda g = \frac 5 2 g + x \cdot \nabla g,\quad
\widetilde\Lambda \nabla=\nabla\Lambda,\quad 
\vec \Lambda = \begin{pmatrix} \Lambda \\ \widetilde \Lambda \end{pmatrix},
\quad 
\vec {\widetilde \Lambda} = \begin{pmatrix} \widetilde \Lambda \\ \Lambda \end{pmatrix}.
\end{equation}
Recall the Hardy and Sobolev inequalities, used frequently in this article, for any $v\in \dot H^1$,
\begin{equation}\label{z2}
\int \frac {|v|^2}{|x|^2} \lesssim \int |\nabla v|^2,
\end{equation}
\begin{equation}\label{z1}
\|v\|_{L^{10/3}} \lesssim \|\nabla v\|_{L^2}.
\end{equation}
Denote
\[
F(u) = \frac {3}{10}\cdot |u|^{\frac {10}3},\quad 
f(u)= |u|^{\frac 43} u,\quad f'(u)=\frac 73\cdot |u|^\frac43,\quad
f''(u)=\frac73\cdot\frac43\cdot |u|^{-\frac23} u.
\]
For $0<\gamma \ll1$, we set
\begin{equation}\label{phia}
\varphi_\gamma (x) = (1+|x|^2)^{- \gamma}.
\end{equation}

\subsection{Energy linearization around the standing soliton}
Let
\begin{align*}
& L = -\Delta - f'(W) ,\quad 
\psl{L g}g = \int |\nabla g|^2 - f'(W) g^2,
\\
& H = \begin{pmatrix} L & 0 \\0 & {\rm Id}\end{pmatrix},\quad 
\psl{H \vec g} {\vec g} = \psl{L g}g+ \nol h^2,\\
&  J=\begin{pmatrix}0 & {\rm Id} \\-{\rm Id} & 0\end{pmatrix}.
\end{align*}
If $\vec g$ is a small function in the energy space, then it is easily checked (see \eqref{eq:91}) that
\begin{equation}
E(W+g,h) =E(W,0)+\frac 12\psl{L g}g+ \frac 12 \nol h^2 + O(\noh g^3).\label{enerlin}
\end{equation}

We gather here some properties of the operator $L$ (see \cite{DMwave,MMwave1}).

\begin{lemma}[Spectral properties of $L$]\label{le:Q}
The operator $L$ on $L^2$ with domain $H^2$ is a self-adjoint operator with essential spectrum $[0,+\infty)$, no positive eigenvalue and only one negative eigenvalue $-\lambda_0$, associated to a smooth radial positive eigenfunction $Y \in \mathcal C^\infty(\R^5)$
satisfying, for all $p\geq 0$,
\begin{equation}\label{eq:dY}
|Y^{(p)}(x)|\lesssim e^{-\sqrt{\lambda_0}|x|}.
\end{equation}
Moreover, $L (\Lambda W) = L (\partial_j W) =0$, for any $j=1,\ldots,5$.

Furthermore, there exists $\mu>0$ such that, for all $g \in \dot H^1$, the following hold.
\begin{enumerate}[label=\emph{(\roman*)}]
\item Coercivity.
\begin{equation*}
\psl {Lg}g\geq \mu \noh g^2 -\frac 1{\mu} \left( \psh g{\Lambda W}^2 + \sum_j \psh g{\partial_j W}^2+\psl g{Y}^2\right).
\end{equation*}
\item Localized coercivity. For $\gamma>0$ small enough,
\begin{equation*}
\int |\nabla g|^2 \varphi_\gamma^2- f'(W) g^2 \geq \mu \int |\nabla g|^2 \varphi_\gamma^2 
-\frac 1{\mu} \left( \psh g{\Lambda W}^2 + \sum_j \psh g{\partial_j W}^2+\psl g{Y}^2\right).
\end{equation*}
\end{enumerate}
\end{lemma}

\subsection{Energy linearization with Lorentz boost}

For any $-1<\ell<1$, let
\begin{equation}\label{eqWb}
W_{\ell }(x) = W\left(\frac {x_1}{\sqrt{1-\ell^2}}, \overline x\right),\quad 
\vec W_{\ell} = \begin{pmatrix} W_\ell \\ -\ell \pun W_\ell\end{pmatrix}
\end{equation}
so that
\[
\Delta_\ell W_{\ell} + W_\ell^{\frac 73}=0,
\]
and $u(t,x) = W_{\ell }\left(x_1 - \ell t,\overline x\right)$ is a traveling wave solution of \eqref{wave}. 

Let
\begin{equation}\label{Hbeta}
\begin{aligned}
& L_{\ell} = -\Delta_\ell - f'(W_\ell) ,\\
& \psl{L_\ell g}g = \int  (1-\ell^2)  |\pun g|^2+  |\overline \nabla g|^2 -f'(W_\ell) g^2  
,\\
& H_\ell = \begin{pmatrix} -\Delta -f'(W_\ell) & -\ell \pun \\ \ell \pun & {\rm Id}\end{pmatrix},\\
& \psl {H_\ell \vec g}{\vec g} = 
\psl{L_\ell g}g + \| \ell \pun g + h\|_{L^2}^2.
\end{aligned}
\end{equation}
As before, $L_\ell$ and $H_\ell$ are related to the linearization of the energy around $W_\ell$
by the expansion
\begin{align*}
 & E(W_\ell + g, -\ell \pun W_\ell+h) + \ell \int \pun (W_\ell+g) (-\ell \pun W_\ell + h) \\
 &\quad = (1-\ell^2)^{\frac 12} E(W , 0 ) +\frac 12 \psl {H_\ell \vec g}{\vec g} + O(\noh g^3).
\end{align*}
The following functions will be needed in relation with the operators $H_\ell$ and $H_\ell J$
\begin{align*}
&\vec Z_\ell^\Lambda = \begin{pmatrix} \Lambda W_\ell \\ - \ell \pun \Lambda W_\ell\end{pmatrix},\quad
\vec Z_\ell^{(j)} = \begin{pmatrix} \partial_j W_\ell \\ - \ell \pun \partial_j W_\ell\end{pmatrix},\\
&Y_{\ell}(x) = Y\left(\frac {x_1}{\sqrt{1-\ell^2}},\overline x\right),
\quad 
\vec Z_\ell^\pm = \begin{pmatrix} \left(\ell \pun Y_\ell 
\pm \frac {\sqrt{\lambda_0}}{\sqrt{1-\ell^2}} Y_\ell\right) e^{\pm \frac {\ell \sqrt{\lambda_0}}{\sqrt{1-\ell^2}}x_1} \\ 
Y_\ell e^{\pm \frac {\ell \sqrt{\lambda_0}}{\sqrt{1-\ell^2}}x_1 } \end{pmatrix}.
\end{align*}
We recall the following simple properties. 
\begin{lemma}\label{le:22} The following hold for any $-1<\ell<1$.
\begin{enumerate}[label=\emph{(\roman*)}]
\item Properties of $L_\ell$.
\begin{equation}\label{LW}
  L_\ell (\Lambda W_\ell) = L_\ell (\partial_j W_\ell)=0,\quad L_\ell Y_\ell = -\lambda_0 Y_\ell,\quad 
L_\ell W_\ell = - \frac 43 W_\ell^{\frac 73}.
\end{equation}
\item Properties of $H_\ell$ and $H_\ell J$.
\begin{align}\label{ZW}
&H_\ell \vec Z_\ell^\Lambda = H_\ell \vec Z_\ell^{(j)}=0,\\
\label{oZ} 
&\psl { \vec Z_\ell^\Lambda}{\vec Z_\ell^\pm} = \psl { \vec Z_\ell^{(j)}}{\vec Z_\ell^\pm}=0,
\\ \label{Zpm}
&- H_\ell J (\vec Z_\ell^\pm)  = \pm \sqrt{\lambda_0} (1-\ell^2)^{\frac 12} \vec Z_\ell^\pm,
\\ \label{eq:PM}
&\psl { J \vec Z_\ell^+}{\vec Z_\ell^-} =
-2 \frac{\sqrt{\lambda_0}}{\sqrt{1-\ell^2}} \|Y_\ell\|_{L^2}^2.
\end{align}
\end{enumerate}
\end{lemma}
The identity \eqref{eq:PM} is new but it is easily checked
\begin{align*}
\psl { J \vec Z_\ell^+}{\vec Z_\ell^-} &=
\begin{pmatrix} 
Y_\ell e^{ \frac {\ell \sqrt{\lambda_0}}{\sqrt{1-\ell^2}}x_1 } \\
-\left(\ell \pun Y_\ell 
+ \frac {\sqrt{\lambda_0}}{\sqrt{1-\ell^2}} Y_\ell\right) e^{\frac {\ell \sqrt{\lambda_0}}{\sqrt{1-\ell^2}}x_1} \\ \end{pmatrix}
\begin{pmatrix} \left(\ell \pun Y_\ell 
- \frac {\sqrt{\lambda_0}}{\sqrt{1-\ell^2}} Y_\ell\right) e^{-\frac {\ell \sqrt{\lambda_0}}{\sqrt{1-\ell^2}}x_1} \\ 
Y_\ell e^{- \frac {\ell \sqrt{\lambda_0}}{\sqrt{1-\ell^2}}x_1 } \end{pmatrix}\\
& = -2 \frac{\sqrt{\lambda_0}}{\sqrt{1-\ell^2}} \|Y_\ell\|_{L^2}^2.
\end{align*}
Useful coercivity results involve the directions $\vec Z_\ell^\pm$ instead of $Y_\ell$.
Such results are stated now (from \cite{MMwave1,MMwave2}).
\begin{lemma}\label{pr:22}
Let $-1<\ell<1$. There exists $\mu>0$ such that, for all $\vec g \in \dot H^1\times L^2$, the following hold.
\begin{enumerate}[label=\emph{(\roman*)}]
\item Coercivity of $H_\ell$.
\begin{align}
\psl {H_\ell \vec g}{\vec g} &\geq \mu \|\vec g\|_\cE^2 
- \frac 1{\mu}\biggl( \pshb{g}{\Lambda W_\ell}^2 +\sum_j \pshb {g}{\partial_jW_\ell}^2
+ \psl {\vec g}{\vec Z_\ell^{+}}^2 + \psl {\vec g}{\vec Z_\ell^{-}}^2\biggr)\label{eq:22}.\end{align}
\item Localized coercivity. For $\gamma>0$ small enough,
\begin{align}
& \int \left(| \nabla g|^2 \varphi_\gamma^2 -f'(W_\ell) g^2 + h^2 \varphi_\gamma^2
 + 2 \ell (\pun g) h \varphi_\gamma^2 \right)
\nonumber \\ &\geq \mu \int \left(|\nabla g|^2 + h^2 \right)\varphi_\gamma^2 
 - \frac 1{\mu }\biggl(\pshb{g}{\Lambda W_\ell}^2
+\sum_j \pshb {g}{\partial_j W_\ell}^2
+ \psl {\vec g}{\vec Z_\ell^{+}}^2 + \psl {\vec g}{\vec Z_\ell^{-}}^2\biggr)\label{eq:2.30}.\end{align}
\end{enumerate}
\end{lemma}

\subsection{Second order relations}
For future reference, we introduce functions related to the second variation along the null directions of
the operator $L_\ell$.
Let 
\begin{align*}
\vec Z_\ell^{\Lambda\Lambda} = \begin{pmatrix} \Lambda^2 W_\ell \\ - \ell \pun \Lambda^2 W_\ell\end{pmatrix},\quad
\vec Z_\ell^{(j)\Lambda} = \begin{pmatrix} \partial_j \Lambda W_\ell \\ - \ell \pun \partial_j \Lambda W_\ell\end{pmatrix},
\quad
\vec Z_\ell^{(j')(j)} = \begin{pmatrix} \partial_{j'} \partial_j W_\ell \\ - \ell \pun \partial_{j'} \partial_j W_\ell\end{pmatrix}.
\end{align*}
\begin{lemma}
The following hold for any $-1<\ell<1$,
\begin{equation}\label{L2}
\begin{aligned}
& L_\ell \Lambda^2 W_\ell = f''(W_\ell) (\Lambda W_\ell)^2 ,\quad
L_\ell \partial_j \Lambda W_\ell = f''(W_\ell) \partial_j W_\ell \Lambda W_\ell  ,\\
& L_\ell \partial_{j'}\partial_j W_\ell = f''(W_\ell) \partial_j W_\ell \partial_{j'} W_\ell.
\end{aligned}
\end{equation}
As a consequence,
\begin{equation}\label{H2}
\begin{aligned} 
&H_\ell \vec Z_\ell^{\Lambda\Lambda} =  \begin{pmatrix} f''(W_\ell) (\Lambda W_\ell)^2 \\ 0\end{pmatrix},\quad
H_\ell \vec Z_\ell^{(j)\Lambda} =  \begin{pmatrix} f''(W_\ell)\partial_j W_\ell \Lambda W_\ell  \\ 0\end{pmatrix},\\
&H_\ell \vec Z_\ell^{(j')(j)} =   \begin{pmatrix} f''(W_\ell)\partial_j W_\ell \partial_{j'} W_\ell \\ 0\end{pmatrix}. 
\end{aligned}
\end{equation}
\end{lemma}
\begin{proof}
Setting
\[
W_{\ell,\lambda} = \lambda^{-\frac32} W_\ell\big(\lambda^{-1} x\big),
\]
and differentiating the equation
\[
\Delta_\ell W_{\ell,\lambda} + W_{\ell,\lambda}^\frac73=0
\]
at the second order in $x$ or $\lambda$, one easily derives \eqref{L2}.
The identities \eqref{H2} follow.
\end{proof}

\subsection{Notation for general multi-solitons}
For the rest of this paper, we follow the notation of Theorem \ref{th:0}. We fix $K\geq 2$,
and for any $k\in \{1,\ldots,K\}$, we let $\lambda^\infty_k>0$, $\by_k^\infty \in \R^5$, $\epsilon_k=\pm 1$ and
$\ell_k\in (-1,1)$ with $\ell_{m}\neq \ell_k$ for $m\neq k$.

For $\vec G= (G,H)^\trans$, we set
\begin{equation}\label{thetai}
\begin{aligned}
&(\theta_k^\infty G)(t,x) = \frac {\epsilon_k}{(\lambda_k^\infty)^\frac32} G\left(\frac {x -\ell_k t e_1 - \by_k^\infty}{\lambda_k^\infty}\right),\\
& \vec \theta_k^\infty \vec G 
= \begin{pmatrix}\theta_k^\infty G \\[.2cm] 
\displaystyle \frac {\theta_k^\infty} {\lambda_k^\infty} H \end{pmatrix},
\quad 
\vec {\tilde \theta}_k^\infty \vec G 
= \begin{pmatrix}\displaystyle \frac {\theta_k^\infty} {\lambda_k^\infty} G \\[.4cm] 
\theta_k^\infty H \end{pmatrix},\\
&\Lambda_k^\infty G=\frac 32 G + (x-\ell_k t e_1 - \by_k^\infty) \cdot \nabla G.
\end{aligned}
\end{equation}
We set
\begin{equation}\label{defWi}
W_k^\infty =\theta_k^\infty W_{\ell_k}, \quad
\vec W_k^\infty = \vec\theta_k^\infty \vec W_{\ell_k}
=\begin{pmatrix} W_k^\infty \\ - \ell \pun W_k^\infty \end{pmatrix}.
\end{equation}
Now, we introduce time dependent modulation.
For $C^1$ functions $\lambda_k(t)>0$, $\by_k(t)\in \R^5$ to be chosen later, let
\begin{equation}\label{thetak}
\begin{aligned}
& (\theta_k G)(t,x) = \frac {\epsilon_k}{\lambda_k^\frac32(t)} G\left(\frac {x -\ell_k t e_1- \by_k(t)}{\lambda_k(t)}\right),
\\
& \vec \theta_k \vec G = \begin{pmatrix}\theta_k G \\[.2cm] \displaystyle \frac {\theta_k} {\lambda_k } H\end{pmatrix},
\quad
\vec {\tilde \theta}_k \vec G = \begin{pmatrix}\displaystyle \frac {\theta_k}{\lambda_k} G\\[.4cm] \theta_k H \end{pmatrix}.
\end{aligned}
\end{equation}
Let
\begin{align*}
\Lambda_k g & = \frac 32 G+ (x -\ell_k t e_1 -\by_k) \cdot \nabla G,\\
\widetilde \Lambda_k G & = \frac 5 2 G + (x -\ell_k te_1 -\by_k) \cdot \nabla G,\\
\vec \Lambda_k & = \begin{pmatrix} \Lambda_k \\ \widetilde \Lambda_k \end{pmatrix},
\quad 
\vec {\widetilde \Lambda}_k = \begin{pmatrix} \widetilde \Lambda_k \\ \Lambda_k \end{pmatrix}.
\end{align*}
We observe that for a time independent function $G$,  by \eqref{aL}, the following holds
\begin{equation}\label{eq:tL}
\partial_1 \Lambda_k \theta_k G = \widetilde\Lambda_k \frac{\theta_k}{\lambda_k}\partial_1 G,\quad
\partial_1 \nabla \theta_k G = \nabla \frac{\theta_k}{\lambda_k}\partial_1 G,
\end{equation}
\begin{equation}\label{eq:tG}
\begin{aligned}
\partial_t \theta_k G  
&= -\frac{\ell_k}{\lambda_k}\theta_k \pun G-\frac{\dot \lambda_k}{\lambda_k} \theta_k \Lambda G 
- \frac{\dot \by_k}{\lambda_k}\cdot \theta_k \nabla G\\
&= -\ell_k \pun \theta_k G-\frac{\dot \lambda_k}{\lambda_k} \Lambda_k \theta_k G - \dot \by_k \cdot \nabla \theta_k G ,
\end{aligned}
\end{equation}
and for a vector-valued function $\vec G$,
\begin{align}
\begin{aligned}
\partial_t \vec \theta_k \vec G 
&= -\frac{\ell_k}{\lambda_k}\vec \theta_k \pun \vec G-\frac{\dot \lambda_k}{\lambda_k} \vec \theta_k \vec \Lambda \vec G
- \frac{\dot \by_k}{\lambda_k}\cdot \vec \theta_k\nabla \vec G\\
&= -\ell_k \pun \vec\theta_k \vec G-\frac{\dot \lambda_k}{\lambda_k} \vec \Lambda_k \vec \theta_k \vec G
- \dot \by_k \cdot \nabla \vec \theta_k \vec G,
\end{aligned}\label{eq:tg} \\
\begin{aligned}
\partial_t \vec {\tilde \theta}_k \vec G  
&= -\frac{\ell_k}{\lambda_k}\vec{\tilde \theta}_k \pun \vec G
-\frac{\dot \lambda_k}{\lambda_k} \vec {\tilde \theta}_k \vec{\widetilde \Lambda}  \vec G
- \frac{\dot \by_k}{\lambda_k}\cdot \vec {\tilde \theta}_k \nabla \vec G\\
&= -\ell_k \pun \vec{\tilde \theta}_k \vec G-\frac{\dot \lambda_k}{\lambda_k} \vec{\widetilde \Lambda}_k \vec {\tilde \theta}_k \vec G
- \dot \by_k \cdot \nabla \vec {\tilde \theta}_k \vec G.
\end{aligned}\label{eq:th}
\end{align}
Moreover, we set
\begin{equation}\label{defWk}
\begin{aligned}
&W_k = \theta_k W_{\ell_k} ,\quad
\vec W_k = \vec\theta_k \vec W_{\ell_k},\\ 
&\vec Z_k^\Lambda = \vec {\theta}_k \vec Z_{\ell_k}^\Lambda,
\quad \vec Z_k^{(j)} = \vec {\theta}_k \vec Z^{(j)}_{\ell_k},\\
&\vec Z_k^{\Lambda\Lambda} = \vec {\theta}_k \vec Z_{\ell_k}^{\Lambda\Lambda},
\quad \vec Z_k^{(j)\Lambda} = \vec {\theta}_k \vec Z^{(j)\Lambda}_{\ell_k},
\quad \vec Z_k^{(j')(j)} = \vec {\theta}_k \vec Z^{(j')(j)}_{\ell_k},\\
& \vec Z_k^\pm = \vec {\tilde \theta}_k \vec Z_{\ell_k}^\pm.
\end{aligned}
\end{equation}
When there is no risk of confusion, $\sum_{k=1}^K$ will be denoted simply by $\sum_k$.

Let
\begin{equation}\label{eq:ok}
\omega_k(t,x)=(1+ |x-\ell_k t e_1 -\by_k(t)|^2)^{-\frac12},\quad
\omega = \sum_k \omega_k.
\end{equation}
For future reference, by the definition and the decay properties of the function $W$, we observe that
for any $\beta\in \N^5$, 
\begin{equation}\label{eq:dW}
|\partial^\beta \vec W_k|+|\partial^\beta \vec\Lambda_k\vec W_k|
+|\partial^\beta \vec{\widetilde\Lambda}_k\vec W_k| \lesssim \omega_k^{3+|\beta|}.
\end{equation}
We also set
\[
\zeta_k (t,x)= e^{-\mu_0 |x-\ell_k t e_1-\by_k(t)|}, \quad \zeta=\sum_k \zeta_k .
\]
By the definition of $\vec Z_\ell^\pm$ and the decay property \eqref{eq:dY} of the function $Y$, 
we observe that
for any $\beta\in \N^5$,
\begin{equation*}
|\partial^\beta\vec Z_k^\pm (t,x)| \lesssim \exp\Biggl(-\frac{\sqrt{\lambda_0}}{\lambda_k(t)} \Biggl(\frac{1-|\ell_k|}
{(1-\ell_k^2)^\frac12} |x_1-\ell_k t -\by_{k,1}(t)| + |\overline x-\overline \by_k(t)| \Biggr)\Biggr) .
\end{equation*}
Thus, there exists $\mu_0>0$ (depending only on the parameters of the multi-soliton), such that assuming
\[
\left| \frac{\lambda_k(t)}{\lambda_k^\infty} - 1 \right| \leq \frac 12,
\]
it holds for any $\beta\in \N^5$,
\begin{equation}\label{eq:dZ}
|\partial^\beta\vec Z_k^\pm | \lesssim \zeta_k .
\end{equation}
Lastly, we recall from \cite[Claim 2]{MMwave1} that for any
$0 < r_2 \leq r_1$ such that $r_1+r_2>5$, and any $k\neq m$, for $t$ sufficiently large
\begin{equation}\label{eq:C2}
\begin{aligned}
&\mbox{if $r_1>5$ then} \quad \int \omega_k^{r_1}\omega_m^{r_2}\lesssim t^{-r_2},\\
&\mbox{if $r_1\leq 5$ then} \quad \int \omega_k^{r_1}\omega_m^{r_2}\lesssim t^{5-r_1-r_2}.
\end{aligned}
\end{equation}
Moreover, it is clear that, for any $r_1>0$, $r_2>0$, for any $k\neq m$, for $t$ sufficiently large
\begin{equation}\label{eq:C3}
\int \zeta_k^{r_1} \omega_m^{r_2} \lesssim t^{-r_2}.
\end{equation}

\subsection{Definition of a cut-off function} We introduce here a cut-off function used in the
definition of  the energy functional (see \cite[Section 4]{MMwave1}).
Fix
\begin{equation}\label{eq:lb}
\overline \ell = \frac 12 \Bigl( 1+ \max_k (|\ell_k|) \Bigr).
\end{equation}
For 
\[
0<\sigma< \frac 1{100} \min ( 1-\overline \ell, |\ell_k-\ell_m| \, ;\,  k,m\in \{1,\ldots, K\}, k\neq m)>0,
\]
small enough to be fixed, we set
\begin{align*}
\hbox{for $k=1,\dots, K-1$},\quad &\ell_k^{+}=\ell_k+\sigma(\ell_{k+1}-\ell_k),\\
\hbox{for $k=2,\dots, K$},\quad &\ell_k^{-}=\ell_k-\sigma(\ell_{k}-\ell_{k-1}),
\end{align*}
and for $t>0$,
\[
\Omega(t) = ( (\ell_1^+ t,\ell_{2}^- t)\cup\ldots \cup (\ell_{K-1}^+ t,\ell_{K}^- t)) \times \R^4,\quad
\Omega^C(t) = \R^5\setminus \Omega(t).
\]
To simplify notation, we also denote $\ell_1^-=-\infty$ and $\ell_K^+=+\infty$.
Then, we consider the continuous function $\chi(t,x)=\chi(t,x_1)$ defined as follows, for all $t>0$,
\begin{equation}\label{defchiK}
\left\{\begin{aligned}
& \hbox{$\chi(t,x) = \ell_k $ for $x_1\in (\ell_k^- t, \ell_k^+ t)$, $k\in \{1,\ldots,K\}$,}\\
& \chi(t,x) = \frac{x_1}{(1-2\sigma)t} - \frac {\sigma}{1-2 \sigma} (\ell_{k+1}+\ell_k) 
\hbox{ for $x_1 \in [\ell_k^+ t,\ell_{k+1}^-t ]$, $k\in \{1,\ldots,K-1\}$}.
\end{aligned}
\right.
\end{equation}
In particular,
\begin{equation}\label{derchi}\left\{\begin{aligned}
& \partial_t \chi(t,x) =0,\quad \nabla \chi(t,x)=0, \quad \hbox{on $\Omega^C(t)$},\\
 & \pun \chi(t,x)= \frac{1}{(1-2\sigma)t} \quad \hbox{for $x\in \Omega(t)$},\\
 	& \partial_{t} \chi(t,x)= -\frac 1t \frac{x_1}{(1-2\sigma)t} \quad \hbox{for $x\in \Omega(t)$}.
\end{aligned} \right.
\end{equation}

\section{Strong multi-solitons}\label{S:3}

The main objective of this section is to prove a first rigidity result related
to multi-solitons, see Proposition \ref{pr:2} below.
To do this, we closely follow the notation and the techniques introduced in \cite{MMwave2}.
(Note that the techniques used in \cite{MMwave1} are not sufficient here).

\subsection{Non homogeneous linearized equation}

Let $\ell\in (-1,1)$.
Let the functions $F$ and $G$ be defined by
\begin{equation}\label{eq:FG}
\begin{aligned}
&F = W^{\frac 43} + \kappa_\ell \Lambda W , \quad 
G = (1-\ell^2)^{-\frac 12} \kappa_\ell \ell\partial_1 \Lambda W ,\\
&\kappa_{\ell} = - (1-\ell^2) \frac{(W^{\frac 43},\Lambda W)}{\|\Lambda W\|_{L^2}^2}>0 .
\end{aligned}
\end{equation}
Set
\[
\zeta_\ell(t,x) = \left(\frac{x_1-\ell t}{\sqrt{1-\ell^2}},\overline x\right)
\]
and 
\[
f_\ell (t,x) = t^{-3} F ( \zeta_\ell(t,x)),\quad
g_\ell (t,x) = t^{-2} G ( \zeta_\ell(t,x) ),\quad w_\ell(t,x)=W(\zeta_\ell(t,x)).
\]
We recall the properties of the function $v_\ell(t,x)$ constructed in \cite[Lemma 3.1]{MMwave2}.
\begin{lemma}\label{le:as}
There exists a smooth function $v_\ell:(t,x)\in [1,+\infty)\times \R^5\mapsto v_\ell(t,x)$ such that, for all $0<\delta<1$, 
 for all $m\geq 0$,  $t\geq 1$, $x\in \R^5$,
\begin{equation}\label{eq:Av}
\begin{aligned}
|A_\ell^m v_\ell(t,x)| \lesssim_\delta (t+\langle \zeta_\ell\rangle)^{-1} t^{-(1+m)} \langle \zeta_\ell\rangle^{-2+\delta} ,\\
|A_\ell^m\nabla v_\ell(t,x)| \lesssim_\delta (t+\langle \zeta_\ell\rangle)^{-1}t^{-(1+m)} \langle \zeta_\ell\rangle^{-3+\delta} ,\\
|A_\ell^m\nabla\nabla v_\ell(t,x)| \lesssim_\delta (t+\langle \zeta_\ell\rangle)^{-1}t^{-(1+m)} \langle \zeta_\ell\rangle^{-4+\delta},
\end{aligned}
\end{equation}
and
\begin{equation}\label{eq:AE}\begin{aligned}
|A_\ell^m \cE_{\ell}(t,x)|&\lesssim_\delta t^{-(4+m)+\delta} \langle \zeta_\ell\rangle^{-3},\\
|A_\ell^m\nabla \cE_{\ell}(t,x)|&\lesssim_\delta t^{-(4+m)+\delta} \langle \zeta_\ell \rangle^{-4},
\end{aligned}\end{equation}
where
\begin{equation}\label{eq:vl}
\cE_\ell=
\partial_t^2 v_\ell - \Delta v_\ell - \frac 73 w_\ell^{\frac 43} v_\ell 
- f_\ell - g_\ell.
\end{equation}
\end{lemma} 
\begin{remark}\label{rk:31}
Lemma \ref{le:as} means that asymptotically in large time, the function $v_\ell$ 
is an approximate solution of the linear non homogeneous problem $\cE_\ell=0$
at the order of $t^{-4+\delta}$, for any $\delta>0$.
Setting
\[
z_\ell=\partial_t v_\ell + \frac{\kappa_\ell}{2t^2}\Lambda W(\zeta_\ell),
\]
we check by direct computation that the pair $(v_\ell, z_\ell)$ satisfies the following system
\begin{equation}\label{eq:vz}
\left\{\begin{aligned}
\partial_t v_\ell & = z_\ell - \frac{\kappa_\ell}{2t^2}\Lambda W(\zeta_\ell) , \\
\partial_t z_\ell & = \Delta v_\ell + \frac 73 W^\frac43(\zeta_\ell) v_\ell
+ \frac{1}{t^3} W^\frac43(\zeta_\ell)
+ \frac{\ell \kappa_\ell}{2 t^2\sqrt{1-\ell^2}} \partial_1 \Lambda W(\zeta_\ell) + \mathcal{E}_\ell.
\end{aligned}\right.
\end{equation}
The pair $(v_\ell,z_\ell)$ will allow us to keep track of the nonlinear interactions between the solitons.
Indeed, the source term $t^{-3} W^\frac43(\zeta_\ell)$ is related to soliton interactions
since a soliton at a distance of order $t$ has a size of order $t^{-3}$
in a neighborhood of the soliton $W(\zeta_\ell)$; see Lemma \ref{le:in} below.
The additional scaling terms in the definitions \eqref{eq:FG} of $F$ and $G$, which involve the constant $\kappa_\ell$,
are related to a spectral issue of the operator $L$ (the function $W^\frac43$ is not orthogonal to the kernel of $L$).
In the system \eqref{eq:vz}, the terms of order $t^{-2}$ in the right-hand side are designed 
to solving the spectral issue while suitably modifying the scaling law; see the definitions of ${\rm Mod}_\bW$ and ${\rm Mod}_\bX$ in Lemma~\ref{le:43}.
This observation justifies the definitions of the functions $f_\ell$, $g_\ell$ and $z_\ell$.
\end{remark}

\begin{remark}\label{rk:pr}
Let us define
\[
\widetilde\cE_\ell
=-\Delta_{\ell} v_\ell - f'(w_\ell) v_\ell - f_\ell - g_\ell .
\]
Using 
\[
\widetilde\cE_\ell=\cE_\ell - (\partial_t^2- \ell^2 \pun^2) v_\ell, \quad
(\partial_t^2- \ell^2 \pun^2)= A_\ell^2 - 2 \ell \pun A_\ell,
\]
and estimates \eqref{eq:Av}-\eqref{eq:AE}, we see that (taking $0<\delta<\frac12$)
\begin{equation}\label{eq:pl}
| \widetilde\cE_\ell | \lesssim t^{-3} \langle \zeta_\ell \rangle^{-\frac 52}.
\end{equation}
In particular, setting $V_{\ell}(t,x)=v_{\ell}(t,x+\ell e_1 t)$, we have 
from \eqref{eq:pl}
\begin{equation}\label{eq:Vl}
\left| L_\ell V_\ell - t^{-2}  \kappa_\ell \ell \partial_1 \Lambda W_{\ell}\right|
\lesssim t^{-3} \langle x\rangle^{-\frac52}.
\end{equation}
This estimate will be useful later. However, while \eqref{eq:Vl} may suggest to use separation of the variables
$(t,x)$ to prove Lemma \ref{le:as}, solving first $L_\ell V_\ell = t^{-2}  \kappa_\ell \ell \partial_1 \Lambda W_{\ell}$
with $V_\ell=t^{-2} U_\ell$ where $L_\ell U_\ell =  \kappa_\ell \ell \partial_1 \Lambda W_{\ell}$,
this is not a correct strategy to solve $\cE_\ell = 0$ (indeed, $U_\ell$ would have a low spatial decay).
We refer to the detailed proof of Lemma 3.1 in \cite{MMwave2}.
\end{remark}

\subsection{Approximate multi-soliton}\label{S:2.2}
Let $K\geq 2$. For all $k\in \{1,\ldots,K\}$, let $\lambda_k^\infty>0,$ ${\by}_k^\infty\in \R^5$, $\epsilon_k = \pm1$ and $\ell_k\in (-1,1)$
such that $\ell_k\neq \ell_m$ for any $k\neq m$.
We consider $\cC^1$ functions $\lambda_k$, ${\by}_k$ defined on a certain interval of time $I$ and satisfying on $I$
\begin{equation}\label{eq:bs1}
|\lambda_k-\lambda_k^\infty| + |\by_k-\by_k^\infty|\ll 1.
\end{equation}
In the next lemma, we recall from \cite{MMwave2} the computation of the main nonlinear interactions terms
between the $K$ solitons (at order $t^{-3}$).
\begin{lemma}[{\cite[Lemma 4.1]{MMwave2}}]\label{le:in}
For $m\neq k$, set
\[
\bs_{k,m} = \frac{\ell_k-\ell_m}{\sqrt{1-|\ell_m|^2}}
\]
and
\[
c_{k} = \frac 73(15)^{\frac 32} \sum_{m\neq k}{\epsilon_m(\lambda_m^\infty)^{\frac 32}}{|\bs_{k,m}|^{-3}}.
\]
Then,
\begin{equation}\label{eq:an}
f\Bigl(\sum_k W_k\Bigr)- \sum_{k} f(W_k) =
t^{-3}\sum_{k} c_k |W_k|^{\frac 43}+\bR_{\Sigma},
\end{equation}
where, for all $t\in I$, $x\in \R^5$,
\[
 |\bR_{\Sigma}| \lesssim t^{-4} \omega^3,\quad
 |\nabla \bR_\Sigma|\lesssim t^{-4} \omega^4.
\]
\end{lemma}
To construct a refined approximate multi-soliton which cancels the main interaction terms $t^{-3} c_k |W_k|^{\frac 43}$ that appear in \eqref{eq:an}, we define correction terms
to be added to the sum of solitons.
Inspired by Remark \ref{rk:31} and \cite[Section 4.2]{MMwave2}, we set
\begin{align*}
v_k(t,x) & = \frac 1{\lambda_k^3} v_{\ell_k}\left( \frac t{\lambda_k} , \frac {x-\by_k}{\lambda_k}\right),\\
z_k(t,x) & = \frac 1{\lambda_k^4} (\partial_tv_{\ell_k})\left( \frac t{\lambda_k} , \frac {x-\by_k}{\lambda_k}\right)
+ \frac{\kappa_{\ell_k}\epsilon_k}{2 \lambda_k^{\frac 12}t^2}\Lambda_k W_k(t,x)
\end{align*}
where the function $v_{\ell}$ is defined in Lemma \ref{le:as}
and $\kappa_\ell$ is defined in \eqref{eq:FG}.
Define
\begin{equation}\label{eq:ak}
a_k = - \frac {c_k\epsilon_k\kappa_{\ell_k}}2,\quad 
\vec v_k = \begin{pmatrix} v_k 
\\[.2cm]z_k \end{pmatrix}
\end{equation}
and
\begin{equation}\label{eq:WW}
\vec\bW =  \begin{pmatrix} \bW
\\[.2cm] \bX \end{pmatrix}
= \begin{pmatrix} \sum_k W_k + \sum_k c_k v_k
\\[.2cm] -\sum_k \ell_k \pun W_k + c_k z_k\end{pmatrix}=
\sum_k \bigl( \vec W_k + c_k\vec v_k \bigr).
\end{equation}
Note that from \eqref{eq:Av} the function $v_k$ satisfies
\begin{equation}\label{eq:vk}
\|v_k\|_{\dot H^1}\lesssim t^{-2},\quad
\|\partial_t v_k + \ell_k \partial_1 v_k\|_{L^\frac{10}3} \lesssim t^{-3}.
\end{equation}
The next lemma states that the function $\vec\bW$ defined above eliminates the 
first order of the soliton interactions of size $t^{-3}$, so that it is an approximate
solution to the multi-soliton problem at order $t^{-4^-}$.

\begin{lemma}[{\cite[Lemma 4.3]{MMwave2}}]\label{le:43}
Assume~\eqref{eq:bs1}.
Then, the function $\vec\bW$ satisfies on $I\times \R^5$ 
\begin{equation}\label{eq:WX}
\left\{\begin{aligned}
\partial_t\bW & = \bX - {\rm Mod}_{\bW} - \bR_{\bW},\\
\partial_t	\bX & 
	= \Delta \bW +|\bW|^{\frac 43} \bW - {\rm Mod}_{\bX}- \bR_{\bX},
\end{aligned}\right.
\end{equation}
where
\begin{align*}
{\rm Mod}_{\bW} &= \sum_k \Biggl( \frac{\dot\lambda_k}{\lambda_k} - \frac{a_k}{\lambda_k^{\frac 12}t^2} \Biggr) \Lambda_k W_k
+ \sum_k \dot {\by}_k \cdot \nabla W_k,\\
{\rm Mod}_{\bX} &= 
- \sum_k \Biggl( \frac{\dot\lambda_k}{\lambda_k} - \frac{a_k}{\lambda_k^{\frac 12}t^2} \Biggr) (\ell_k \cdot \nabla) \Lambda_k W_k 
- \sum_k ({\dot {\by}}_k \cdot \nabla) ( \ell_k \cdot \nabla) W_k,
\end{align*}
and, for all $0<\delta\leq \alpha/4$, it holds on $I\times\R^5$,
\begin{equation}\label{eq:Rr}
\begin{aligned}
\|\nabla \bR_\bW\|_{\dot H^1}+\|\bR_\bW\|_{\dot H^1} &\lesssim_\delta t^{-2} \sum_k\left(|\dot{\lambda}_k|+|\dot{\by}_k|\right),\\
\|\bR_\bX\|_{H^1} &\lesssim_\delta t^{-4+\delta}+t^{-2} \sum_k\left(|\dot{\lambda}_k|+|\dot{\by}_k|\right).
\end{aligned}
\end{equation}
Moreover, for all $0<\delta< 1$, it holds on $I\times\R^5$,
\begin{equation}\label{eq:bW}\begin{aligned}
&|\bW|
\lesssim_\delta \omega^{3}+t^{-1}\omega^{3-\delta} ,\quad
|\bX|\lesssim_\delta \omega^{4}+t^{-2}\omega^{3-\delta} ,\\
&|\nabla\bW - \sum_k \nabla W_k|
+|\bX+\sum_k \ell_k\cdot \nabla W_k|
\lesssim_\delta t^{-2} \omega^{3-\delta}.
\end{aligned}\end{equation}
\end{lemma}
The proof of this result is a consequence of Lemma \ref{le:as}, \eqref{eq:vz} and \eqref{eq:an}.

\begin{remark}\label{rk:fo}
The definitions of ${\rm Mod}_\bW$ and ${\rm Mod}_\bX$ say formally that at order $t^{-3}$,
\[
\frac{\dot\lambda_k}{\lambda_k} - \frac{a_k}{\lambda_k^{\frac 12}t^2}\approx 0,\qquad  \dot {\by}_k \approx 0.
\]
This will imply a modification of the scaling behavior $\lambda_k$ at order $t^{-1}$
with respect to the constant $\lambda_k^\infty$. See \eqref{eq:c3}.
\end{remark}

\subsection{Modulation around an approximate multi-soliton}
We define the counterpart of the pair $(v_k,z_k)$ with time-independent modulation.
Let
\begin{align*}
v_k^\infty(t,x) & = \frac 1{(\lambda_k^\infty)^3} v_{\ell_k}\left( \frac t{\lambda_k^\infty} , \frac {x-\by_k^\infty}{\lambda_k^\infty}\right),\\
z_k^\infty(t,x) & = \frac 1{(\lambda_k^\infty)^4} (\partial_tv_{\ell_k})\left( \frac t{\lambda_k^\infty} , \frac {x-\by_k^\infty}{\lambda_k^\infty}\right)
+ \frac{\kappa_{\ell_k}\epsilon_k}{2 (\lambda_k^\infty)^{\frac 12}t^2}\Lambda_k^\infty W_k^\infty(t,x)
\end{align*}
and 
\[
\vec v_k^{\, \infty} = \begin{pmatrix} v_k^\infty 
\\[.2cm]z_k^\infty \end{pmatrix}.
\]
We state a property of the refined approximate solution $\vec \bW$, which is a variant
of standard modulation arguments.

\begin{lemma}[{\cite[Lemma 4.4]{MMwave2}}]\label{le:dc} 
There exist $T_0\gg 1$ and $0<\omega_0\ll 1$ such that if $u(t)$ is a solution of~\eqref{wave} which satisfies for 
some time interval $I\subset[T_0,+\infty)$,
\begin{equation}\label{hyp:4}
\sup_{t\in I}
\biggl\|\vec u - \sum_{k} \left( \vec W_k^\infty + c_k\vec v_k^{\, \infty} \right) \biggr\|_{\dot H^1\times L^2}< \omega_0,
\end{equation}
then there exist $\cC^1$ functions $\lambda_k>0$, $\by_k$ on $I$ such that, 
the function $\vec \varepsilon(t)$ being defined by
\[
\vec \varepsilon=\begin{pmatrix}\varepsilon \\ \eta \end{pmatrix},\quad 
\vec u = \begin{pmatrix} u \\ \partial_t u \end{pmatrix} =
\vec \bW + \vec \varepsilon, 
\]
the following hold on $I$.
\begin{enumerate}[label=\emph{(\roman*)}]
\item{Orthogonality and smallness.} For $j=1,\ldots,5$, 
\begin{align}\label{eq:or}
&(\varepsilon,\Lambda_k W_k)_{\dot H_{\ell_k}^1}=
(\varepsilon,\partial_j W_k)_{\dot H_{\ell_k}^1}= 0,
\\
&|\lambda_k -\lambda_k^{\infty}|
+|\by_k -\by_k^\infty|
+\|\vec \varepsilon\, \|_{\dot H^1\times L^2}
\lesssim \biggl\|\vec u -\sum_{k} \left( \vec W_k^\infty + c_k\vec v_k^{\, \infty} \right) \biggr\|_{\dot H^1\times L^2}.
\label{bounds}
\end{align}
\item{Equation of $\vec \varepsilon$.} 
\begin{equation}\label{eq:ee}
\left\{\begin{aligned}
\partial_t \varepsilon & = \eta + {\rm Mod}_{\bW}+\bR_{\bW},\\
\partial_t \eta, & 
= \Delta \varepsilon +\left| \bW + \varepsilon\right|^{\frac 43} (\bW + \varepsilon)
- |\bW|^{\frac 43}\bW + {\rm Mod}_{\bX}+ \bR_{\bX} .
\end{aligned}\right.
\end{equation}
\item{Parameter estimates.}
\begin{equation}
\label{eq:ly}
\biggl|\frac {\dot \lambda_k}{\lambda_k} - \frac{a_k}{\lambda_k^{\frac 12}t^2}\biggr|
+|\dot {\by}_k|
\lesssim \|\vec\varepsilon\,\|_{\dot H^1\times L^2} +t^{-4}.
\end{equation}
\item{Unstable directions.} Let
\begin{equation}\label{eq:zm}
z_k^{\pm} = \big(\vec \varepsilon ,\vec {\tilde\theta}_k \vec Z_{\ell_k}^{\pm}\big)_{L^2},\quad
\mu_k = \frac{\sqrt{\lambda_0}}{\lambda_k}(1- \ell_k^2)^{\frac 12}.
\end{equation}
Then, for any $0<\delta<1$, 
\begin{equation}
\label{eq:zk}
\left| \frac d{dt} z_k^{\pm} \mp \mu_k z_k^{\pm} \right|
\lesssim_\delta \| \vec\varepsilon\, \|_{\dot H^1\times L^2}^2
+ t^{-1} \| \vec\varepsilon\,\|_{\dot H^1\times L^2} + t^{-4+\delta}.
\end{equation}
\end{enumerate}
\end{lemma}
\begin{remark}
The equation \eqref{eq:ee} of $\vec \varepsilon$, coupled with the orthogonality conditions
\eqref{eq:or} imply estimates \eqref{eq:ly} and \eqref{eq:zk} by projection and direct computations
using Lemma \ref{le:22}.
\end{remark}

\subsection{Notion of strong multi-soliton}

Now, we introduce a notion of \emph{strong multi-soliton}, following the properties that have been proved on the multi-solitons constructed in \cite{MMwave1,MMwave2} and also having in mind what we need on the sequel to 
obtain a classification result.
Let (see \eqref{eq:ak})
\begin{equation}\label{eq:dk}
d_k = -\frac{a_k}{(\lambda_k^\infty)^\frac12}=\frac {c_k \kappa_{\ell_k} \epsilon_k}{2(\lambda_k^\infty)^\frac12}.
\end{equation}

\begin{definition}\label{SMS}
Let $K\geq 2$. For $k\in \{1,\ldots,K\}$, let $\lambda_k^\infty>0$, $\by_k^\infty\in \R^5$, 
$\epsilon_k=\pm 1$ and $\ell_k\in (-1,1)$ with $\ell_k\neq \ell_{m}$ for $k\neq m$.
A solution $(\MS,\partial_t \MS)\in \cC([T,+\infty),\dot H^1\times L^2)$ of \eqref{wave} 
for a certain $T>0$, is called
a \emph{strong multi-soliton} if there exist 
 $\cC^1$ functions $\lambda_k:[T,+\infty)\to(0,+\infty)$,
$\by_k:[T,+\infty)\to\R^5$ such that
\begin{enumerate}[label=\emph{(\roman*)}]
\item Parameter estimates. For any $\delta>0$,
\begin{equation}\label{eq:c3}
\begin{aligned}
\left|\lambda_k(t)-\lambda_k^\infty \left(1+\frac{d_k}{t}\right)\right|
+|\by_k(t)-\by_k^\infty|
& \lesssim_\delta t^{-2+\delta},\\
\left|\frac{\dot\lambda_k }{\lambda_k}(t)+\frac{d_k}{t^2}\right|+|\dot\by_k(t)|
&\lesssim_\delta t^{-3+\delta} .
\end{aligned}
\end{equation}
\item Refined estimates.  
\begin{equation}\label{eq:c2}
\left\|\nabla_{t,x} \left(\MS - \sum_k W_k \right) (t)\right\|_{L^2} \lesssim t^{-2}.
\end{equation}
\item Regularity. It holds
$\nabla_{t,x}\MS\in\cC([T,+\infty),H^1)$ and 
\begin{equation}\label{eq:c4}
\left\|\nabla_{t,x} \left(\MS - \sum_k W_k \right) (t)\right\|_{H^1} \lesssim t^{-\frac32}.
\end{equation}
\end{enumerate}
\end{definition}
\begin{remark}
Of course \eqref{eq:c3} and \eqref{eq:c2} imply that, for all $t$ large,
\begin{equation}\label{eq:c1}
\left\|\nabla_{t,x} \left(\MS - \sum_k W_k^\infty \right) (t)\right\|_{L^2} \lesssim t^{-1},
\end{equation}
but \eqref{eq:c2} is a more precise estimate, saying that modulating the solitons, 
the solution $S$ is close to a sum of solitons at order $t^{-2}$.
This means that the first manifestation of the nonlinear interaction appears at order $t^{-1}$ through the scaling parameter $\lambda_k$. Moreover, to go beyong the order $t^{-2}$, one needs the approximate solution $\WW$.
\end{remark}

\begin{proposition}[\cite{MMwave2}]\label{pr:vp}
Under the assumptions of Theorem \ref{th:0}, there exists
a strong multi-soliton $\MS$ of \eqref{wave} on some time interval $[T,+\infty)$.
Moreover, it satisfies
\begin{equation}\label{eq:oS}
\left\| ( S ,\partial_t S) (t) - \vec \bW(t)\right\|_\cE \lesssim t^{-3+\delta},
\end{equation}
where $\vec \bW$ is defined in \eqref{eq:WW} for the functions $(\lambda_k,\by_k)_k$ corresponding to $\MS$ in
Definition \ref{SMS}.
\end{proposition}
\begin{proof}
We check that the multi-soliton constructed in~\cite{MMwave2} satisfies the properties
(i)--(iii) of Definition \ref{SMS}.
Properties (i)-(ii) are consequences of \cite[Proposition 5.1]{MMwave2},
the definitions of $\mathbf{W}$, $\mathbf{X}$,
$\vec a_k$ in \cite[(4.8)--(4.10)]{MMwave2}, the estimates in \cite[Lemma 3.1]{MMwave2}
and the refined estimates on $\dot \lambda_k$ and $\dot \by_k$ from \cite[(4.21)]{MMwave2}.
The value of the parameter  $d_k$ in \eqref{eq:dk} is deduced
from \cite{MMwave2}, see (4.21), the definition of $a_k$ page 1306,
the definition of $\kappa_\ell$ in (3.1) and the definition of $c_k$ in Lemma 4.1.
Formally, it can also be derived from Remark \ref{rk:fo}.

Property (iii) follows from \cite[Section 5.4]{MMwave2}.
We point out that the estimates given in \cite[(4.27)]{MMwave1} would not be enough for our needs.
\end{proof}

\subsection{Pre-classification of multi-solitons}\label{s:3.5}

Here, we prove a preliminary classification result for multi-solitons.
Let $S$ be a strong multi-soliton as given by Proposition \ref{pr:vp}.

\begin{proposition}\label{pr:2}
Under the assumptions of Theorem \ref{th:0}, if 
$u(t)$ is a solution of \eqref{wave} on some time interval $[T,+\infty)$ such that for some $\delta'>0$,
for all $t>T$,
\begin{equation}\label{eq:py}
\left\|\nabla_{t,x} \left(u - \sum_k W_k^\infty \right) (t)\right\|_{L^2} \lesssim t^{-\frac12-\delta'},
\end{equation}
then for all $\delta>0$,
\begin{equation}\label{eq:cy}
\left\|\nabla_{t,x} \left(u - S \right) (t)\right\|_{L^2} \lesssim_\delta t^{-2+\delta}.
\end{equation}
Moreover, $u$ is a strong multi-soliton.
\end{proposition}
\begin{remark}
Proposition \ref{pr:2} says that any mild multi-soliton (in the sense that \eqref{eq:py} holds),
is a strong multi-soliton.
This result is not proved in \cite{MMwave2}, but its proof will be deduced from the estimates established in \cite{MMwave2}
combined with a simple technical lemma (Lemma \ref{le:tc} below).
\end{remark}
\begin{proof}
We use the framework of the construction of the multi-solitons in \cite{MMwave2},
except that we cannot use the sharp bootstrap estimate \cite[(5.8)]{MMwave2} in the present proof.
Indeed, we deal with a multi-soliton in the sense of \eqref{eq:py}, not with a construction proof
where suitable estimates are proved by bootstraping estimates.
Thus, we will rewrite parts of the proof in \cite[\S 5]{MMwave2}, using the same notation 
and similar estimates, but not using any bootstrap  estimates other than
\begin{equation}\label{eq:bs}
|\lambda_k -\lambda_k^{\infty}|
+|\by_k -\by_k^\infty|
+\|\vec \varepsilon \,\|_{\dot H^1\times L^2} \lesssim t^{-\frac 12 - \delta'},
\end{equation}
for some $\delta'>0$, which is deduced from \eqref{eq:vk}, \eqref{bounds} and \eqref{eq:cy}.
Moreover, by \eqref{eq:ly} and the definitions of ${\rm Mod}_\WW$ and ${\rm Mod}_\XX$, we have
\begin{equation}\label{Idem}
\begin{aligned}
|\partial^\beta{{\rm Mod}}_{\WW}(t)|&\lesssim 
\left(\left\|\vec \varepsilon(t)\right\|_{\dot H^1\times L^2} +t^{-4}\right)
\sum_k \omega_k^{3+|\beta|},\\
|\partial^\beta{{\rm Mod}}_{\XX}(t)|&\lesssim 
\left(\left\|\vec \varepsilon(t)\right\|_{\dot H^1\times L^2} +t^{-4}\right)
\sum_k \omega_k^{4+|\beta|}.
\end{aligned}
\end{equation}
The above estimates will replace \cite[(5.10)]{MMwave2}.
Define the energy functional
\[
\cH = \int |\nabla\varepsilon|^2+|\eta|^2 + 2 \chi (\partial_1 \varepsilon ) \eta
- 2 (F (\WW +\varepsilon )-F (\WW )-f (\WW )\varepsilon )
\]
and the quantities
\begin{align*}
\mathcal{N}_\Omega & = \int_\Omega \left(|\nabla \varepsilon|^2+\eta^2+2 \chi (\pun \varepsilon) \eta \right),\\
\mathcal{N}_{\Omega^C} & = \int_{\Omega^C} \left(|\nabla \varepsilon|^2+\eta^2\right).
\end{align*}
Note that, since $|\chi|<\overline \ell$,
\begin{equation}\label{nintee} 
\begin{aligned}
\Nint & = \overline \ell \int_{\Omega} \left|\frac {\chi}{\overline \ell} \pun \varepsilon + \eta\right|^2 
+\int_{\Omega} |\overline \nabla \varepsilon|^2
+ \int_{\Omega} \left( 1- \frac {\chi^2}{\overline \ell} \right) (\pun \varepsilon)^2 
+ (1-\overline \ell) \int \eta^2\\
 & \geq \overline \ell \int_{\Omega} \left|\frac {\chi}{\overline \ell} \pun \varepsilon + \eta\right|^2 
 + (1-\overline \ell) \int_{\Omega}\left( |\nabla \varepsilon|^2 + \eta^2\right).
\end{aligned}
\end{equation}
We replace \cite[Lemma 5.4]{MMwave2} concerning the time derivative of the functional $\cH$ by the following result.

\begin{lemma}\label{mainprop}
Assume \eqref{eq:bs}.
For all $0<\delta<\frac 1{10}$,
there exists $C,\mu>0$ such that, for all $t$ sufficiently large, the following hold.
\begin{enumerate}[label=\emph{(\roman*)}]
\item\emph{Bound.}
\begin{equation}\label{boun}
|\cH(t)| \leq C {\|\vec \varepsilon(t)\|_{\dot H^1\times L^2}^2}.
\end{equation}
\item\emph{Coercivity.}
\begin{equation}\label{coer}
\cH(t) \geq \mu\|\vec \varepsilon(t)\|_{\dot H^1\times L^2}^2 - C \sum_k \left( (z_k^-)^2+(z_k^+)^2 \right).
\end{equation}
\item\emph{Time variation.}
\begin{equation}\label{time}
- \frac d{dt}  \cH  \leq \frac{1+\delta}{t} \cH
+ C t^{-7+2\delta} + \frac Ct \sum_k \left( (z_k^-)^2+(z_k^+)^2 \right).
\end{equation}
\end{enumerate}
\end{lemma}
\begin{remark}
Recall that the functional $\cH$ is a key ingredient in the construction proof of \cite{MMwave1}.
\end{remark}
\begin{proof}
The estimate \eqref{boun} is the same as the one in \cite{MMwave2}.
We summarize how to obtain \eqref{coer} and \eqref{time} from the computations in the proof
of Lemma 5.4 in \cite{MMwave2}.

For \eqref{coer}, we only mention that the proof follows from the following analogue of \cite[(5.17)]{MMwave2}
\begin{equation}\label{lF} 
\cH(t) \geq 
 {\cN_{\Omega}}(t) + \mu {\cN_{\Omega^C}}(t)   -C t^{-4\gamma} \left\|\vec \varepsilon \,\right\|_{\dot H^1\times L^2}^2 
- C \left\|\vec \varepsilon\,\right\|_{\dot H^1\times L^2}^3
- C \sum_k \left( (z_k^-)^2+(z_k^+)^2 \right),
\end{equation}
($\gamma$ is any sufficiently small positive number),
the only difference being that after using \eqref{eq:2.30},
we do not insert the bootstrap estimate for the directions $(z_k^\pm)_k$
as was done for the term $\bbf_1$.
From \eqref{nintee}, we see that ${\cN_{\Omega}}(t) + \mu {\cN_{\Omega^C}}(t) \gtrsim \|\vec \varepsilon(t)\|_\cE^2$, and thus using also \eqref{eq:bs}, \eqref{lF} implies \eqref{coer}
for $t$ sufficiently large.

To prove \eqref{time}, we follow closely the proof in \cite{MMwave2}, by using only the weak bootstrap estimate \eqref{eq:bs}.
First, we differentiate $\cH$ and we decompose
\begin{align*}
\frac d{dt} \cH& = \int \partial_t \left(|\nabla\varepsilon|^2+|\eta|^2- 2 (F (\WW +\varepsilon )-F (\WW )-f (\WW )\varepsilon ) \right)\\&\quad
 + 2 \int \chi \partial_t( \pun \varepsilon \eta ) 
+ 2 \int (\partial_t \chi) \pun \varepsilon \eta = {\bg_1} + {\bg_2} + {\bg_3},
\end{align*}
\emph{Estimate on $\bg_1$.}
Differentiating and integrating by parts, we have
\begin{align*}
{\bg_1} & = 2 \int (\partial_t \varepsilon) \left( -\Delta \varepsilon - \left( f(\WW + \varepsilon) - f(\WW) \right)\right) +2 \int (\partial_t \eta) \eta\\
& \quad -2 \int (\partial_t \WW) \left( f(\WW + \varepsilon) - f(\WW)-f'(\WW) \varepsilon \right) .
\end{align*}
Using~\eqref{eq:WX} and~\eqref{eq:ee}
\begin{align*}
{\bg_1} & = 2 \int \left(-\Delta \varepsilon -f'(\WW)\varepsilon\right){\rm Mod}_\WW + 2 \int \eta{\rm Mod}_\XX \\
&\quad +2\int \left(-\Delta \varepsilon -f'(\WW)\varepsilon\right){\bR}_{\WW} + \int \eta {\bR}_\XX \\
& \quad -2 \int \XX \left( f(\WW + \varepsilon) - f(\WW)-f'(\WW) \varepsilon\right) 
 ={\bg_{1,1}} + {\bg_{1,2}}+ {\bg_{1,3}}.
\end{align*}
By integration by parts,
\[
{\bg_{1,1}} = 2 \int \varepsilon \left(-\Delta {\rm Mod}_\WW  -f'(\WW) {\rm Mod}_\WW \right)+ 2 \int \eta{\rm Mod}_\XX.
\]
By integration by parts, the Cauchy Schwarz inequality and then \eqref{eq:Rr}, \eqref{eq:ly},
\begin{align*}
|{\bg_{1,2}}| &\lesssim (\|\bR_\WW\|_{\dot H^1}+\|\bR_\XX\|_{L^2}) \|\vec \varepsilon\|_{\cE}\\
&\lesssim \left(t^{-4+\delta} + t^{-2} \sum_k \left(|\dot \lambda_k| + |\dot \by_k|\right) \right) \|\vec \varepsilon\|_{\cE}\\
&\lesssim \left(t^{-4+\delta} + t^{-2} \|\vec \varepsilon\|_{\cE} \right) \|\vec \varepsilon\|_{\cE}.
\end{align*}
By $\sum_k\|z_k\|_{L^{\frac {10}3}}\lesssim t^{-2}$ (see \eqref{eq:vk}),
\begin{align*}
& \left|{\bg_{1,3}} - 2 \int \left( \sum_k \ell_k \partial_1 W_k\right) \left( f(\WW + \varepsilon) - f(\WW)-f'(\WW) \varepsilon\right)\right|\\
& \quad \lesssim 
\int \left(\sum_k |z_k|\right) \left( |\varepsilon|^{\frac 73} +\varepsilon^2 |\WW|^{\frac 13}\right)
\lesssim \|z_k\|_{L^{\frac {10}3}} \|\varepsilon\|_{L^{\frac {10}3}}^2 \lesssim t^{-2} \|\vec \varepsilon\|_{\cE}^2.
\end{align*}
Thus,
\begin{align*} 
{\bg_1} & = 2 \int \varepsilon \left( -\Delta {\rm Mod}_{\WW} - f'(\WW) {\rm Mod}_{\WW} \right) +2 \int \eta{\rm Mod}_{\XX} \\
& \quad + 2 \int \left( \sum_k \ell_k \pun W_k\right) \left( f(\WW + \varepsilon) - f(\WW)-f'(\WW) \varepsilon\right)\\
& \quad + O(t^{-4+\delta} \|\vec \varepsilon\|_\cE)
+ O(t^{-2}  \|\vec \varepsilon\|_\cE^2).
\end{align*}

\emph{Estimate on $\bg_2$.} Using \eqref{eq:ee},
\begin{align*}
{\bg_2} & = 2 \int (\chi \partial_1 \partial_t \varepsilon) \eta+ 2 \int (\chi \partial_1 \varepsilon) \partial_t \eta \\
& = 2 \int (\chi \partial_1 \eta) \eta + 2 \int (\chi \partial_1 \varepsilon) \left[ \Delta \varepsilon + \left( f(\WW + \varepsilon) - f(\WW) \right) \right]
 \\& \quad+2 \int (\chi \partial_1 {\rm Mod}_{\WW}) \eta + 2 \int (\chi \partial_1 \varepsilon) {\rm Mod}_{\XX}
 +2 \int (\chi \partial_1 {\bR}_{\WW}) \eta + 2 \int (\chi \partial_1 \varepsilon) {\bR}_{\XX}
\end{align*}
By integration by parts and~\eqref{derchi}
\begin{equation*}
 2 \int (\chi \partial_1 \eta) \eta + 2 \int (\chi \partial_1 \varepsilon) \Delta \varepsilon 
   =- \frac 1{(1-2\sigma) t} \int_{\Omega}
 \left(\eta^2 + (\partial_1 \varepsilon)^2
 - |\overline \nabla \varepsilon|^2 \right).
\end{equation*}
Then,
\begin{align*}
 \int (\chi \partial_1 \varepsilon) \left( f(\WW + \varepsilon) - f(\WW) \varepsilon\right) 
 & = \int \chi \partial_1 ( F(\WW + \varepsilon) - F(\WW)-f(\WW) \varepsilon )\\
&\quad - \int \chi (\partial_1 \WW) \left( f(\WW + \varepsilon) - f(\WW)-f'(\WW) \varepsilon\right).
\end{align*}
Integrating by parts and using~\eqref{derchi},
\[ -\int \chi \partial_1 ( F(\WW + \varepsilon) - F(\WW)-f(\WW) \varepsilon )\\
 = \frac 1{(1-2\sigma) t} \int_{\Omega} \left( F(\WW + \varepsilon) - F(\WW)-f(\WW) \varepsilon\right).
\]
Thus, by~\eqref{eq:BS} and $\|\WW\|_{L^{\frac{10}3}(\Omega)}\lesssim t^{-\frac 32}$,
we obtain
\begin{align*}
\left| \int \chi \partial_1 ( F(\WW + \varepsilon) - F(\WW)-f(\WW) \varepsilon )\right|
&\lesssim t^{-1} \left(\|\varepsilon\|_{L^{\frac{10}3}}^{\frac{10}3}+\|\varepsilon\|_{L^{\frac{10}3}}^2
\|\WW\|_{L^{\frac{10}3}(\Omega)}^{\frac43}\right)\\
&\lesssim t^{-1} \|\vec \varepsilon\|_{\cE}^{\frac{10}3}+ t^{-3}\|\vec\varepsilon\|_{\cE}^2.
\end{align*}
By~\eqref{eq:Av} and~\eqref{eq:BS}
\begin{align*}
&\left| \int \chi \left(\partial_1 \WW -\partial_1 \sum W_k\right) \left( f(\WW + \varepsilon) - f(\WW)-f'(\WW) \varepsilon\right)\right|
\\ &\quad= \left| \int \chi \partial_1 \left(\sum c_k v_k\right) \left( f(\WW + \varepsilon) - f(\WW)-f'(\WW) \varepsilon\right)\right|
\\&\quad
\lesssim \left( \sum_k \|\partial_1 v_k\|_{L^\frac{10}3} \right) \|\varepsilon\|_{L^{\frac {10}3}}^2
\lesssim t^{-2}\|\vec\varepsilon\|_{\cE}^2.
\end{align*}
Lastly, integrating by parts, 
\[
 2 \int (\chi \partial_1 \varepsilon) {\rm Mod}_{\XX} 
 = -2 \int (\chi \varepsilon) \partial_1 {\rm Mod}_{\XX} + O\left(t^{-7}\right),
\]
since by~\eqref{eq:BS},~\eqref{derchi} and~\eqref{Idem}
\begin{align*}
\left|\int (\partial_1\chi)\varepsilon {\rm Mod}_{\XX}\right|
&\lesssim t^{-1}\left(t^{-4}+\|\vec \varepsilon\|_\cE \right) \int_{\Omega} |\varepsilon| \left( \sum_k |W_k|^{\frac 43}\right)\\
&\lesssim t^{-1}\left(t^{-4}+\|\vec \varepsilon\|_\cE \right) \|\varepsilon\|_{L^{\frac {10}3}} \left(\sum_k \|W_k\|_{L^{\frac{40}{21}}(\Omega)} \right)^{\frac 43}\\
&\lesssim t^{-\frac32}\left(t^{-4}+\|\vec\varepsilon\|_\cE\right)\|\vec\varepsilon\|_\cE.
\end{align*}
We complete the estimate of $\bg_2$ by observing that \eqref{eq:Rr} and \eqref{eq:ly} yield
\begin{align*}
\left| \int (\chi \partial_1 {\bR}_{\WW}) \eta \right|+\left|\int (\chi \partial_1 \varepsilon) {\bR}_{\XX}\right|
&\lesssim \|\eta\|_{L^2} \|\partial_1 {\bR}_{\WW}\|_{L^2} + \|\partial_1 \varepsilon\|_{L^2} \|{\bR}_{\XX}\|_{L^2}\\
&\lesssim \left( t^{-4+\delta}+t^{-2} \|\vec\varepsilon\|_\cE\right) \|\vec\varepsilon\|_\cE.
\end{align*}
Therefore,
\begin{align*}
{\bg_2} & = 
 -\frac 1{(1-2\sigma) t} \int_{\Omega}
 \left(\eta^2 + (\partial_1 \varepsilon)^2
 - |\overline \nabla \varepsilon|^2 \right)\nonumber\\
& \quad - 2\int \chi\left( \sum_k \partial_1 W_k\right)\left( f(\WW + \varepsilon) - f(\WW)-f'(\WW) \varepsilon\right)\nonumber\\
& \quad +2 \int (\chi \partial_1 {\rm Mod}_{\WW}) \eta - 2 \int \varepsilon \chi \partial_1 {\rm Mod}_{\XX}\\
&\quad + O(t^{-4+\delta} \|\vec \varepsilon\|_\cE )
+ O( t^{-\frac32}  \|\vec \varepsilon\|_\cE^2)
+ O\Big(t^{-1}  \|\vec \varepsilon\|_\cE^\frac{10}3\Big).
\end{align*}
\emph{Estimate on $\bg_3$.} By \eqref{derchi}, we observe that
\begin{align*}
{\bg_3} = -\frac 2{(1-2\sigma)t} \int_{\Omega} \frac{x_1}{t} (\partial_1 \varepsilon) \eta.
\end{align*}
Gathering the above estimates, we obtain
\begin{align*}
\frac d{dt} \cH
& = {\bh_1}+{\bh_2}+{\bh_3}+{\bh_4}\\
&\quad + O(t^{-4+\delta} \|\vec \varepsilon\|_\cE) + 
O(t^{-\frac32}  \|\vec \varepsilon\|_\cE^2 )
+ O\Big(t^{-1}  \|\vec \varepsilon\|_\cE^\frac{10}3\Big),
\end{align*}
where
\begin{align*}
{\bh_1} &= -\frac 1{(1-2\sigma) t} \int_{\Omega}
 \left(\eta^2 + (\partial_1 \varepsilon)^2 +2\frac{x_1}{t} (\partial_1 \varepsilon) \eta
 - |\overline \nabla \varepsilon|^2 \right), \\ 
{\bh_2}& = 2 \int \left(\sum_k \left(\ell_k - \chi\right) \partial_1 W_k\right)
\left( f(\WW + \varepsilon) - f(\WW)-f'(\WW) \varepsilon\right),\\
{\bh_3} &= 2 \int \eta\left({\rm Mod}_{\XX} + \chi \partial_1 {\rm Mod}_{\WW}\right),\\
{\bh_4} &= 2 \int \varepsilon \left( -\Delta {\rm Mod}_{\WW} - \chi \partial_1 {\rm Mod}_{\XX}- f'(\WW) {\rm Mod}_{\WW} \right).
\end{align*}
Compared to the proof in \cite[Lemma 5.4]{MMwave2}, we treat the term $\bh_1$ exactly in the same way,
proving that
\[
-\bh_1 \leq \frac{1+C\sigma}{(1-2\sigma)t}\cN_\Omega.
\]
Moreover, we check that
\[
|\bh_2| \lesssim 
\Big(\sum_k \left\|\left(\ell_k - \chi\right) \partial_1 W_k\right\|_{L^\frac{10}3}\Big)
\left(\|\bW\|_{L^\frac{10}3}^\frac13\|\varepsilon\|_{L^\frac{10}3}^2+\|\varepsilon\|_{L^\frac{10}3}^\frac73\right)
\lesssim t^{-\frac52} \|\vec\varepsilon\|_\cE^2.
\]
For the other two terms, we check using \eqref{Idem} and computations in \cite{MMwave2} that
\begin{multline*}
\| {\rm Mod}_{\XX} + \chi \pun  {\rm Mod}_{\WW} \|_{L^2} 
+\|-\Delta {\rm Mod}_{\WW} - \chi \partial_1  {\rm Mod}_{\XX} 
-f'(\bW)  {\rm Mod}_{\WW} \|_{L^\frac{10}7}\\
\lesssim t^{-\frac32} \left(t^{-4}+\|\vec \varepsilon\|_{\cE}\right).
\end{multline*}
Thus,
\begin{equation*}
|\bh_3|+ |\bh_4| \lesssim t^{-\frac32} \left(t^{-4}+\|\vec\varepsilon\|_\cE \right) \|\vec\varepsilon\|_\cE.
\end{equation*}

In conclusion, using also \eqref{lF}, we obtain
\[
- \frac d{dt}  \cH  \leq \frac{1+C\sigma}{t} \cH  
+ C t^{-4+\delta} \|\vec\varepsilon\|_\cE  + \frac Ct \sum_k \left( (z_k^-)^2+(z_k^+)^2 \right).
\]
Using \eqref{coer} again, this implies \eqref{time} taking $\sigma$ small enough depending of $\delta$.
\end{proof}
\begin{remark}
In \eqref{time}, the term $C t^{-7+2\delta}$ is due to the error of size $t^{-4+\delta}$ of the approximate solution which behaves as a source term for the equation of $\vec \varepsilon$.
\end{remark}

To finish the proof of Proposition \ref{pr:2}, we will use the following technical lemma.
For the lemma, we only need to assume that the functions $\mu_k$ defined in \eqref{eq:zm}
are continuous and satisfy
for any $t\geq T$,
\begin{equation}\label{eq:mu}
\mu_k(t) \geq \mu_0>0,
\end{equation}
which here is ensured by \eqref{eq:bs}.
Moreover, we set
\[
\beta_k(t) = \int_T^t \mu_k.
\]

\begin{lemma}\label{le:tc}
Let $K\geq 1$.
Assume that there exist constants $\delta\in(0,\frac 14)$, $C_0,C_1>0$, $T\geq 1$ and differentiable functions 
\begin{align*}
N : [T,+\infty) & \to [0,+\infty),\\
\cH : [T,+\infty) & \to \R,\\
z_k^\pm:[T,+\infty) & \to \R,
\end{align*}
such that, for all $t\geq T$,
\begin{align}
N^2 & \leq C_0 \cH + C_1 \sum_k \left( (z_k^-)^2+(z_k^+)^2 \right), \label{eq:m2} \\
|\cH| &\leq C_0 N^2,\label{eq:m3}\\
\left| \frac d{dt} z_k^{-} + \mu_k z_k^{-} \right| &\leq C_0 \left(  N^2 + t^{-1} N+ t^{-4+\delta} \right),\label{eq:m4} \\
\left| \frac d{dt} z_k^{+} - \mu_k z_k^{+} \right| &\leq C_0 \left(  N^2 + t^{-1} N+ t^{-4+\delta} \right),\label{eq:m5} \\
- \frac d{dt} \cH & \leq \frac{1+\delta}{t} \cH
+ C_0 t^{-7+2\delta} + \frac{C_0}t \sum_k \left( (z_k^-)^2+(z_k^+)^2 \right). \label{eq:m6}
\end{align}
Moreover, for a constant $\delta'>\delta$, assume that
\begin{equation}\label{eq:m1}
N^2 + \sum_k \left( (z_k^+)^2 + (z_k^-)^2 \right) \leq C_0 t^{-1- \delta'}.
\end{equation}
Then, there exists $C>0$ such that for all $t\geq T$,
\begin{equation}\label{eq:m7}
N^2\leq C t^{-6+2\delta},\quad \sum_k \left( (z_k^+)^2 + (z_k^-)^2 \right) \leq C  t^{-8+2\delta}.
\end{equation}
\end{lemma}
\begin{remark}
The reason why this lemma is needed here and not in the construction proof in \cite{MMwave2} is that
we cannot rely on a strong bootstrap assumption, but, in view of obtaining the strongest possible classification result, 
only on the weak assumption \eqref{eq:bs}.
Note that \eqref{eq:bs} seems the weakest possible assumption in view of solving a differential equation of the form
$-\cH' \leq 1^+ \frac{\cH}{t} + O(t^{-K})$ from $t\to+\infty$, where $\cH\approx \|\vec \varepsilon\|_\cE^2$.
\end{remark}
\begin{remark}
Estimate \eqref{eq:m7} is optimal in view of the source terms of size $t^{-7+2\delta}$ in \eqref{eq:m6}
and of size $t^{-4+\delta}$ in \eqref{eq:m4} and \eqref{eq:m5}.
Recall that these source terms are due to the use of an approximate multi-soliton at the order $t^{-4^-}$
for the decomposition of the solution.
\end{remark}
\begin{proof}[Proof of Lemma \ref{le:tc}]
We prove the first estimate of \eqref{eq:m7} by induction, using the following induction assumption,
for $n\geq 0$,
\begin{equation}\label{eq:m8}
N^2\leq C_n t^{-6+2\delta} + C_n t^{-1-n } .
\end{equation}
We only need to prove \eqref{eq:m8} for $n=5$ since it implies $N^2\leq C_5 t^{-6+2\delta}$.

Observe first that \eqref{eq:m1} implies \eqref{eq:m8} for $n=0$.
Assuming \eqref{eq:m8} for some $0\leq n \leq 4$, we prove it for $n+1$.
By $N\lesssim t^{-\frac 12}$, we have 
\[
N^2 + t^{-1} N+ t^{-4+\delta} \lesssim t^{-\frac 12} N + t^{-4+\delta}
\lesssim t^{-\frac 72 +\delta} + t^{-1-\frac n2} +  t^{-4+\delta}
\lesssim t^{-1-\frac n2}.
\]
Thus, by \eqref{eq:m4}, one has
\[
\left| \left( e^{\beta_k} z_k^- \right)' \right| \lesssim t^{-1-\frac n2} e^{\beta_k} ,
\]
and integrating on $[T,t]$,  using $|z_k^-(T)|\lesssim 1$ from \eqref{eq:m1},
\[
|z_k^-(t)| \lesssim e^{-\beta_k}  + t^{-1-\frac n2}.
\]
Similarly, by \eqref{eq:m5}, one has
\[
\left| \left( e^{-\beta_k}  z_k^+ \right)' \right| \lesssim t^{-1-\frac n2} e^{-\beta_k} ,
\]
and integrating on $[t,+\infty)$, using $|z_k^+(t)|\lesssim 1$ from \eqref{eq:m1},
\[
|z_k^+(t)| \lesssim t^{-1-\frac n2} \lesssim t^{-1-\frac n2}.
\]
Inserting the above estimates for $|z_k^\pm|$ in \eqref{eq:m6}, we obtain
\[
-\left( t^{1+\delta} \cH \right)' \lesssim t^{-6+3\delta} + t^{-2-n+\delta}.
\]
Integrating on $[t,+\infty)$, using $\lim_{t\to +\infty} t^{1+\delta}\cH(t)=0$ 
(since $0<\delta<\delta'$),
we obtain
\[
\cH(t) \lesssim t^{-6+2\delta} + t^{-2-n}.
\]
From \eqref{eq:m2} and the estimates on $z_k^\pm$, this implies
\[
N^2(t) \lesssim t^{-6+2\delta} + t^{-2-n}.
\]
This finishes the induction argument and thus proves \eqref{eq:m7} for $N$.

Using this, we obtain
\[
N^2 + t^{-1} N+ t^{-4+\delta} \lesssim t^{-4+\delta}.
\]
Thus, by \eqref{eq:m4} and \eqref{eq:m5}, one has
\begin{equation*}
\left| \left( e^{\beta_k} z_k^- \right)' \right| \lesssim t^{-4+\delta} e^{\beta_k} ,\quad
\left| \left( e^{-\beta_k}  z_k^+ \right)' \right| \lesssim t^{-4+\delta} e^{-\beta_k} ,
\end{equation*}
and proceeding as before, we obtain
\eqref{eq:m7} for $z_k^\pm$.
\end{proof}

Now, setting
$N=\|\vec \varepsilon\, \|_\cE$,
we check that the hypothesis of Lemma \ref{le:tc} are satisfied.
Indeed, \eqref{coer} implies \eqref{eq:m2}, \eqref{boun} implies \eqref{eq:m3},
\eqref{eq:zk} implies \eqref{eq:m4} and \eqref{eq:m5}.
Moreover, \eqref{time} and \eqref{eq:bs} respectively imply
\eqref{eq:m6} and \eqref{eq:m1}.
Applying Lemma \ref{le:tc}, we obtain
\[
\|\vec \varepsilon\|_\cE\lesssim t^{-3+\delta},
\]
strongly improving the estimate on $\vec\varepsilon$.
Moreover, by integrating \eqref{eq:ly}, we obtain the estimate \eqref{eq:c3} on the parameters $\lambda_k$ and $\by_k$.

At this stage, we have the solution $\vec S$, defined in Proposition \ref{pr:vp} and satisfying \eqref{eq:oS} for
a corresponding approximate solution $\vec \WW[S]$, with certain parameters $(\lambda_k[S], \by_k[S])$ satisfying \eqref{eq:c3}.
Now, we have just proved that any solution $u$ of equation \eqref{wave} such that \eqref{eq:py} holds 
also satisfies
\begin{equation}\label{eq:o2}
\| \vec u - \vec \bW[u] \|_\cE \lesssim t^{-3+\delta},
\end{equation}
for an approximate solution $\vec \WW[u]$, with certain parameters $(\lambda_k[u], \by_k[u])$ also satisfying \eqref{eq:c3}.
The estimate
\begin{equation*}
\| \vec u -\vec S\|_\cE \lesssim
\| \vec u - \vec \WW[u]\|_\cE+\|  \vec \WW[u] - \vec \WW[S]\|_\cE+\| \vec S - \vec \WW[S]\|_\cE
\end{equation*}
with
\[
|\lambda_k[u]-\lambda_k[S]|+|\by_k[u]-\by[S]|\lesssim t^{-2+\delta},
\]
then implies $\| \vec u -\vec S\|_\cE \lesssim t^{-2+\delta}$, which is \eqref{eq:cy}.

Finally, the additional regularity \eqref{eq:c4} on the solution $\vec u$ is obtained as in \cite[Section 5.4]{MMwave2}.
\end{proof}

\begin{corollary}\label{co:in}
Under the assumptions of Theorem \ref{th:0}, if 
$u(t)$ is a solution of \eqref{wave} on some time interval $[T,+\infty)$ and such that for some $\delta>0$,
\begin{equation}\label{eq:pz}
\left\|\nabla_{t,x} \left(u - \sum_k W_k^\infty \right) (t)\right\|_{L^2} \lesssim t^{-\frac12-\delta}
\end{equation}
for any $t\geq T$ large, then there exists a constant $C>0$ such that for all $R>0$
large enough,
\begin{equation}\label{eq:RR}
\liminf_{t\to -\infty} \|\nabla u(t)\|_{L^2(|x|>|t|+R)} \geq C R^{-\frac 52}.
\end{equation}
\end{corollary}
\begin{proof}
Let $R>1$ large enough and $t_R=R^{\frac{11}{12}}$.
By \eqref{eq:o2}, one has
\begin{equation*}
\|(u,\partial_t u)(t_R)-\vec \WW[u](t_R)\|_\cE \lesssim t_R^{-6+4\delta} \lesssim R^{-\frac{77}{15}}.
\end{equation*}
Following \cite[\S 6]{MMwave2}, this is enough to prove \eqref{eq:RR},
using the properties of the approximate solution $\vec \bW[u]$ and the method of channels of energy.
\end{proof}

\section{Decomposition around a strong multi-soliton}\label{S:4}
We fix a strong multi-soliton $S$ and we follow the notation of Definition \ref{SMS},
in particular, the functions $\vec \bW = \vec\bW[S]$, $\lambda_k=\lambda_k[S]$ and $\by_k=\by_k[S]$ are now fixed from the choice of $S$.
We consider another multi-soliton $\vec u(t)$ of \eqref{wave} satisfying \eqref{eq:hy} and thus by Proposition \ref{pr:2}
satisfying \eqref{eq:cy}.

In this section, instead of decomposing $\vec u(t)$ around a suitably modulated
approximate solution $\vec\bW[u]$ as in the previous section, we decompose it around the
strong multi-soliton $S$, using additive modulation (first and second order approximations of a modulation
by scaling and translation). This decomposition will be useful both for the proof of the
multi-existence result Theorem \ref{th:1} (see Section \ref{S:5}) and the one of the classification result Theorem \ref{th:2}
(see Section~\ref{S:6}).

\subsection{Ordering the solitons}\label{S:4.1}
As in the previous section, we set
\begin{equation}\label{eq:MK}
\mu_k^\infty= \frac{\sqrt{\lambda_0}}{\lambda_k^\infty}(1-\ell_k^2)^{\frac 12},\quad
\mu_k = \frac{\sqrt{\lambda_0}}{\lambda_k}(1-\ell_k^2)^{\frac 12},\quad
\beta_k = \int_T^t \mu_k.
\end{equation}
For technical reasons, we need to reorder the solitons $\{W_k\}_k$ according to the asymptotic behavior of the functions $\{\mu_k\}_k$.
While the velocities $\{\ell_k\}_k$ are two-by-two different, 
which is crucial to split the soliton trajectories,
we cannot exclude the possibility of having $\mu_k^\infty=\mu_{m}^\infty$ for some $k\neq m$.
Because of the
asymptotic of $\lambda_k(t)$ in \eqref{eq:c3}, 
in case of equality of $\mu_k^\infty$, we also need to order the $d_k$.
Recall that by \eqref{eq:c3} (with $0<\delta<\frac14$)
\[
\lambda_k(t)=\lambda_k^\infty \left(1+\frac{d_k}{t}\right) +O(t^{-\frac74})
\]
and thus, 
\begin{equation}\label{eq:mm}
\mu_k(t)
=\mu_k^\infty
\left(1-\frac{d_k}{t}\right) +O(t^{-\frac74}).
\end{equation}
By integration, we obtain
\[
\beta_k(t) =\mu_k^\infty \left(t-d_k \ln t\right) + C_k +O(t^{-\frac34})
\]
where $C_k=C_k(T)$ is an integration constant.
In particular, 
\[
e^{-\beta_k(t)} = D_k t^{\nu_k} e^{-\mu_k^\infty t} (1+O(t^{-\frac34}))
\]
where
\[
\nu_k = \mu_k^\infty d_k,\quad D_k=D_k (T) = e^{- C_k(T)}.
\]
Now, without losing generality, we change the order of the solitons to ensure that for all $k=1,\ldots,K-1$,
\begin{equation}\label{eq:er}
\mbox{$\mu_k^\infty \leq \mu_{k+1}^\infty$ and if $\mu_k^\infty=\mu_{k+1}^\infty$ then 
$\nu_k \geq \nu_{k+1} $.}
\end{equation}
Let
\begin{equation}\label{eq:n0}
\nu_0 = \frac18 \min \left\{ 1, \nu_k-\nu_{k+1} , \mbox{for $k$ such that
$\mu_k^\infty=\mu_{k+1}^\infty$ and $\nu_k > \nu_{k+1}$}\right\}>0.
\end{equation}
Moreover, for a given $m\in\{1,\ldots,K\}$, we set
\begin{align*}
I_m & = \left\{ 1\leq k\leq K : \mu_k^\infty< \mu_m^\infty \mbox{ or } \mu_k^\infty=\mu_m^\infty \mbox{ and then }
\nu_k\geq \nu_m \right\}, \\
J_m & = \left\{1,\ldots,K\right\}\setminus I_m.
\end{align*}
Note that if for some $k\neq m$, $\mu_k^\infty = \mu_{m}^\infty$ and $\nu_k = \nu_{m}$, then
the way the solitons can be ordered is not unique and $I_m$ is not necessarily equal to $\{1,\ldots,m\}$.
By the definition of $\mu_0$, we see that the following property holds true
\begin{equation}\label{eq:Jm}
\mbox{if $k\in J_m$ then either $\mu_k^\infty>\mu_m^\infty$ or $\mu_k^\infty=\mu_m^\infty$
and then $\nu_k\leq \nu_m - 8 \nu_0$.}
\end{equation}

\subsection{Decomposition of the solution}\label{s:4.1} 

Let $m\in\{1,\ldots,K\}$ and $A\in\R$. 
Here, $A$ is a free parameter related to the exponential stable direction for
the soliton $W_m$. It will be used only in the proof of the construction result (Section \ref{S:5}).
Set
\[
\vec E_\mm= \begin{pmatrix} E_\mm \\ F_\mm\end{pmatrix}
= e^{-\beta_m} \vec {\theta}_\mm J \vec Z_{\ell_\mm}^+ 
= e^{-\beta_m} J \vec Z_m^+,
\]
where $\vec Z_m^+$ is defined in \eqref{defWk}.
Define
\[
\vec a =\begin{pmatrix} a \\ b \end{pmatrix}= \vec u - \vec \MS - A\vec E_\mm - \vec Q,
\]
where
\[
\vec Q =\begin{pmatrix} Q  \\[.1cm]R  \end{pmatrix} = \vec Q^{\rm (I)} + \vec Q^{\rm (II)},
\]
with
\begin{align*}
\vec Q^{\rm (I)} & =\begin{pmatrix} Q^{\rm (I)} \\[.1cm]R^{\rm (I)} \end{pmatrix} = \sum_k \vec Q_k^{\rm (I)},\\
\vec Q_k^{\rm (I)} & =\begin{pmatrix} Q_k^{\rm (I)} \\[.2cm] R_k^{\rm (I)}\end{pmatrix}= \alpha_k^\Lambda \vec Z_{k}^\Lambda + \sum_{j} \alpha_k^{(j)} \vec Z_{k}^{(j)}
\end{align*}
and
\begin{align*}
\vec Q^{\rm (II)} &  =\begin{pmatrix} Q^{\rm (II)} \\[.1cm] R^{\rm (II)} \end{pmatrix}
= \sum_k \vec Q_k^{\rm (II)} ,\\ 
\vec Q_k^{(\rm II)} & =\begin{pmatrix} Q_k^{\rm (II)} \\[.2cm] R_k^{\rm (II)}\end{pmatrix}
= \frac 12(\alpha_k^\Lambda)^2 \vec Z_{k}^{\Lambda\Lambda}
+  \sum_{j} \alpha_k^\Lambda \alpha_k^{(j)} \vec Z_{k}^{(j)\Lambda}
+ \frac 12\sum_{j,j'} \alpha_k^{(j')}\alpha_k^{(j)} \vec Z_{k}^{(j')(j)}.
\end{align*}
The parameters $\{\alpha_k^\Lambda\}_k$ and $\{\alpha_k^{(j)}\}_{j,k}$ are uniquely chosen so that
\begin{equation}\label{eq:OR}
 \left( a , \Lambda_k W_k \right)_{\dot H^1_{\ell_k}}=0,\quad
 \left( a , \partial_j W_k\right)_{\dot H^1_{\ell_k}}=0.
\end{equation}
The existence and uniqueness of such parameters $\{\alpha_k^\Lambda\}_k$ and $\{\alpha_k^{(j)}\}_{j,k}$
for large $t$ follows from straightforward arguments (see also the proof of \eqref{eq:ke} in LEmma \ref{le:bL}).
We also define
\begin{equation}\label{defq}
q = \sum_k |\alpha_k^\Lambda|+\sum_{j,k} |\alpha_k^{(j)}|.
\end{equation}
In order to derive the equation of $\vec a$, we start by computing the equations of $\vec E_m$ and $\vec Q$.

First, setting
\[
\vec M_E=
\frac{\dot\lambda_m}{\lambda_m} \vec \Lambda_m \vec E_m
+\dot\by_m \cdot \nabla \vec E_m = \begin{pmatrix}M_E\\M_F\end{pmatrix},
\]
we observe that by \eqref{eq:tg} and \eqref{eq:MK},
\begin{align*}
\partial_t \vec E_m 
&= - \mu_m \vec E_m
- \frac{\ell_m}{\lambda_m} e^{-\beta_m}\vec { \theta}_m \pun J \vec Z_{\ell_m}^- - \vec M_E \\
&= \frac1{\lambda_m} e^{-\beta_m}\vec { \theta}_m \left(-\sqrt{\lambda_0}(1-\ell_m^2)^\frac12 J \vec Z_{\ell_m}^+ 
- \ell_m \pun J\vec Z_{\ell_m}^+ \right)-\vec M_E.
\end{align*}
Thus, using \eqref{Zpm},
\begin{equation*}
\partial_t \vec E_m 
= \frac1{\lambda_m} e^{-\beta_m}\vec { \theta}_m \left( J H_{\ell_m} J \vec Z_{\ell_m}^+ 
- \ell_m \pun J\vec Z_{\ell_m}^+ \right)-\vec M_E.
\end{equation*}
But we also remark from the definition of $H_{\ell_m}$ in \eqref{Hbeta} that
\[
J H_{\ell_m} = \begin{pmatrix}
0 & {\rm Id} \\ \Delta + f'(W_{\ell_m}) & 0
\end{pmatrix} + \begin{pmatrix}
\ell_{m} \pun & 0 \\ 0 & \ell_{m} \pun 
\end{pmatrix},
\]
and thus
\begin{align*}
\partial_t \vec E_m 
& = \frac1{\lambda_m} e^{-\beta_m}\vec { \theta}_m 
\begin{pmatrix}
0 & {\rm Id} \\ \Delta + f'(W_{\ell_m}) & 0
\end{pmatrix} J\vec Z_{\ell_m}^+
-\vec M_E .
\end{align*}
Using the identity
\[
\frac 1{\lambda_m} \vec \theta_m \begin{pmatrix}
0 & {\rm Id} \\ \Delta + f'(W_{\ell_m}) & 0
\end{pmatrix}
= \begin{pmatrix}
0 & {\rm Id} \\ \Delta + f'(W_m) & 0
\end{pmatrix} \vec \theta_m,
\]
we obtain
\[
\partial_t \vec E_m 
 =\begin{pmatrix}
0 & {\rm Id} \\ \Delta + f'(W_m) & 0
\end{pmatrix} \vec E_m
-\vec M_E .
\]
This means that
\begin{equation}\label{ee:EF}
\begin{cases}
\partial_t E_\mm = F_\mm - M_E \\
\partial_t F_\mm = \Delta E_\mm + f'(W_m) E_\mm - M_F.
\end{cases}
\end{equation}

Second, using \eqref{eq:tg}, we observe that
\begin{equation*}
\partial_t \vec Q^{\rm (I)} = - \sum_k \ell_k \pun \vec Q_k^{\rm (I)} + \vec P^{\rm (I)} - \vec M_Q^{\rm (I)},
\end{equation*}
where
\begin{align*}
\vec P^{\rm (I)} & = \sum_k \frac{d}{dt}\alpha_k^\Lambda\vec Z_k^\Lambda+\sum_{j,k}\frac{d}{dt}\alpha_k^{(j)}\vec Z_k^{(j)}
=\begin{pmatrix} P_1^{\rm (I)} \\ P_2^{\rm (I)} \end{pmatrix},\\
\vec M_Q^{\rm (I)} & = \sum_k\frac{\dot\lambda_k}{\lambda_k} \vec\Lambda_k \vec Q_k^{\rm (I)}
+\sum_{k} \dot\by_k\cdot \nabla \vec Q_k^{\rm (I)}
=\begin{pmatrix} M_Q^{\rm (I)}\\[4pt] M_R^{\rm (I)} \end{pmatrix}.
\end{align*}
For each $k$, using the expressions of $\vec Z_\ell^\Lambda$ and $\vec Z_\ell^{(j)}$ and then \eqref{LW}, we have
\begin{align*}
-\ell_k \pun Q_k^{\rm (I)} & =R_k^{\rm (I)},\\ 
- \ell_k\pun R_k^{\rm (I)} & = \Delta Q_k^{\rm (I)} + f'(W_k) Q_k^{\rm (I)},
\end{align*}
and so
\begin{equation*}
\partial_t \vec Q^{\rm (I)} = 
\sum_k \begin{pmatrix}
0 & 1 \\ \Delta + f'(W_k) & 0
\end{pmatrix} \vec Q_k^{\rm (I)} +\vec P^{\rm (I)} - \vec M_Q^{\rm (I)}.
\end{equation*}
This means that
\begin{equation}\label{ee:QR}
\begin{cases}
\partial_t Q^{\rm (I)} = R^{\rm (I)} + P_1^{\rm (I)} - M_Q^{\rm (I)} , \\
\partial_t R^{\rm (I)} = \Delta Q^{\rm (I)} + \sum_k f'(W_k) Q_k^{\rm (I)} + P_2^{\rm (I)} - M_R^{\rm (I)}.
\end{cases}
\end{equation}
Third, we compute
\[
\partial_t \vec Q^{\rm (II)} = - \sum_k \ell_k \partial_1 \vec Q_k^{\rm (II)} + \vec P^{\rm (II)} - \vec M_Q^{\rm (II)},
\]
where
\begin{align*}
\vec P^{\rm (II)} &
= \frac12\sum_k \frac{d}{dt}(\alpha_k^\Lambda)^2 \vec Z_{k}^{\Lambda\Lambda}
+ \sum_{k,j} \frac{d}{dt}(\alpha_k^\Lambda \alpha_k^{(j)}) \vec Z_{k}^{(j)\Lambda}
+ \frac12\sum_{k,j,j'} \frac{d}{dt}( \alpha_k^{(j')}\alpha_k^{(j)}) \vec Z_{k}^{(j)(j')}
=\begin{pmatrix} P_1^{\rm (II)} \\ P_2^{\rm (II)} \end{pmatrix},\\
\vec M_Q^{\rm (II)} & = \sum_k\frac{\dot\lambda_k}{\lambda_k} \vec\Lambda_k \vec Q_k^{\rm (II)}
+\sum_{k} \dot\by_k \cdot \nabla \vec Q_k^{\rm (II)} 
 =\begin{pmatrix} M_Q^{\rm (II)}\\[4pt] M_R^{\rm (II)} \end{pmatrix}.
\end{align*}
For each $k$, using the expressions of $\vec Z_\ell^{\Lambda\Lambda}$, $\vec Z_\ell^{(j)\Lambda}$, $\vec Z_\ell^{(j')(j)}$, and \eqref{L2}, we have
\begin{align*}
-\ell_k \pun Q_k^{\rm (II)} & =R_k^{\rm (II)},\\ 
- \ell_k\pun R_k^{\rm (II)} & = \Delta Q_k^{\rm (II)} + f'(W_k) Q_k^{\rm (II)}+ \frac12(\alpha_k^\Lambda)^2 (\Lambda_k W_k)^2 f''(W_k)
\\
& \quad+ \lambda_k \sum_j \alpha_k^\Lambda \alpha_k^{(j)} \Lambda_k W_k \partial_j W_k f''(W_k)  +\frac12\lambda_k^2 \sum_{j,j'} \alpha_k^{(j)}\alpha_k^{(j')}  \partial_{j'} W_k \partial_j W_k f''(W_k)\\
& = \Delta Q_k^{\rm (II)} + f'(W_k) Q_k^{\rm (II)} + \frac 12  f''(W_k) \big(Q_k^{\rm (I)}\big)^2 ,
\end{align*}
and so, summing in $k$,
\begin{equation}\label{ee:QR2}
\begin{cases}
\partial_t Q^{\rm (II)} = R^{\rm (II)} + P_1^{\rm (II)} - M_Q^{\rm (II)} , \\
\partial_t R^{\rm (II)} = \Delta Q^{\rm (II)} + \sum_k f'(W_k) Q_k^{\rm (II)} +\frac12\sum_k f''(W_k) \big(Q_k^{\rm (I)}\big)^2 
+ P_2^{\rm (II)} - M_R^{\rm (II)}.
\end{cases}
\end{equation}
Defining
\[
\vec P = \vec P^{\rm (I)} + \vec P^{\rm (II)}=\begin{pmatrix} P_1 \\ P_2 \end{pmatrix},
\quad \vec M_Q=\vec M_Q^{\rm (I)}+ \vec M_Q^{\rm (II)}=\begin{pmatrix} M_Q \\ M_R  \end{pmatrix},
\]
we obtain
\begin{equation}\label{ee:QRf}
\begin{cases}
\partial_t Q = R + P_1  - M_Q  , \\
\partial_t R = \Delta Q + \sum_k f'(W_k) Q_k + P_2 - M_R 
+\frac12\sum_k f''(W_k) \big(Q_k^{\rm (I)}\big)^2 .
\end{cases}
\end{equation}

Therefore, using that $\MS$ and $u$ are solutions of \eqref{wave},
and replacing 
\begin{equation}\label{eq:dS}
\vec u = \vec\MS+A\vec E_m + \vec Q + \vec a,
\end{equation}
we obtain the following system
\begin{equation}\label{eq:vv}
\begin{cases}
\partial_t a  = b - P_1 + A M_E + M_Q , \\
\partial_t b  = \Delta a + f(\MS+A E_\mm +Q+a) - f(\MS) - f'(W_\mm ) AE_\mm\\  \qquad\quad
-\tsum_k f'(W_k) Q_k - \frac12 \tsum_k f''(W_k) \big(Q_k^{\rm (I)}\big)^2 - P_2+ A M_F + M_R.
\end{cases}
\end{equation}
Recall that the function $\vec \MS$ and the corresponding parameters $\lambda_k=\lambda_k[S]$ and $\by_k=\by_k[S]$ were fixed,
so that terms in $\vec M_E$ and $\vec M_Q$ are not modulation terms related to the decomposition of $\vec u$. 
The modulation terms related to the solution $\vec u$ are contained in $\vec P$.
In this section \eqref{eq:vv} is both a system for the pair $(a,b)$ and
for the parameters $\{\alpha_k^\Lambda\}_k$, $\{\alpha_k^{(j)}\}_{j,k}$
through the orthogonality relations \eqref{eq:OR} (see Lemma \ref{le:bL}).

We will need another form of the second equation
\begin{equation}\label{eq:34}
\partial_t b = \Delta a + f'(\MS) a + \pp_1 a + \pp_2 + \pp_3 - P_2+ A M_F + M_R ,
\end{equation}
where
\begin{equation}\label{eq:pp}
\begin{aligned}
\pp_1 & = f'(\MS +AE_\mm+Q) - f'(\MS) , \\
\pp_2 & = f(\MS+AE_\mm +Q+a)
- f(\MS+AE_\mm +Q) - f'(\MS +AE_\mm +Q) a ,\\
\pp_3 & = f(\MS+AE_\mm +Q)
- f(\MS) - f'(W_m) AE_\mm - \tsum_k f'(W_k) Q_k \\ & \quad - \frac12 \tsum_k f''(W_k) \big(Q_k^{\rm (I)}\big)^2.
\end{aligned}
\end{equation}
For future reference, we note that
\begin{align}
|\pp_1| & \lesssim \big(|\MS|^{\frac 13}+ |A|^\frac13 |E_m|^\frac13+|Q|^\frac13\big)(|A| |E_m|+|Q|), \label{eq:o10}\\
|\pp_2| & \lesssim \big(|\MS|^\frac13 + |A|^\frac13 |E_m|^\frac13+|Q|^\frac13+|a|^\frac13\big) a^2.\label{eq:ov0}
\end{align}
Moreover, we rewrite
\begin{equation}\label{eq:p4}
\begin{aligned}
\pp_3 & = f(\MS+AE_\mm +Q)- f(\MS) - f'(\MS) (AE_\mm+Q)
-\tfrac 12 f''(\MS) (AE_\mm+Q)^2 \\
&\quad +( f'(\MS) - f'(W_m) ) AE_\mm + \tsum_k(f'(\MS) - f'(W_k))Q_k  \\
&\quad + \frac 12 f''(S) (A^2E_m^2+2AE_m Q) 
+ \frac 12 f''(S) \Big( Q^2 - \tsum_k  \big(Q_k^{\rm(I)}\big)^2 \Big) \\
&\quad + \frac 12\tsum_k ( f''(\MS) -f''(W_k)) \big(Q_k^{\rm(I)}\big)^2 .
\end{aligned}
\end{equation}
From the decomposition above, we deduce the following pointwise estimate on $\pp_3$
\begin{equation}\label{eq:o20}
\begin{aligned}
|\pp_3|
&\lesssim |A E_m|^\frac73 +|Q|^\frac73+ \big(|\MS|^\frac13+|W_m|^\frac13\big) |A| |E_m| |\MS-W_m| \\
&\quad + \tsum_k \big(|\MS|^\frac13+|W_k|^\frac13\big) |Q_k| |\MS-W_k| 
+ |S|^\frac13 |AE_m|(|AE_m|+|Q|)\\
&\quad +|S|^\frac13 \Bigl(|Q^{\rm(II)}|(|Q^{\rm(I)}|+|Q^{\rm(II)}|)+\sum_{l\neq k} |Q_k^{\rm (I)}||Q_{l}^{\rm(I)}|\Bigr)
+\sum_k \big(Q_k^{\rm(I)}\big)^2|S-W_k|^\frac13.
\end{aligned}
\end{equation}
\begin{remark}
We see in \eqref{eq:o20} that the term $\vec Q^{\rm (II)}$ (second order additive modulation terms) in the definition of $\vec Q$ 
allows to remove terms of order $q^2$
in the estimate of the quantity $p_3$. This will be essential to obtain sufficiently precise energy estimates in the classification proof.
\end{remark}
Defining
\[
\vec {\mathcal L} = \begin{pmatrix}0 & 1 \\ \Delta + f'(\MS) & 0\end{pmatrix}
\]
the system \eqref{eq:vv} of $\vec a = (a,b)^\trans$ takes the following condensed form
\begin{equation}\label{eq:vb}
\partial_t \vec a = \vec {\mathcal L} \vec a + \vec \pp - \vec P + A \vec M_E + \vec M_Q
\end{equation}
where
\[
\vec \pp = \begin{pmatrix} 0 \\ \pp_1 a + \pp_2 + \pp_3 \end{pmatrix}
=\vec \pp_1 a + \vec \pp_2 + \vec \pp_3 .
\]

In the next subsections, we provide estimates on the decomposition
$\big(\vec a, \{\alpha_k^\Lambda\}_k, \{ \alpha_k^{(j)}\}_{j,k}\big)$ of any solution $\vec u$,
assuming that $\MS$ is a fixed strong multi-soliton in the sense of Definition \ref{SMS} and under the following weak bootstrap assumptions for $t$ large,
\begin{equation}\label{BPweak}
q\ll 1,\quad \|\vec a\|_\cE\ll 1.
\end{equation}
We will keep track of the dependency on the parameter $A$ so that the estimates will also be relevant in the case where $A=0$.

\subsection{Equations of the kernel directions}
\begin{lemma}\label{le:bL}
Assuming \eqref{BPweak}, it holds
\begin{equation}\label{eq:ke}
\left| \frac{d}{dt} \alpha_k^\Lambda \right| + \left| \frac{d}{dt} \alpha_k^{(j)} \right|
\lesssim \|\vec a\|_{\mathcal E} + t^{-2} q + |A| t^{-2} e^{-\beta_m} .
\end{equation}
\end{lemma}
\begin{proof}
We start by differentiating the first orthogonality condition \eqref{eq:OR}
and then we use the first line of \eqref{eq:vv} and \eqref{eq:tG}
\begin{align*}
0 = \frac{d}{dt} \left( a, \Lambda_k W_k \right)_{{\dot H^1_{\ell_k}}}
& = \left(\partial_t a, \Lambda_k W_k \right)_{{\dot H^1_{\ell_k}}}
+\left( a, \partial_t \Lambda_k W_k \right)_{\dot H^1_{\ell_k}}\\
& = - \left( b, \Delta_{\ell_k}\Lambda_k W_k \right)_{L^2}
- \ell_k \left( a, \pun \Lambda_k W_k \right)_{\dot H^1_{\ell_k}}\\
&\quad - \frac{\dot \lambda_k}{\lambda_k} (a, \Lambda_k^2 W_k)_{\dot H^1_{\ell_k}} 
- \dot \by_k \cdot (a,\nabla \Lambda_k W_k)_{\dot H^1_{\ell_k}}\\
&\quad -(P_1,\Lambda_k W_k)_{\dot H^1_{\ell_k}}
+ A (M_E,\Lambda_k W_k)_{\dot H^1_{\ell_k}}+(M_Q,\Lambda_k W_k)_{\dot H^1_{\ell_k}}.
\end{align*}
First, it is clear that
\[
\left|\left( b, \Delta_\ell \Lambda_k W_k \right)_{L^2}\right|
+ \left| \left( a, \pun \Lambda_k W_k \right)_{\dot H^1_{\ell_k}}\right| \lesssim \|\vec a\|_\cE.
\]
Moreover, by \eqref{eq:c3},
\[
\left|\frac{\dot \lambda_k}{\lambda_k} (a, \Lambda^2 W_k)_{\dot H^1_{\ell_k}}\right|
+\left| \dot \by_k \cdot (a,\nabla\Lambda_k W_k)_{\dot H^1_{\ell_k}}\right|
\lesssim t^{-2} \|\vec a\|_\cE.
\]
Then, we deal with the term
\begin{align*}
(P_1,\Lambda_k W_k)_{\dot H^1_{\ell_k}} & = 
 \sum_l \frac{d}{dt}\alpha_l^\Lambda \left(\Lambda_l W_l,\Lambda_k W_k\right)_{\dot H^1_{\ell_k}}
+ \sum_{j,l}\frac{d}{dt}\alpha_l^{(j)} \lambda_l \left(\partial_j W_l,\Lambda_k W_k\right)_{\dot H^1_{\ell_k}}\\
& \quad  + \frac 12 \sum_l \frac{d}{dt} (\alpha_l^\Lambda)^2 \left(\Lambda_l^2 W_l,\Lambda_k W_k\right)_{\dot H^1_{\ell_k}}
+ \sum_{j,l} \lambda_l \frac{d}{dt} ( \alpha_l^\Lambda \alpha_l^{(j)} )
 \left(\partial_j\Lambda_l W_l,\Lambda_k W_k\right)_{\dot H^1_{\ell_k}}
\\
&\quad + \frac 12 \sum_{j',j,l} \lambda_l^2 \frac{d}{dt} (\alpha_l^{(j')}\alpha_l^{(j)})  \left(\partial_{j'}\partial_j W_l,\Lambda_k W_k\right)_{\dot H^1_{\ell_k}}
\end{align*}
For any $l\neq k$, by \eqref{eq:C2},
\[
\left|\left(\Lambda_l W_l,\Lambda_k W_k\right)_{\dot H^1_{\ell_k}}\right|\lesssim t^{-3},
\]
while for any $k$, by symmetry and \eqref{eq:C2},
\[
\left|\left(\partial_j W_l,\Lambda_k W_k\right)_{\dot H^1_{\ell_k}}\right|\lesssim t^{-3}.
\]
Therefore, by the definition of $q$ in \eqref{defq},
\[
\left| (P_1,\Lambda_k W_k)_{\dot H^1_{\ell_k}} -
\frac{d}{dt}\alpha_k^\Lambda \left(\Lambda_k W_k,\Lambda_k W_k\right)_{\dot H^1_{\ell_k}}\right|
\lesssim (q+t^{-3}) \left(\sum_{l} \left|\frac{d}{dt}\alpha_l^\Lambda\right|+  \sum_{j,k}\left|\frac{d}{dt}\alpha_k^{(j)}\right|\right).
\]
Last, again by \eqref{eq:c3},
\[
\left| A (M_E, \Lambda_k W_k )_{\dot H^1_{\ell_k}}\right|+\left|(M_Q,\Lambda_k W_k)_{\dot H^1_{\ell_k}}\right|
\lesssim t^{-2} ( |A| e^{-\beta_m}+q).
\]
In conclusion, we obtain the following estimate for any $k$,
\begin{align*}
\left|\frac{d}{dt}\alpha_k^\Lambda\right|
&\lesssim (q+t^{-3}) \left(\sum_{l} \left|\frac{d}{dt}\alpha_l^\Lambda\right|+  \sum_{j,k}\left|\frac{d}{dt}\alpha_k^{(j)}\right|\right)
+\|\vec a\|_\cE + t^{-2}(q+|A| e^{-\beta_m}).
\end{align*}

Proceeding similarly with the relation $\left( a,  \partial_j W_k \right)_{\dot H^1_{\ell_k}}=0$,
we obtain for any $k$ and any $j$,
\begin{align*}
\left|\frac{d}{dt}\alpha_k^{(j)}\right|
&\lesssim (q+t^{-3}) \left(\sum_{l} \left|\frac{d}{dt}\alpha_l^\Lambda\right|+  \sum_{j,k}\left|\frac{d}{dt}\alpha_k^{(j)}\right|\right)
+\|\vec a\|_\cE + t^{-2}(q+|A| e^{-\beta_m}).
\end{align*}
Combining all these estimates, for $t$ sufficiently large (using \eqref{BPweak}), we obtain \eqref{eq:ke}.
\end{proof}

\subsection{Equations of the exponential directions}
We set
\begin{equation}\label{eq:ap}
\alpha_k^\pm = \left(\vec a, \vec Z_k^\pm \right)_{L^2}
\end{equation}
and
\begin{equation}\label{eq:BB}
B_k  = \sqrt{(\alpha_k^-)^2 + (\alpha_k^+)^2},\quad B = \sqrt{\sum_k B_k^2}.
\end{equation}
\begin{lemma}\label{le:bp}
Assuming \eqref{BPweak}, it holds
\begin{equation}\label{eq:zo}
\left| \frac{d}{dt} \alpha_k^\pm \mp \mu_k \alpha_k^\pm \right| 
\lesssim \|\vec a\|_{\mathcal E}^2 + q^2 + t^{-2} (\|\vec a\|_{\mathcal E}+q) + |A| t^{-2} e^{-\beta_m} .
\end{equation}
\end{lemma}
\begin{proof}
Using \eqref{eq:th} and \eqref{eq:vb} we obtain
\begin{align*}
\frac{d}{dt} \left(\vec a, \vec Z_k^\pm \right)_{L^2}
&= \left(\partial_t \vec a, \vec Z_k^\pm  \right)_{L^2}
+\left(\vec a, \partial_t\vec Z_k^\pm\right)_{L^2}\\
&=
\left(\vec{\mathcal L}\vec a, \vec Z_k^\pm\right)_{L^2}
- \ell_k \left(\vec a, \pun \vec Z_k^\pm \right)_{L^2}
-\frac{\dot \lambda_k}{\lambda_k}\left(\vec a, \vec{\widetilde \Lambda} \vec Z_k^\pm\right)_{L^2}
 - \dot \by_k \cdot \left(\vec a, \nabla \vec Z_k^\pm\right)_{L^2}\\
&\quad + \left(\vec \pp,\vec Z_k^\pm \right)_{L^2}
-\left(\vec P , \vec Z_k^\pm \right)_{L^2}
+A \left(\vec M_E , \vec Z_k^\pm\right)_{L^2}
+\left(\vec M_Q , \vec Z_k^\pm \right)_{L^2}.
\end{align*}
Using now \eqref{Hbeta} and \eqref{Zpm}, we have
\begin{align*}
&\left(\vec{\mathcal L}\vec a, \vec Z_k^\pm\right)_{L^2}
- \ell_k \left(\vec a, \pun \vec Z_k^\pm \right)_{L^2}\\
&\qquad =\frac 1{\lambda_k}\left( \vec a, \vec{\tilde \theta}_k(-H_{\ell_k}J\vec Z_{\ell_k}^\pm\right)_{L^2}
+ \left(\begin{pmatrix}0\\(f'(\MS)-f'(W_k))a\end{pmatrix}, \vec Z_k^\pm \right)_{L^2}\\
&\qquad = \pm \mu_k \left(\vec a, \vec Z_k^\pm \right)_{L^2}
+ \left(\begin{pmatrix}0\\(f'(\MS)-f'(W_k))a\end{pmatrix}, \vec Z_k^\pm \right)_{L^2}.
\end{align*}
Now, we estimate the error terms one by one, starting with the last term on the right-hand side
of the above estimate.
Using 
\[
\left| f'(\MS)-f'(W_k )\right| \lesssim (|\MS|^\frac13+|W_k|^\frac13)|\MS-W_k|
\lesssim (|\MS|^\frac13+|W_k|^\frac13)\left(|\MS-\tsum_l W_l|+\tsum_{l\neq k} |W_l|\right)
\]
and so by \eqref{eq:dZ}, \eqref{eq:C3} and \eqref{eq:c2},
\begin{align*}
& \left\|(f'(\MS)-f'(W_k)) \zeta_k \right\|_{L^\frac{10}7}\\
&\quad \lesssim \left( \|\MS \zeta_k\|_{L^\frac56}^\frac13 + \|W_k\zeta_k\|_{L^\frac56}^\frac13\right)
\biggl(\big\|(\MS-\sum_l W_l)\zeta_k^\frac23\big\|_{L^\frac{10}3}+\sum_{l\neq k} \big\|W_l\zeta_k^\frac23\big\|_{L^\frac{10}3}\biggr)
\lesssim t^{-2}.
\end{align*}
Thus,
\begin{equation*}
\left|\left(\begin{pmatrix}0\\(f'(\MS)-f'(W_k ))a \end{pmatrix}, \vec Z_k^\pm \right)_{L^2} \right|
\lesssim \left\|(f'(\MS)-f'(W_k)) \zeta_k\right\|_{L^\frac{10}7}
\|a\|_{L^\frac{10}3} \lesssim t^{-2} \|\vec a\|_\mathcal{E}.
\end{equation*}
Then, it follows from \eqref{eq:dZ} and \eqref{eq:c3} that
\begin{equation*}
\left| \frac{\dot \lambda_k}{\lambda_k}\left(\vec a,\vec{\widetilde\Lambda}\vec Z_k^\pm\right)_{L^2}\right|
+\left| \dot \by_k \cdot \left(\vec a, \nabla \vec Z_k^\pm\right)_{L^2}\right|
\lesssim t^{-2} \left(\| a \|_{L^\frac{10}3}+\|b\|_{L^2}\right)\lesssim t^{-2} \|\vec a\|_\cE.
\end{equation*}
Besides, by \eqref{eq:dZ} and \eqref{eq:o10} we have
\begin{equation*}
\left|\left(\vec \pp_1 a, \vec Z_k^\pm \right)_{L^2}\right|
\lesssim \|p_1 \zeta_k\|_{L^\frac{10}7}\|a\|_{L^\frac{10}3}
\lesssim (|A| e^{-\beta_m}+q) \|\vec a\|_\cE.
\end{equation*}
Using \eqref{eq:ov0}, we have
\begin{equation*}
\left|\left(\vec \pp_2, \vec Z_k^\pm \right)_{L^2}\right|
 \lesssim \|a\|_{L^\frac{10}3}^2 \lesssim \|\vec a\|_\cE^2.
\end{equation*}
By \eqref{eq:o20} and \eqref{eq:c2}, we have
\begin{align*}
\left|\left(\vec \pp_3, \vec Z_k^\pm\right)_{L^2}\right|
& \lesssim A^2 e^{-2\beta_m} + t^{-\frac23} q^2 + q^\frac73 + t^{-2} (|A|e^{-\beta_m}+q)\\
& \lesssim A^2 e^{-2\beta_m} + q^2 + t^{-2} (|A|e^{-\beta_m}+q) .
\end{align*}
Summarizing for the error term in $\pp$, we have obtained
\begin{align*}
\left|\left(\vec \pp, \vec Z_k^\pm \right)_{L^2} \right|
& \lesssim (|A| e^{-\beta_m}+q) \|\vec a\|_{\mathcal E} + \|\vec a\|_{\mathcal E}^2 
+A^2 e^{-2\beta_m} + q^2 + t^{-2} (|A|e^{-\beta_m}+q) \\
& \lesssim \|\vec a\|_{\mathcal E}^2 + q^2 + t^{-2}q+ |A|t^{-2}e^{-\beta_m}.
\end{align*}
Then, using \eqref{oZ} and \eqref{eq:C3} we have, for $j$, $k\neq l$,
\[
\left(\vec Z_k^\Lambda, \vec Z_k^\pm\right)=0,\quad
\left(\vec Z_k^{(j)}, \vec Z_k^\pm\right)=0,\quad
\left|\left(\vec Z_l^\Lambda, \vec Z_k^\pm\right)\right|\lesssim t^{-3},\quad
\left|\left(\vec Z_l^{(j)}, \vec Z_k^\pm\right)\right|\lesssim t^{-4},
\]
and thus
\[
\left| \left(\vec P^{\rm (I)} , \vec Z_k^\pm \right)_{L^2} \right| 
\lesssim t^{-3} \sum_l \left( \left| \frac{d}{dt} \alpha_l^\Lambda \right| + \left| \frac{d}{dt} \alpha_l^{(j)} \right| \right).
\]
By \eqref{eq:ke}, this implies the estimate
\[
\left| \left(\vec P^{\rm (I)} , \vec Z_k^\pm \right)_{L^2} \right| 
\lesssim t^{-3} \left(\|\vec a\|_{\mathcal E} + t^{-2} q + |A| t^{-2} e^{-\beta_m}\right).
\]
For $\vec P^{\rm (II)}$, we have
\begin{align*}
\left| \left(\vec P^{\rm (II)}, \vec Z_k^\pm \right)_{L^2} \right| 
&\lesssim \sum_l \left( \left| \frac{d}{dt} \alpha_l^\Lambda \right| + \left| \frac{d}{dt} \alpha_l^{(j)} \right| \right)\\
&\lesssim q \left(\|\vec a\|_{\mathcal E} + t^{-2} q + |A| t^{-2} e^{-\beta_m}\right).
\end{align*}
By \eqref{eq:c3} and the definition of $E_m$,
\[
\left|A \left(\vec M_E ,\vec Z_k^\pm\right)_{L^2}\right|
\lesssim |A| t^{-2} e^{-\beta_m}.
\]
Last, by \eqref{eq:c3}, we have
\[
\left| \left(\vec M_Q , \vec Z_k^\pm \right)_{L^2} \right| \lesssim t^{-2} q.
\]
All the error terms have been estimated and \eqref{eq:zo} is proved.
\end{proof}

\subsection{Definition of the energy functional}

We define the following energy functional $\cH$, which is adapted to the equation \eqref{eq:vv} of $(a,b)$
\begin{equation*}
\mathcal{H} = \int |\nabla a|^2 + b^2 +2 \chi (\pun a) b
- 2 \left( F(\MS{+}AE_\mm{+}Q{+}a) - F(\MS{+}AE_\mm{+}Q) - f(\MS{+}AE_\mm{+}Q) a \right)
\end{equation*}
where the function $\chi$ is defined in \eqref{defchiK}.
Note the following straightforward estimate
\begin{equation}\label{eq:eH}
|\cH|\lesssim \|\vec a\|_\cE^2.
\end{equation}
Moreover, we set
\[
\mathcal{N}_\Omega =\int_\Omega \left(|\nabla a|^2+b^2+2 \chi (\pun a) b \right),
\]
and
\[
\mathcal{N}_{\Omega^C} =\int_{\Omega^C} \left(|\nabla a|^2+b^2\right).
\]
Notation here is similar to the one introduced in Section \ref{s:3.5}, but this should not lead to confusion.
Note that as in \eqref{nintee},
\begin{equation}\label{nint} 
\Nint \geq \overline \ell \int_{\Omega} \left|\frac {\chi}{\overline \ell} \pun a + b\right|^2 
 + (1-\overline \ell) \int_{\Omega}\left( |\nabla a|^2 + b^2\right).
\end{equation}

\subsection{Coercivity of the energy}
\begin{lemma}\label{le:co}
Assuming \eqref{BPweak}, for some $\mu,C>0$, it holds
\begin{equation}\label{eq:cH}
\mathcal{H}\geq (1-C \sigma) \mathcal{N}_\Omega + \mu \mathcal N_{\Omega^C} 
- C B^2.
\end{equation}
\end{lemma}
\begin{proof}
By the localized coercivity property \eqref{eq:2.30}, the closeness of $S$ to a sum of solitons \eqref{eq:c2}, and standard localization arguments
(see for example \cite[Proof of (4.21)]{MMwave1}), we have, for some constants $C,\mu>0$, for $t$ large,
\begin{align*}
& \int |\nabla a|^2 + b^2 - \int f'(\MS) a^2 
+ 2\int \chi \pun a b\geq (1-\tfrac 12 \sigma) \cN_\Omega + 2 \mu \cN_{\Omega^C}\\
&\qquad 
 - C \sum_k\biggl( \psh{a}{\Lambda_k W_k}^2
+\sum_{j=1}^5 \psh {a}{\partial_j W_k}^2
+ \psl {\vec {a}}{\vec Z_k^{+}}^2 + \psl {\vec {a}}{\vec Z_k^{-}}^2\biggr) .
\end{align*}
Thus, by \eqref{eq:OR} and \eqref{oZ}, we obtain
\begin{align*}
& \int |\nabla a|^2 + b^2 - \int f'(\MS) a^2 
+ 2\int \chi \pun a b
\geq ((1-\tfrac 12 \sigma)\cN_\Omega  + 2 \mu \cN_{\Omega^C} - C B^2.
\end{align*}
Now, we estimate
\begin{align*}
& \left|\int \left( F(\MS{+}AE_\mm{+}Q{+}a) - F(\MS{+}AE_\mm{+}Q) - f(\MS{+}AE_\mm{+}Q) a
-f'(\MS) a^2\right)\right|\\
& \quad \lesssim
\int \left|F(\MS{+}AE_\mm{+}Q{+}a) - F(\MS{+}AE_\mm{+}Q) - f(\MS{+}AE_\mm{+}Q) a - 
f'(\MS{+}AE_\mm{+}Q)a^2\right|\\
&\qquad+ \int \left|f'(\MS{+}AE_\mm{+}Q)-f'(\MS)\right| a^2 \\
&\quad \lesssim \int \left(|\MS|^\frac13+|AE_m|^\frac13+|Q|^\frac13\right)|a|^3
+ \int |\MS|^{\frac13} \left(|AE_m|+|Q|\right) a^2\\
&\quad \lesssim (|A| e^{-\beta_m} + q + \|\vec a\|_\cE ) \|\vec a\|_\cE^2.
\end{align*}
The result then follows from the bootstrap estimate \eqref{BPweak} and \eqref{nint}.
\end{proof}

\subsection{Variation of the energy}
\begin{lemma}\label{le:vE}
Suppose \eqref{BPweak}. Then, the following estimates hold, for $\delta>0$ small.
\begin{enumerate}[label=\emph{(\roman*)}]
\item Estimate on the variation of the energy.
\begin{equation}\label{eq:di}\begin{aligned}
& - \frac d{dt} \left( t^{1+\delta} \mathcal H \right)
\\ &\quad \lesssim t^{1+\delta} \left(\|\vec a\|_\cE^2+|A|t^{-2} e^{-\beta_m} + q^\frac73+t^{-\frac 23} q^2 +q\|\vec a\|_\cE + t^{-\frac32} \|\vec a\|_\cE  +t^{-2} q\right)\|\vec a\|_\cE + t^\delta  B^2.
\end{aligned}\end{equation}
\item Refined estimate on the variation of the energy.
\begin{equation}\label{eq:ri}
\begin{aligned}
&- \frac d{dt} \left( t^{1+\delta} \mathcal H \right) + t^{1+\delta} \sum_{l=1}^4\br_l\\
&\quad  \lesssim t^{1+\delta} \left(\|\vec a\|_\cE^2+|A|t^{-2} e^{-\beta_m} + q^\frac73+t^{-\frac 23} q^2 +q\|\vec a\|_\cE + t^{-\frac32} \|\vec a\|_\cE  +t^{-\frac52} q\right)\|\vec a\|_\cE 
+ t^\delta B^2.
\end{aligned}
\end{equation}
where
\begin{equation}\label{eq:rr}
\begin{aligned}
\br_1 & = \sum_{k} r_{1,k} ,\quad r_{1,k} =- \frac{d_k}{t^2}\alpha_k^\Lambda\int a (\Delta_{\ell_k} + f'(W_k)) \Lambda_k^2 W_k,\\
\br_2 & = \sum_{j,k} r_{2,j,k} ,\quad r_{2,k} =- \frac{d_k}{t^2}\alpha_k^{(j)}\lambda_k \int a (\Delta_{\ell_k} + f'(W_k)) \Lambda_k\partial_j W_k,\\
\br_3 & = \sum_k r_{3,k}, \quad r_{3,k}= c_k \alpha_k^\Lambda \int (\ell_k \pun a + b) f''(W_k) v_k \Lambda_k W_k, \\
\br_4 & = \sum_{j,k} r_{4,j,k}, \quad r_{4,j,k}= c_k \alpha_k^{(j)} \lambda_k \int (\ell_k \pun a + b) f''(W_k)  v_k \partial_j W_k .
\end{aligned}
\end{equation}
\end{enumerate}
\end{lemma}
\begin{remark}
With respect to \eqref{eq:di}, the estimate \eqref{eq:ri} is more precise as it identifies
in the expression of $\frac d{dt} \cH$ all the terms (denoted by $\br_l$, for $l=1,2,3,4$) that are of order $t^{-2} q \|\vec a \|_\cE$.
This will be useful in the proof of the classification result, since then such terms are critical and cannot be considered as error terms.

The functions $(\Delta_{\ell_k} + f'(W_k)) \Lambda_k^2 W_k$ and
$(\Delta_{\ell_k} + f'(W_k)) \lambda_k \Lambda_k \partial_j W_k$
involved in the definitions of $\br_1$ and $\br_2$ can be expressed more explicitly using relations \eqref{L2}.
\end{remark}
\begin{proof}
This proof is similar to the one of \eqref{time} in Lemma \ref{mainprop}.
We decompose
\begin{align*}
\frac{d}{dt}\mathcal{H}
&=
 \int \partial_t \left(|\nabla a|^2 + b^2 
- 2 \left( F(\MS{+}AE_\mm{+}Q{+}a) - F(\MS{+}AE_\mm{+}Q) - f(\MS{+}AE_\mm{+}Q) a \right)\right)\\
&\quad + 2 \int \chi \partial_t ( \pun a b)
+ 2 \int (\partial_t\chi) \pun a b
= \bg_1+\bg_2+\bg_3.
\end{align*}

\emph{Computation of $\bg_1$.}
We compute
\begin{align*}
\bg_1 & = 2 \int \left(\nabla\partial_t a \cdot \nabla a + \partial_t b b\right)
-2 \int \partial_t a \left( f(\MS{+}AE_\mm{+}Q{+}a)-f(\MS{+}AE_\mm{+}Q)\right) \\
&\quad -2 \int \partial_t (\MS{+}AE_\mm{+}Q) p_2.
\end{align*}
Using \eqref{eq:vv}, the definition of $\pp_3$ and integrating by parts, we have
\begin{align}
\bg_1 & = 
2 \int b \pp_3 - 2 \int b (P_2-AM_F-M_R) + 2 \int a\Delta(P_1-AM_E-M_Q) \nonumber\\
&\quad 
+ 2 \int (P_1-AM_E-M_Q)\left( f(\MS{+}AE_\mm{+}Q{+}a)-f(\MS{+}AE_\mm{+}Q)\right)\label{eq:g1}\\ 
&\quad
-2 \int \partial_t (\MS{+}AE_\mm{+}Q) p_2.\nonumber
\end{align}

\emph{Computation of $\bg_2$.}
Using \eqref{eq:vv}, we have
\begin{align*}
\bg_2 & = 2 \int \chi \partial_t \pun a b
+ 2 \int \chi \pun a \partial_t b\\
& = 2 \int \chi \pun b b
+ 2 \int \chi \pun a \left(\Delta a + f(\MS{+}AE_\mm{+}Q{+}a)-f(\MS{+}AE_\mm{+}Q)\right)
+ 2 \int \chi\pun a p_3\\
&\quad - 2 \int \chi \pun a (P_2-AM_F-M_R) -2\int \chi b \pun (P_1-AM_E-M_Q).
\end{align*}
After integration by parts, we obtain
\begin{align*}
\bg_2
& = - \int \pun \chi \left(b^2+(\pun a)^2-|\overline\nabla a|^2\right)\\
&\quad -2\int (\pun \chi) \left(F(\MS{+}AE_\mm{+}Q{+}a)-F(\MS{+}AE_\mm{+}Q)-f(\MS{+}AE_\mm{+}Q)a\right)\\
&\quad -2\int\chi\pun (\MS{+}AE_\mm{+}Q) \left( f(\MS{+}AE_\mm{+}Q{+}a)-f(\MS{+}AE_\mm{+}Q)-f'(\MS{+}AE_\mm{+}Q)a\right)\\
&\quad + 2 \int \chi\pun a p_3+ 2 \int \pun \chi a (P_2-AM_F-M_R)\nonumber\\
&\quad + 2 \int \chi a \pun(P_2-AM_F-M_R) -2\int \chi b \pun (P_1-AM_E-M_Q).
\end{align*}
Using \eqref{derchi}, we obtain
\begin{align}
\bg_2
& = -\frac{1}{(1-2\sigma)t} \int_\Omega \left(b^2+(\pun a)^2-|\overline\nabla a|^2\right)\nonumber\\
&\quad -\frac{2}{(1-2\sigma)t} \int_\Omega \left(F(\MS{+}AE_\mm{+}Q{+}a)-F(\MS{+}AE_\mm{+}Q)-f(\MS{+}AE_\mm{+}Q)a\right)\nonumber\\
&\quad -2\int\chi\pun (\MS{+}AE_\mm{+}Q) \left( f(\MS{+}AE_\mm{+}Q{+}a)-f(\MS{+}AE_\mm{+}Q)-f'(\MS{+}AE_\mm{+}Q)a\right)\nonumber\\
&\quad + 2 \int \chi\pun a p_3+ \frac2{(1-2\sigma)t} \int_\Omega a (P_2-AM_F-M_R)\nonumber\\
&\quad + 2 \int \chi a \pun(P_2-AM_F-M_R) -2\int \chi b \pun (P_1-AM_E-M_Q).\label{eq:g2}
\end{align}

\emph{Computation of $\bg_3$.}
Using \eqref{derchi}, we have
\begin{equation}\label{eq:g3}
\bg_3 = - \frac{2}{(1-2\sigma)t}\int_\Omega \frac{x_1}{t} \pun a b.
\end{equation}

Gathering the expressions obtained in \eqref{eq:g1}, \eqref{eq:g2} and \eqref{eq:g3}, we reorganize
\begin{equation*}
\frac{d}{dt}\mathcal{H}
= \bh_1+2 \sum_{i=2}^6 \bh_i,
\end{equation*}
where
\begin{align*}
\bh_1&=-\frac{1}{(1-2\sigma)t} \int_\Omega \left(b^2+(\pun a)^2+2\frac{x_1}t\pun a b-|\overline\nabla a|^2\right),\\
\bh_2&= -\frac{1}{(1-2\sigma)t} \int_\Omega \left(F(\MS{+}AE_\mm{+}Q{+}a)-F(\MS{+}AE_\mm{+}Q)-f(\MS{+}AE_\mm{+}Q)a\right),\\
\bh_3&= - \int (\partial_t+\chi\pun)(\MS{+}AE_m{+}Q) p_2,\\
\bh_4&= \int a \left( \Delta(P_1-AM_E-M_Q) + \chi \pun (P_2-AM_F-M_R)\right),\\
&\quad + \int (P_1-AM_E-M_Q)\left( f(\MS{+}AE_\mm{+}Q{+}a)-f(\MS{+}AE_\mm{+}Q)\right),\\
\bh_5&= - \int b \left( P_2-AM_F-M_R +\chi \pun (P_1-AM_E-M_Q)\right),\\
\bh_6&= \int (\chi \pun a + b) \pp_3 + \frac1{(1-2\sigma)t} \int_\Omega a (P_2-AM_F-M_R).
\end{align*}
We now estimate the terms $\bh_j$.

\emph{Estimate of $\bh_1$.}
We claim
\begin{equation}\label{eq:h1}
\bh_1 \geq -\frac{1+C\sigma}{t} \Nint.
\end{equation}
Indeed, we have
\[
\bh_1 \geq - \frac{1}{(1-2\sigma)t} \int_\Omega \left(b^2+(\pun a)^2
+2\frac{x_1}{t} \pun a b\right),
\]
and so, by \eqref{nint} and the definition of $\chi$ in \eqref{defchiK},
\begin{align*}
-((1-2\sigma)t) {\bf h_1} 
& \leq \Nint+ 2 \int_{\Omega} \left( \frac{x_1}t- \chi\right) \pun a b\\
& \leq \Nint + C\sigma \int_\Omega \left(|\pun a|^2 + b^2\right) \leq (1+C \sigma) \Nint.
\end{align*}

\emph{Estimate of $\bh_2$.}
We claim
\begin{equation}\label{eq:h2}
|\bh_2|\lesssim t^{-3} \|\vec a\|_\cE^2+ t^{-1} \|\vec a\|_\cE^\frac{10}3.
\end{equation}
By standard estimates, we have
\begin{align*}
|\bh_2|& \lesssim \frac 1t\int_\Omega a^2 \left( |\MS|^\frac43+|AE_m|^\frac43+|Q|^\frac43+|a|^\frac43 \right)\\
&\lesssim \frac 1t\|a\|_{L^\frac{10}3}^2 \left(\|\MS\|_{L^\frac{10}3(\Omega)}^\frac43
+|A|\|E_m\|_{L^\frac{10}3(\Omega)}^\frac43+\|Q\|_{L^\frac{10}3}^\frac43+\|a\|_{L^\frac{10}3}^\frac43\right).
\end{align*}
Moreover, by \eqref{eq:c2} and $|W_k|\lesssim \omega_k^3$ (see \eqref{eq:dW})
\[
\|\MS\|_{L^\frac{10}3(\Omega)}
\lesssim \|\MS-\tsum_k W_k\|_{L^\frac{10}3}+\|\tsum_k W_k\|_{L^\frac{10}3(\Omega)}
\lesssim t^{-\frac32}.
\]
By $|E_m|\lesssim \zeta_m$, we obtain $\|E_m\|_{L^\frac{10}3(\Omega)}\lesssim t^{-10}$ and by \eqref{eq:dW},
$\|Q\|_{L^\frac{10}3(\Omega)}\lesssim t^{-\frac32} q$.
Thus, we have proved \eqref{eq:h2}.

\emph{Estimate of $\bh_3$.}
We claim
\begin{equation}\label{eq:h3}
|\bh_{3}|\lesssim t^{-\frac32}  \|\vec a\|_\cE^2  .
\end{equation}
We decompose $\bh_3$ as follows
\begin{align*}
\bh_3&= - \int (\partial_t+\chi\pun)\MS p_2 - A \int (\partial_t+\chi\pun) E_m p_2 - \int (\partial_t+\chi\pun)Q p_2\\
&= \bh_{3,1}+\bh_{3,2}+\bh_{3,3}.
\end{align*}
First, we estimate, using \eqref{eq:c4},
\begin{align*}
\|(\partial_t+\chi\pun ) \MS \|_{L^\frac{10}3}
&\lesssim \|\partial_t (\MS-\tsum_k W_k) \|_{L^\frac{10}3}
+\|(\partial_t+\chi\pun ) \tsum_k W_k \|_{L^\frac{10}3}
+\| \pun (\MS-\tsum_k W_k) \|_{L^\frac{10}3}\\
& \lesssim \|\nabla\partial_t (\MS-\tsum_k W_k) \|_{L^2}
+\|(\partial_t+\chi\pun ) \tsum_k W_k \|_{L^\frac{10}3}
+\|\nabla \pun (\MS-\tsum_k W_k) \|_{L^2}\\
&\lesssim t^{-\frac32} + \tsum_k \|(\partial_t+\ell_k \pun ) W_k \|_{L^\frac{10}3}
+ \tsum_k \|(\chi - \ell_k) \pun W_k \|_{L^\frac{10}3}.
\end{align*}
We have by \eqref{eq:tG},
\begin{equation*}
(\partial_t+\ell_k \pun ) W_k = -\frac{\dot \lambda_k}{\lambda_k}  \Lambda_k W_k - \dot \by_k\cdot  \nabla W_k,
\end{equation*}
and thus by \eqref{eq:c3}, we obtain
\begin{equation*}
\| (\partial_t+\ell_k \pun ) W_k \|_{L^\frac{10}3} \lesssim t^{-2}.
\end{equation*}
Moreover, by \eqref{defchiK} and \eqref{eq:dW},
\begin{equation*}
\|(\chi - \ell_k) \pun W_k\|_{L^\frac{10}3}^\frac{10}3
\lesssim 
\|\pun W_k \|_{L^\frac{10}3(x_1\not\in (\ell_k^-t,\ell_k^+t))}^\frac{10}3
\lesssim \int_{r>\sigma^2 t} r^{-\frac{40}3} r^4 dr \lesssim t^{-\frac{25}3}.
\end{equation*}
Thus, we have proved
\begin{equation}\label{eq:00}
\|(\partial_t+\chi\pun ) \MS \|_{L^\frac{10}3}
\lesssim t^{-\frac32}.
\end{equation}
Second, we observe that
\[
\left| p_2 \right|
\lesssim |a|^2 (|\MS|^\frac13+|E_\mm|^\frac13+|Q|^\frac13) + |a|^\frac73,
\]
and thus, also using \eqref{BPweak},
\begin{equation}\label{eq:z3}
 \left\| p_2\right\|_{L^\frac{10}7} 
  \lesssim \left(\|\MS\|_{L^\frac{10}3}^\frac13+\|E_\mm\|_{L^\frac{10}3}^\frac13
+\|Q\|_{L^\frac{10}3}^\frac13  + \|a\|_{L^\frac{10}3}^\frac13\right) \|a\|_{L^\frac{10}3}^2 
\lesssim \|\vec a\|_\cE^2 .
\end{equation}
Combining \eqref{eq:00} and \eqref{eq:z3}, we have proved
\begin{equation*}
|\bh_{3,1}|\lesssim t^{-\frac32} \|\vec a\|_\cE^2 .
\end{equation*}
To estimate $\bh_{3,2}$ and $\bh_{3,3}$, using \eqref{eq:z3}, we only need to estimate
$\|(\partial_t+\chi\pun ) E_m \|_{L^\frac{10}3}$ and $\|(\partial_t+\chi\pun ) Q \|_{L^\frac{10}3}$.
Note that
\begin{align*}
\|(\partial_t+\chi\pun ) E_m \|_{L^\frac{10}3} & \lesssim 
\| \partial_t E_m \|_{L^\frac{10}3}+\| \pun E_m \|_{L^\frac{10}3} \lesssim e^{-\beta_m},\\
\|(\partial_t+\chi\pun ) Q \|_{L^\frac{10}3} & \lesssim 
\tsum_k \|(\partial_t+\ell_k\pun ) Q_k \|_{L^\frac{10}3}
+\tsum_k \|(\chi-\ell_k) Q_k \|_{L^\frac{10}3}\lesssim t^{-\frac32}q,
\end{align*}
which is sufficient to complete the proof of \eqref{eq:h3}.

\emph{Estimate of $\bh_4$.} We claim
\begin{equation}\label{eq:h4}
|\bh_4|\lesssim 
 \left(\|\vec a\|_\cE^2+|A|t^{-2} e^{-\beta_m} +q\|\vec a\|_\cE+ t^{-\frac32} \|\vec a\|_\cE +t^{-2} q\right)\|\vec a\|_\cE
\end{equation}
and
\begin{align}
&\left| \bh_{4} 
-\int a \sum_k \frac{d_k}{t^2}\left(\alpha_k^\Lambda \Theta_k^{\Lambda,\Lambda}
+ \sum_j \alpha_k^{(j)} \Theta_k^{\Lambda,j}\right)\right|\nonumber \\
&\quad 
\lesssim  \left(\|\vec a\|_\cE^2+|A|t^{-2} e^{-\beta_m} +t^{-2} q^2 + q \|\vec a\|_\cE+ t^{-\frac32} \|\vec a\|_\cE +t^{-\frac52} q\right)\|\vec a\|_\cE \label{eq:h4b}
\end{align}
where
\begin{align*}
\Theta_k^{\Lambda,\Lambda} = -\Delta_{\ell_k}\Lambda_k^2W_k
-f'(W_k)\Lambda_k^2W_k,\quad
\Theta_k^{\Lambda,j} = -\Delta_{\ell_k}\Lambda_k\partial_j W_k
-f'(W_k)\Lambda_k\partial_j W_k.
\end{align*}

We split $\bh_4$ into two terms
\begin{equation*}
\bh_4
= \int a \Delta_{4,1} + \int (P_1-AM_E-M_Q)\Delta_{4,2}
=\bh_{4,1}+\bh_{4,2}.
\end{equation*}
where
\begin{align*}
\Delta_{4,1} & = \Delta(P_1-AM_E-M_Q) + \chi \pun (P_2-AM_F-M_R)
+f'(\MS) (P_1-AM_E-M_Q)\\
\Delta_{4,2} & = f(\MS{+}AE_\mm{+}Q{+}a)-f(\MS{+}AE_\mm{+}Q)-f'(\MS)a .
\end{align*}
Again, we split
\begin{align*}
\Delta_{4,1} & = \left[\Delta P_1^{\rm (I)} + \chi \pun P_2^{\rm (I)} + f'(\MS) P_1^{\rm (I)}\right]
+\left[\Delta P_1^{\rm (II)} + \chi \pun P_2^{\rm (II)} + f'(\MS) P_1^{\rm (II)}\right]\\
& \quad
-\left[A (\Delta M_E + \chi \pun M_F + f'(\MS) M_E)\right]\\
& \quad
 - \left[\Delta M_Q^{\rm (I)} + \chi \pun M_R^{\rm (I)} + f'(\MS) M_Q^{\rm (I)}\right] 
- \left[\Delta M_Q^{\rm (II)} + \chi \pun M_R^{\rm (II)} + f'(\MS) M_Q^{\rm (II)}\right]\\
&=\Delta_{4,1,1}+\Delta_{4,1,2}+\Delta_{4,1,3}+\Delta_{4,1,4}+\Delta_{4,1,5}.
\end{align*}
For the term $\Delta_{4,1,1}$, it is important to use some cancellations related to the properties of the operator
$-\Delta-f'(W)$. By the definitions of $\vec Z_\ell^\Lambda$ and $\vec Z_\ell^{(j)}$,
\begin{align*}
\Delta_{4,1,1} & =
\sum_k \frac{d}{dt}\alpha_k^\Lambda \left(\Delta_{\ell_k} \Lambda_k W_k + f'(W_k) \Lambda_k W_k \right)
+\sum_{j,k}\lambda_k \frac{d}{dt}\alpha_k^{(j)}\left(\Delta_{\ell_k} \partial_j W_k + f'(W_k) \partial_j W_k \right)\\
&\quad + \sum_k \frac{d}{dt}\alpha_k^\Lambda \left(f'(\MS)-f'(W_k)\right) \Lambda_k W_k 
+\sum_{j,k}\lambda_k\frac{d}{dt}\alpha_k^{(j)}\left(f'(\MS)-f'(W_k)\right) \partial_j W_k\\
&\quad - \sum_{k} \ell_k \frac{d}{dt}\alpha_k^{\Lambda} (\chi-\ell_k) \pun^2 \Lambda_k W_k
- \sum_{j,k}\ell_k\lambda_k \frac{d}{dt}\alpha_k^{(j)} (\chi-\ell_k) \pun^2 \partial_j W_k.
\end{align*}
By \eqref{LW}, the first line on the right-hand side vanishes. 
For the second line, we note that
\[
\left|f'(\MS)-f'(W_k)\right| \left( |\Lambda_k W_k|+|\partial_j W_k| \right)
\lesssim (|\MS|^\frac13+|W_k|^\frac13)\left(|\MS-\tsum_lW_l|+\tsum_{l\neq k}|W_l|\right)\omega_k^3
\]
and thus, using \eqref{eq:c2} and \eqref{eq:C2},
\begin{align*}
&\|\left(f'(\MS)-f'(W_k)\right) \Lambda_k W_k \|_{L^\frac{10}7}
+\|\left(f'(\MS)-f'(W_k)\right) \partial_j W_k\|_{L^\frac{10}7}\\
& \quad \lesssim 
\left(\|\MS\|_{L^\frac{10}3}+\|W_k\|_{L^\frac{10}3}\right)^\frac13
 \left(\|\MS-\tsum_lW_l\|_{L^\frac{10}3}+\tsum_{l\neq k}\|\omega_k\omega_l\|_{L^5}^3\right)
\lesssim t^{-2}.
\end{align*}
Moreover, by the definition of $\chi$, we have
\begin{align*}
\|(\chi-\ell_k) \pun^2 \Lambda_k W_k\|_{L^\frac{10}7} +\|(\chi-\ell_k) \pun^2 \partial_j W_k\|_{L^\frac{10}7}
& \lesssim \|\omega_k^5\|_{L^\frac{10}7(x\not\in(\ell_k^- t,\ell_k^+t))}\lesssim t^{-\frac32}.
\end{align*}
Thus, using \eqref{eq:ke}, we obtain
\begin{equation*}
\|\Delta_{4,1,1}\|_{L^\frac{10}7}
\lesssim t^{-\frac32} \left( \|\vec a\|_\cE+t^{-2}q+|A|t^{-2}e^{-\beta_m}\right).
\end{equation*}
Using the definition of $\vec P^{\rm (II)}$ and \eqref{eq:ke}, we have directly
\[
\|\Delta_{4,1,2}\|_{L^\frac{10}7}
\lesssim \left(\sum_k \left| \frac{d}{dt} \alpha_k^\Lambda \right| + \sum_{j,k} \left| \frac{d}{dt} \alpha_k^{(j)} \right|\right)q
\lesssim \left(\|\vec a\|_{\mathcal E} + t^{-2} q + |A| t^{-2} e^{-\beta_m}\right) q.
\]
Then, we estimate using the definition of $\vec M_E$ and \eqref{eq:c3}
\begin{equation*}
\|\Delta_{4,1,3}\|_{L^\frac{10}7}\lesssim |A| t^{-2} e^{-\beta_m}.
\end{equation*}
For $\Delta_{4,1,4}$, we use \eqref{eq:c3} to estimate
\[
\|\Delta_{4,1,4}\|_{L^\frac{10}7} \lesssim t^{-2}q \sum_k \|\omega_k^5\|_{L^\frac{10}7}\lesssim t^{-2}q.
\]
To prove a refined estimate on $\Delta_{4,1,4}$ (involving explicit correction terms of order $t^{-2} q$),
we decompose
\begin{align*}
\Delta_{4,1,4} 
&= - \sum_k\frac{\dot\lambda_k}{\lambda_k}\left(\Delta \Lambda_k Q_k^{\rm (I)}
+ \chi \pun \widetilde \Lambda_k R_k^{\rm (I)}
+f'(S) \Lambda_k Q_k^{\rm (I)} \right)\\
&\quad - \sum_{k} \dot\by_k \cdot\left(\Delta \nabla Q_k^{\rm (I)}
+ \chi \pun \nabla R_k^{\rm (I)} +f'(S) \nabla Q_k^{\rm (I)} \right).
\end{align*}
To identify the correction terms, we further decompose,
using $\pun \widetilde \Lambda_k R_k^{\rm (I)} = - \ell_k\pun^2\Lambda_k Q_k^{\rm (I)}$,
\begin{align*}
\Delta_{4,1,4} 
&= - \sum_k\frac{\dot\lambda_k}{\lambda_k}\left(\Delta_{\ell_k} \Lambda_k Q_k^{\rm (I)}
+f'(W_k) \Lambda_k Q_k^{\rm (I)} \right)\\
&\quad -\sum_k \frac{\dot\lambda_k}{\lambda_k} (\chi-\ell_k) \pun \widetilde \Lambda_k R_k^{\rm (I)} 
-\sum_k \frac{\dot\lambda_k}{\lambda_k} (f(S)-f(W_k))\Lambda_k Q_k^{\rm (I)} \\
&\quad - \sum_{k} \dot\by_k \cdot\left(\Delta \nabla Q_k^{\rm (I)}
+ \chi \pun \nabla R_k^{\rm (I)} +f'(S) \nabla Q_k^{\rm (I)} \right).
\end{align*}
We start by estimating error terms on the right-hand side.
By the definition of $\chi$ and \eqref{eq:c2}, we have
\begin{align*}
\| (\chi-\ell_k) \pun \widetilde \Lambda_k R_k^{\rm (I)} \|_{L^\frac{10}7} &\lesssim t^{-\frac12} q\\
\| (f(S)-f(W_k))\Lambda_k Q_k^{\rm (I)} \|_{L^\frac{10}7} &\lesssim t^{-2} q.
\end{align*}
Moreover, by \eqref{eq:c3}, we have $|\dot\by_k |\lesssim t^{-\frac52}$ (take $\delta\in(0,\frac12)$), and so
\[
\| \dot\by_k \cdot\left(\Delta \nabla Q_k^{\rm (I)}
+ \chi \pun \nabla R_k^{\rm (I)} +f'(S) \nabla Q_k^{\rm (I)} \right)\|_{L^{\frac{10}7}}\lesssim t^{-\frac52} q.
\]
We define the functions
\begin{align*}
\Theta_k^{\Lambda,\Lambda} & = -(\Delta_{\ell_k}+f'(W_k))\Lambda_k^2W_k,\\
\Theta_k^{\Lambda,j} &= -\lambda_k(\Delta_{\ell_k}+f'(W_k))\Lambda_k\partial_j W_k .
\end{align*}
Using the definition of $\vec Q_k^{\rm (I)}$, the previous estimates and \eqref{eq:c3} on $\dot\lambda_k/\lambda_k$, we have proved that
\begin{align*}
\left\|\Delta_{4,1,4} 
+\sum_k\frac{d_k}{t^2}\left(\alpha_k^\Lambda\Theta_k^{\Lambda,\Lambda}
+\sum_{j}  \alpha_k^{(j)}\Theta_k^{\Lambda,j}\right) \right\|_{L^\frac{10}7} 
\lesssim t^{-\frac52} q.
\end{align*}
Lastly, by \eqref{eq:c3},
\[
|\Delta_{4,1,5}|\lesssim   \sum_k\left( \left|\frac{\dot\lambda_k }{\lambda_k}(t) \right|+|\dot\by_k(t)|\right) q^2
\lesssim t^{-2} q^2.
\]
In conclusion,
\[
|\bh_{4,1}|\lesssim \|a\|_{L^\frac{10}3} \|\Delta_{4,1}\|_{L^\frac{10}7}
\lesssim  \left( t^{-\frac32}  \|\vec a\|_\cE  +q\|\vec a\|_\cE  +t^{-2}q+|A|t^{-2}e^{-\beta_m}\right)\|\vec a\|_\cE
\]
and more precisely,
\begin{align*}
&\left| \bh_{4,1} 
+\int a \sum_k \frac{d_k}{t^2}\left(\alpha_k^\Lambda \Theta_k^{\Lambda,\Lambda}
+ \sum_j \alpha_k^{(j)} \Theta_k^{\Lambda,j}\right)\right|\\
&\quad 
\lesssim  \left( t^{-\frac32}  \|\vec a\|_\cE  + q \|\vec a\|_\cE   +t^{-\frac52}q
+t^{-2} q^2+|A|t^{-2}e^{-\beta_m}\right)\|\vec a\|_\cE.
\end{align*}
Now, we deal with $\bh_{4,2}$. 
We observe that
\begin{equation*}
|\Delta_{4,2}|\lesssim 
\left(|\MS|^\frac13 +|AE_m|^\frac13+|Q|^\frac13\right)\left(|AE_m|+|Q|+|a|\right)|a|.
\end{equation*}
Thus, using \eqref{eq:c3} and then \eqref{eq:ke},
\begin{align*}
|\bh_{4,2}|
&\lesssim \int (|P_1| +|AM_E| +|M_Q| ) 
\left(|\MS|^\frac13 +|AE_m|^\frac13+|Q|^\frac13\right)\left(|AE_m|+|Q|+|a|\right)|a|\\
&\lesssim \left(\|P_1\|_{L^\frac{10}3}+|A|\|M_E\|_{L^\frac{10}3}+\|M_Q\|_{L^\frac{10}3}\right)
\left(\|\MS\|_{L^\frac{10}3}+\|AE_m\|_{L^\frac{10}3}+\|Q\|_{L^\frac{10}3}\right)^\frac13\\
&\qquad \times\left(\|AE_m\|_{L^\frac{10}3}+\|Q\|_{L^\frac{10}3}+\|a\|_{L^\frac{10}3}\right)\|a\|_{L^\frac{10}3}\\
&\lesssim \biggl(\sum_k \left|\frac{d}{dt}\alpha_k^\Lambda\right|+\sum_{k,l}\left|\frac{d}{dt}\alpha_k^{(j)}\right|
+|A|t^{-2}e^{-\beta_m} + t^{-2} q\biggr)\left(|A|e^{-\beta_m}+q+\|\vec a\|_\cE\right)\|\vec a\|_\cE\\
&\lesssim \biggl(\|\vec a\|_\cE+|A|t^{-2} e^{-\beta_m} + t^{-2} q\biggr)\left(|A|e^{-\beta_m}+q+\|\vec a\|_\cE\right)\|\vec a\|_\cE
\\
&\lesssim \left(A^2 e^{-2\beta_m}+t^{-2}q^2+q \|\vec a\|_\cE  +\|\vec a\|_\cE^2\right)\|\vec a\|_\cE.
\end{align*}
Combining the estimates on $\bh_{4,1}$ and $\bh_{4,2}$, we obtain \eqref{eq:h4} and \eqref{eq:h4b}.

\emph{Estimate of $\bh_5$.} We claim
\begin{equation}\label{eq:h5}
|\bh_5|\lesssim \left( t^{-\frac32} \|\vec a\|_\cE + |A|t^{-2} e^{-\beta_m} +t^{-\frac72} q   \right) \|\vec a\|_{\cE}.
\end{equation}
Set
\begin{align*}
\Delta_5 & = P_2+\chi \pun P_1 -A(M_F+\chi \pun M_E) - (M_R+\chi \pun M_Q) \\
& =\Delta_{5,1}+\Delta_{5,2}+\Delta_{5,3} .
\end{align*}
and start by
\[
|\bh_5|=\left| \int b \Delta_5 \right| \lesssim \|b\|_{L^2} \|\Delta_5\|_{L^2}
\lesssim \|\vec a\|_\cE \|\Delta_5\|_{L^2}.
\]
By the definition of $\vec P$, we have
\begin{align*}
\Delta_{5,1}
&= P_2+\chi \pun P_1\\
&=\sum_k \frac{d}{dt}\alpha_k^\Lambda (\chi-\ell_k) \pun \Lambda_k W_k
+ \sum_{j,k} \lambda_k \frac{d}{dt}\alpha_k^{(j)}(\chi-\ell_k) \pun \partial_j W_k \\
&\quad +\frac12 \sum_k \frac{d}{dt}(\alpha_k^\Lambda)^2 (\chi-\ell_k) \pun \Lambda_k^2 W_k 
+ \frac 12  \sum_{j,k}  \lambda_k\frac{d}{dt}(\alpha_k^\Lambda\alpha_k^{(j)}) (\chi-\ell_k) \pun \partial_j\Lambda_k W_k \\
&\quad
+ \sum_{j,j',k} \lambda_k^2 \frac{d}{dt}(\alpha_k^{(j')}\alpha_k^{(j)})(\chi-\ell_k) \pun \partial_{j'}\partial_j W_k
\end{align*}
and so
\[
| P_2+\chi \pun P_1 |
\lesssim \biggl(\sum_k \left|\frac{d}{dt}\alpha_k^\Lambda\right| 
+ \sum_{j,k}\left|\frac{d}{dt}\alpha_k^{(j)}\right|\biggr) |\chi-\ell_k| \omega_k^4.
\]
By \eqref{eq:ke},
\begin{equation*}
\left| \frac{d}{dt} \alpha_k^\Lambda \right| + \left| \frac{d}{dt} \alpha_k^{(j)} \right|
\lesssim \|\vec a\|_{\mathcal E} + t^{-2} q + |A| t^{-2} e^{-\beta_m} .
\end{equation*}
By the definition of $\chi$,
\[
\| (\chi-\ell_k) \omega_k^4 \|_{L^2} \lesssim t^{-\frac32}.
\]
Thus,
\[
\|\Delta_{5,1}\|_{L^2} \lesssim t^{-\frac32} \|\vec a\|_\cE +t^{-\frac72} q + |A|t^{-\frac72} e^{-\beta_m}.
\]
Moreover, by \eqref{eq:c3},
\begin{equation*}
\|\Delta_{5,2}\|_{L^2}  \lesssim \|\pun M_E\|_{L^2}+ \|M_F\|_{L^2} \lesssim |A|t^{-2}e^{-\beta_m}.
\end{equation*}
For $\Delta_{5,3}$, we observe from the definition of 
$M_Q^{\rm (I)}$ and then \eqref{eq:tL} that
\begin{align*}
\partial_1 M_Q^{\rm (I)} 
&= \sum_k\frac{\dot\lambda_k}{\lambda_k}
\left(\alpha_k^\Lambda\partial_1\Lambda_k\theta_k\Lambda W_{\ell_k}
+\sum_j \alpha_k^{(j)}\partial_1\Lambda_k\theta_k\partial_jW_{\ell_k}\right)\\
&\quad +\sum_{ k} \dot\by_k\cdot\left(\alpha_k^\Lambda\partial_1\nabla\theta_k\Lambda W_{\ell_k}
+\sum_j\alpha_k^{(j)}\partial_1\nabla\theta_k\partial_jW_{\ell_k}\right)\\
&= \sum_k\frac{\dot\lambda_k}{\lambda_k}
\left(\alpha_k^\Lambda\widetilde\Lambda_k\frac{\theta_k}{\lambda_k}\partial_1\Lambda W_{\ell_k}
+\sum_j\alpha_k^{(j)}\widetilde\Lambda_k\frac{\theta_k}{\lambda_k}\partial_1\partial_jW_{\ell_k}\right)\\
&\quad +\sum_{k} \dot\by_k\cdot\left(\alpha_k^\Lambda\nabla\frac{\theta_k}{\lambda_k}\partial_1\Lambda W_{\ell_k}
+\sum_j\alpha_k^{(j)} \nabla\frac{\theta_k}{\lambda_k}\partial_1\partial_jW_{\ell_k}\right).
\end{align*}
Thus, from the definition of $M_R^{\rm (I)}$, it holds
\begin{align*}
\chi \partial_1 M_Q^{\rm (I)}  + M_R^{\rm (I)}
&= \sum_k (\chi-\ell_k) \frac{\dot\lambda_k}{\lambda_k}
\left(\alpha_k^\Lambda\widetilde\Lambda_k\frac{\theta_k}{\lambda_k}\partial_1\Lambda W_{\ell_k}
+\sum_j\alpha_k^{(j)}\widetilde\Lambda_k\frac{\theta_k}{\lambda_k}\partial_1\partial_jW_{\ell_k}\right)\\
&\quad +\sum_{k}(\chi-\ell_k) \dot\by_k\cdot\left(\alpha_k^\Lambda\nabla\frac{\theta_k}{\lambda_k}\partial_1\Lambda W_{\ell_k}
+\sum_j\alpha_k^{(j)} \nabla\frac{\theta_k}{\lambda_k}\partial_1\partial_jW_{\ell_k}\right)
\end{align*}
and so, by \eqref{eq:c3},
\[
|\chi \partial_1 M_Q^{\rm (I)}  + M_R^{\rm (I)}|\lesssim t^{-2} q \sum_k |\chi-\ell_k| \omega_k^4.
\]
By the definition of $\chi$ in \eqref{defchiK}, it follows that
\[
\| \chi \partial_1 M_Q^{\rm (I)}  + M_R^{\rm (I)} \|_{L^2} \lesssim t^{-\frac 72} q.
\]
Proceeding similarly, for $\chi \partial_1 M_Q^{\rm (II)}  + M_R^{\rm (II)}$, we obtain
\[
\| \chi \partial_1 M_Q  + M_R \|_{L^2} \lesssim t^{-\frac 72} q^2.
\]

Therefore, we have obtained
\[
\| \Delta_5 \|_{L^2} \lesssim t^{-\frac32} \|\vec a\|_\cE + |A|t^{-2} e^{-\beta_m} +t^{-\frac 72} q .
\]
and \eqref{eq:h5} follows.

\emph{Estimate of $\bh_6$.} We claim
\begin{equation}\label{eq:h6}
|\bh_6|\lesssim  \left(t^{-\frac32} \|\vec a\|_\cE+ |A| t^{-2} e^{-\beta_m}  + q^\frac73+t^{-\frac 23} q^2 + t^{-2} q\right)   \|\vec a\|_{\cE}
\end{equation}
and
\begin{equation}\label{eq:h6b}\begin{aligned}
&\left|\bh_6- \sum_k \int (\ell_k\partial_1 a+b) f''(W_k) c_k v_k Q_k^{\rm (I)}\right|
\\&\quad \lesssim  \left(t^{-\frac32} \|\vec a\|_\cE+ |A| t^{-2} e^{-\beta_m}   + q^\frac73+t^{-\frac 23} q^2 + t^{-\frac52} q \right) \|\vec a\|_{\cE}.
\end{aligned}\end{equation}
For the first term $\bh_{6,1} = \int (\chi \pun a + b) \pp_3$, we note that 
\[
\left| \bh_{6,1}\right| \lesssim\|\vec a\|_{\cE}\|\pp_3\|_{L^2}.
\]
Moreover, we recall that by \eqref{eq:p4},
\begin{equation*}
\pp_3 = p_{3,1}+p_{3,2}+p_{3,3}+p_{3,4}+p_{3,5}+p_{3,6},
\end{equation*}
where
\begin{align*}
p_{3,1} & = f(\MS+AE_\mm +Q)- f(\MS) - f'(\MS) (AE_\mm+Q)
-\tfrac 12 f''(\MS) (AE_\mm+Q)^2 \\
p_{3,2} & = ( f'(\MS) - f'(W_m) ) AE_\mm \\ 
p_{3,3} & = \tsum_k(f'(\MS) - f'(W_k))Q_k \\
p_{3,4} & = \frac 12 f''(S) (A^2E_m^2+2AE_m Q) \\
p_{3,5} & = \frac 12 f''(S) \Big( Q^2 - \tsum_k  \big(Q_k^{\rm(I)}\big)^2 \Big) \\
p_{3,6} & = \frac 12\tsum_k ( f''(\MS) -f''(W_k)) \big(Q_k^{\rm(I)}\big)^2 .
\end{align*}
We observe that
\begin{align*}
|\pp_{3,1}|
&\lesssim |A E_m|^\frac73 +|Q|^\frac73,\\
|\pp_{3,2}|
&\lesssim \left(|\MS-\tsum_l W_l| + \tsum_{l\neq m} |W_l|\right) \big(|\MS|^\frac13+|W_m|^\frac13\big) |A| |E_m|,\\
|\pp_{3,3}|
&\lesssim \tsum_k \left(|\MS-\tsum_l W_l|+ \tsum_{l\neq k} |W_l|\right) \big(|\MS|^\frac13+|W_k|^\frac13\big) |Q_k|,\\
|p_{3,4}|&\lesssim |S|^\frac13 |AE_m|(|AE_m|+|Q|),\\
|p_{3,5}|&\lesssim |S|^\frac13 \Bigl(|Q^{\rm(II)}|(|Q^{\rm(I)}|+|Q^{\rm(II)}|)+\sum_{k\neq m}|Q_k^{\rm (I)}||Q_{m}^{\rm(I)}|\Bigr),\\
|p_{3,6}|&\lesssim \sum_k \big(Q_k^{\rm(I)}\big)^2|S-\sum_l W_l|^\frac13
+\sum_k \big(Q_k^{\rm(I)}\big)^2 \sum_{l\neq k}|W_l|^\frac13,
\end{align*}
which implies using \eqref{eq:C2}, \eqref{eq:C3}, \eqref{eq:c3} and \eqref{eq:c2} that
\begin{align*}
&\|\pp_{3,1}\|_{L^2}\lesssim |A|^\frac73 e^{-\frac73\beta_m} + q^\frac73 ,\\
&\|\pp_{3,2}\|_{L^2}\lesssim |A| t^{-2} e^{-\beta_m} ,\\
&\|\pp_{3,3}\|_{L^2}\lesssim t^{-2} q ,\\
&\|\pp_{3,4}\|_{L^2}\lesssim A^2 e^{-2\beta_m} + q |A| e^{-\beta_m}, \\
&\|\pp_{3,5}\|_{L^2}\lesssim q^3 + t^{-3} q^2 , \\
&\|\pp_{3,6}\|_{L^2}\lesssim t^{-\frac23} q^2
\end{align*}
and so
\[
\|\pp_{3}\|_{L^2}\lesssim |A| t^{-2} e^{-\beta_m} + q^\frac73+t^{-\frac 23} q^2 + t^{-2} q .
\]
Therefore,
\begin{equation*}
\left| \int (\chi \pun a + b) \pp_3\right| \lesssim \left(|A| t^{-2} e^{-\beta_m} + q^\frac73+t^{-\frac 23} q^2 + t^{-2} q \right) \|\vec a\|_{\cE}.
\end{equation*}
To obtain a refined estimate, we turn back to the estimate of the term $p_{3,3}$, writing
\begin{align*}
p_{3,3} &=\sum_k(f'(\MS)  - f'(W_k+c_k v_k))Q_k^{\rm (I)} + \sum_k(f'(W_k+c_k v_k) - f'(W_k))Q_k^{\rm (I)} \\
&\quad +\tsum_k(f'(\MS) - f'(W_k))Q_k^{\rm (II)}.
\end{align*}
We have
\begin{align*}
&| (f'(\MS) - f'(W_k+c_k v_k))Q_k^{\rm (I)} | \\
&\quad  \lesssim 
\left(|\MS-\tsum_l (W_l+c_l v_l)|
+ \tsum_{l\neq k} |W_l+c_lv_l|\right) \big(|\MS|^\frac13+|W_k|^\frac13+|v_k|^\frac13\big) |Q_k^{\rm (I)}|
\end{align*}
and so by \eqref{eq:Av}, \eqref{eq:oS},
\[
\| (f'(\MS) - f'(W_k+c_k v_k))Q_k^{\rm (I)} \|_{L^2} \lesssim t^{-\frac52} q.
\]
Moreover, by Taylor expansion and \eqref{eq:Av},
\[
|(f'(W_k+c_k v_k) - f'(W_k)  -f''(W_k) c_k v_k )Q_k^{\rm (I)}|
\lesssim |v_k|^\frac43 |Q_k^{\rm (I)}| \lesssim t^{-\frac83}q \omega_k^{-5}.
\]
Besides,
\[
|f'(\MS) - f'(W_k))Q_k^{\rm (II)}|\lesssim 
\left(|\MS-\tsum_l W_l|+ \tsum_{l\neq k} |W_l|\right) q^2 \omega_k^3.
\]
Therefore,
\[
\|p_{3,3} - \tsum_k f''(W_k) c_k v_k Q_k^{\rm (I)} \|_{L^2}
\lesssim t^{-\frac52}q +t^{-2} q^2.
\]
Thus, we have proved
\begin{equation*}
|\bh_{6,1}|\lesssim \left(|A| t^{-2} e^{-\beta_m}   + q^\frac73+t^{-\frac 23} q^2 + t^{-2} q\right)\|\vec a\|_\cE,
\end{equation*}
and
\begin{equation*}\left|\bh_{6,1} - \sum_k \int (\ell_k\partial_1 a+b) f''(W_k) c_k v_k Q_k^{\rm(I)}\right| 
 \lesssim \left(|A| t^{-2} e^{-\beta_m}   + q^\frac73+t^{-\frac 23} q^2 + t^{-\frac52}q\right)
\|\vec a\|_\cE.
\end{equation*}

We now estimate the second term $\bh_{6,2}$ in $\bh_6$. First, we note that
\begin{equation*}
\left|\bh_{6,2}\right|
\lesssim t^{-1} \|a\|_{L^\frac{10}3} \left(\|P_2\|_{L^\frac{10}7(\Omega)}+\|AM_F\|_{L^\frac{10}7(\Omega)}+\|M_R\|_{L^\frac{10}7(\Omega)}\right).
\end{equation*}
Moreover, using \eqref{eq:dZ}, \eqref{eq:C3} and \eqref{eq:ke}
\begin{align*}
\|P_2\|_{L^\frac{10}7(\Omega)} & \lesssim \biggl(\sum_k\biggl|\frac{d}{dt}\alpha_k^\Lambda\biggr|
+\biggl|\frac{d}{dt}\alpha_k^{(j)}\biggr|\biggr)\sum_k \|\omega_k^4\|_{L^\frac{10}7(\Omega)}
\lesssim t^{-\frac12} \left(\|\vec a\|_\cE+t^{-2} q + |A| t^{-2} e^{-\beta_m}\right),
\\
\|AM_F\|_{L^\frac{10}7(\Omega)} & \lesssim |A| e^{-\beta_m}\left(|\dot\lambda_m|+|\dot\by_m|\right)
\|\zeta_m\|_{L^\frac{10}7(\Omega)} \lesssim |A| t^{-10} e^{-\beta_m},\\
\|M_R\|_{L^\frac{10}7(\Omega)} & \lesssim \sum_k\left(|\dot\lambda_k|+|\dot\by_k|\right)
\biggl(\sum_k\left|\alpha_k^\Lambda\right|+\left|\alpha_k^{(j)}\right|\biggr)\sum_k\|\omega_k^4\|_{L^\frac{10}7(\Omega)}
\lesssim t^{-\frac 52} q .
\end{align*}
Therefore, we obtain
\begin{equation*}
\left|\bh_{6,2}\right|\lesssim \left|\frac1{t} \int_\Omega a (P_2-AM_F-M_R)\right|
\lesssim t^{-\frac32} \|\vec a\|_\cE\left(\|\vec a\|_\cE+t^{-2} q + |A| t^{-2} e^{-\beta_m}\right).
\end{equation*}
Combining the estimates above, we have proved \eqref{eq:h6} and \eqref{eq:h6b}.

Gathering the estimates \eqref{eq:h1}, \eqref{eq:h2}, \eqref{eq:h3}, \eqref{eq:h4}, \eqref{eq:h5} and \eqref{eq:h6}, we find
\[
\frac{d\mathcal H}{dt} 
\geq -\frac{1+C\sigma}{t} \Nint 
- C \left(\|\vec a\|_\cE^2+|A|t^{-2} e^{-\beta_m} + q^\frac73 + t^{-\frac23} q + q \|\vec a\|_\cE + t^{-\frac32} \|\vec a\|_\cE +t^{-2} q\right)\|\vec a\|_\cE
\]
Therefore, using \eqref{eq:cH}, and taking $\sigma$ small enough, depending on $\delta$,
\begin{align*}
-\frac{d\mathcal H}{dt} & \leq \frac{1+\delta}{t} \mathcal H 
+ C \left(\|\vec a\|_\cE^2+|A|t^{-2} e^{-\beta_m} + q^\frac73 + t^{-\frac23} q + q \|\vec a\|_\cE   + t^{-\frac32} \|\vec a\|_\cE  +t^{-2} q\right)\|\vec a\|_\cE\\
&\quad + \frac Ct B^2,
\end{align*}
which is rewritten as \eqref{eq:di}.

Proceeding similarly, but replacing \eqref{eq:h4} and \eqref{eq:h6} by the refined estimates \eqref{eq:h4b} and \eqref{eq:h6b}, we obtain \eqref{eq:ri}.
\end{proof}

\section{Construction of a family of strong multi-solitons}\label{S:5}

In this section, we prove the existence result of Theorem \ref{th:1}, following closely \cite{Co}
and using the framework of Section \ref{S:4}.

\begin{proposition}\label{pr:31}
Let $\MS$ be a strong multi-soliton of \eqref{wave} on $[T,+\infty)$ in the sense of Definition \ref{SMS}.
Let $\mm\in \{1,\ldots,K\}$ and
\begin{equation}\label{eq:Em}
\vec E_\mm(t,x) = e^{-\beta_m(t)} \vec {\theta}_\mm J \vec Z_{\ell_\mm}^+(t,x)
= e^{-\beta_m(t)} J \vec Z_k^+(t,x)
\end{equation}
where
\begin{equation}\label{eq:mk}
\mu_m(t) = \frac{\sqrt{\lambda_0}}{\lambda_m(t)} (1- \ell_m^2)^\frac12,\quad 
\beta_m(t) = \int_T^t \mu_m(s) ds.
\end{equation}
Then, for any $A\in \R$, there exist $t_1\geq T$ and a solution $u$ of \eqref{wave} defined on $[t_1,+\infty)$ 
 such that for all $t\geq t_1$,
\begin{equation}\label{eq:31a}
\left\|\vec u(t) - \vec \MS(t) -A \vec E_\mm(t)
\right\|_{\dot H^1\times L^2}\lesssim |A| t^{-1} e^{-\beta_m(t)}.
\end{equation}
Moreover, the solution $u$ is a strong multi-soliton
and it satisfies \emph{(ii)-(iii)} of Definition \ref{SMS}
with the same functions $\lambda_k$, $\by_k$ than for the multi-soliton $\MS$.
\end{proposition}

Before proving Proposition \ref{pr:31}, we justify that it implies Theorem \ref{th:1}.

\begin{proof}[Proof of Theorem \ref{th:1}, assuming Proposition \ref{pr:31}]

Fix $\MS$ a strong multi-soliton of \eqref{wave} and let
${\bf A}=(A_1,\ldots,A_K)\in\R^K$.

\emph{Construction of $\MS_{\bf A}$ by induction.}
By Proposition \ref{pr:31} applied to 
$\MS$, $m=1$ and $A=A_1$, there exist $t_1\geq T$ and a solution $\vec S_{A_1}$ of \eqref{wave} on $[t_1,+\infty)$ such that for all $t\geq t_1$,
\begin{equation}\label{eq:A1}
\|\vec \MS_{A_1}(t) - \vec \MS(t) - A_1 \vec E_1(t)\|_{\cE} \lesssim |A_1| t^{-1} e^{-\beta_1}.
\end{equation}
Moreover, $\MS_{A_1}$ is a strong multi-soliton.

Now, we apply  Proposition \ref{pr:31} to the multi-soliton $\MS_{A_1}$, $m=2$ and $A=A_2$.
It follows that there exist $t_2\geq t_1$ and a solution $\MS_{A_1,A_2}$ of \eqref{wave} on $[t_2,+\infty)$ such that for all $t\geq t_2$,
\[
\|\vec \MS_{A_1,A_2}(t) - \vec \MS_{A_1}(t) - A_2 \vec E_2(t)\|_{\cE} \lesssim |A_2| t^{-1} e^{-\beta_2}.
\]

By induction, for any $k=2,\ldots,K$, we obtain $t_k\geq t_{k-1}\geq\cdots\geq T$ and a solution $\MS_{A_1,\ldots,A_k}$ of \eqref{wave} on $[t_k,+\infty)$ such that for all $t\geq t_k$,
\[
\|\vec \MS_{A_1,\ldots,A_k}(t) - \vec \MS_{A_1,\ldots,A_{k-1}}(t) - A_k \vec E_k(t)\|_{\cE}
\lesssim |A_k| t^{-1} e^{-\beta_k}.
\]
In particular, the solution $\vec\MS_{\bf A}=\vec\MS_{A_1,\ldots,A_K}$ satisfies (ii)-(iii) of Definition \ref{SMS}.

At each step of the above construction, \emph{i.e.} for every choice of the parameter
$(A_1,\ldots,A_k)$, for $k\in \{1,\ldots,K\}$, we have fixed one of the solution $\varphi_{(A_1,\ldots,A_k)}$
provided by Proposition \ref{pr:31}.

\emph{Injectivity of ${\mathbf A}\mapsto \vec\MS_{\mathbf A}$.}
Let $\tilde{\bf A}=(\tilde A_1,\ldots,\tilde A_K )\in \R^K$ be such that $\tilde{\bf A}\neq {\bf A}$.
For the sake of contradiction, assume that $\vec\MS_{\tilde {\bf A}}=\vec\MS_{\bf A}$.
Denote
\[
k_0 = \min\left\{ k\in \{1,\ldots, K\} \mbox{ such that } \tilde A_{k_0}\neq A_{k_0}\right\}.
\]
From the steps of the construction of the solution $\vec\MS_{ A_1,\ldots, A_K}$, we observe that
\begin{align*}
\vec\MS_{\bf A}
& = \vec\MS_{A_1,\ldots,A_{K-1}} + A_K \vec E_K + \vec a_K\\
& = \vec\MS_{A_1,\ldots,A_{K-2}} + A_{K-1} \vec E_{K-1} + A_K \vec E_K + \vec a_{K-1} + \vec a_K\\
& = \cdots
= \vec\MS_{A_1,\ldots,A_{k_0}} + A_{k_0} \vec E_{k_0} +\sum_{k=k_0+1}^K A_k \vec E_k 
+ \sum_{k=k_0}^K \vec a_k,
\end{align*}
where the functions $\vec a_k$ satisfy $\|\vec a_k\|_{\cE}\lesssim t^{-1}e^{-\beta_k}$
for all $t\geq t_K$.
The same identity holds for the solution $\vec\MS_{\tilde A_1,\ldots,\tilde A_K}$, 
\[
\vec\MS_{\tilde {\bf A}}
=\vec\MS_{\tilde A_1,\ldots,\tilde A_{k_0-1}} + \tilde A_{k_0} \vec E_{k_0} +\sum_{k=k_0+1}^K \tilde A_k \vec E_k 
+ \sum_{k=k_0}^K \vec{\tilde a}_k.
\]
(By convention, if $k_0=1$, then $\vec\MS_{\tilde A_1,\ldots,\tilde A_{k_0-1}}=\vec \MS$.)
By the definition of $k_0$, we have $\tilde A_{k_0} \neq A_{k_0}$ but $(A_1,\ldots,A_{k_0-1})=(\tilde A_1,\ldots,\tilde A_{k_0-1})$,
 and thus $\vec\MS_{A_1,\ldots,A_{k_0-1}}=\vec\MS_{\tilde A_1,\ldots,\tilde A_{k_0-1}}$.
Since we have assumed that $\vec\MS_{\tilde {\bf A}}=\vec\MS_{\bf A}$, we obtain
\begin{equation}\label{eq:pr}
(A_{k_0} -\tilde A_{k_0}) \vec E_{k_0}
= \sum_{k=k_0+1}^K (\tilde A_k-A_k) \vec E_k + \sum_{k=k_0}^K (\vec a_k-\vec{\tilde a}_k).
\end{equation}
Projecting \eqref{eq:pr} on $\vec E_{k_0}$, and using $|\langle \vec E_{k_0},\vec E_k\rangle|\lesssim t^{-10}e^{-2\beta_{k_0}}$,
for $k\geq k_0+1$, we find
\[
|A_{k_0} -\tilde A_{k_0}|\lesssim t^{-1},
\]
for all $t\geq t_K$. This implies $A_{k_0} =\tilde A_{k_0}$, a contradiction with the definition of $k_0$.
\end{proof}

Now, the rest of this section is devoted to the proof of Proposition \ref{pr:31}.
The proof uses a compactness argument based on uniform estimates.

\begin{proposition}\label{pr:33}
Let $\vec \MS$ be a strong multi-soliton on some time interval $[T,+\infty)$ 
in the sense of Definition \ref{SMS}.
Let $m\in \{1,\ldots,K\}$ and $A\in \R$.
Let $T_n=n$.
There exist $n_0\geq 1$, $t_0>0$, and for any $n\geq n_0$ there exist reals $(h_{n,k}^-)_{k\in J_m}\in \R^{K_m}$,
$(h_{n,k}^\Lambda)_{k}\in \R^K$, $(h_{n,k}^{(j)})_{j,k}\in (\R^{5})^K$ such that
if $\vec u_n$ is the solution of \eqref{wave} with
\begin{equation}\label{eq:id}
\vec u_n (T_n) = \vec\MS(T_n) + A \vec E_\mm(T_n)
+ \sum_{k\in J_m} h_{n,k}^- \vec E_k (T_n)
+ \sum_k h_{n,k}^\Lambda \vec Z_{k}^\Lambda + \sum_{j,k} h_{n,k}^{(j)} \vec Z_{k}^{(j)}
\end{equation}
then
\begin{equation}\label{eq:ue}
\forall t\in [t_0,T_n],\quad
\left\|\vec u_n(t) - \vec \MS(t) -A \vec E_\mm(t)
\right\|_{\dot H^1\times L^2} \lesssim |A| t^{-1}e^{-\beta_m}.
\end{equation}
\end{proposition}

We prove Proposition \ref{pr:33} in the next subsections.
We assume $A\neq 0$, since otherwise the statement is obvious with $\vec u_n=\vec S$.
We consider the solution $\vec u_n$ of \eqref{wave} defined by its initial data at $T_n$ given in \eqref{eq:id},
where $(h_k^-)_{k\in J_m}$, $(h_k^\Lambda)_{k}$, $(h_k^{(j)})_{j,k}$ are free parameters to be fixed.

\subsection{Initialisation}
Let 
\[
\ba^+=(\alpha_k^+)_{1\leq k\leq K},\quad
\ba_{I}^- = (\alpha_k^-)_{k\in I_m},\quad 
\ba_J^- = (\alpha_k^-)_{k\in J_m}.
\]
We denote by $K_m$ the cardinality of the set $J_m$.
In the case where $J_m$ is empty, the part of the proof concerning $\ba_J^-$ is simply to be ignored.

\begin{lemma}\label{le:fd}
For $n$ large enough, for all $\ba_{in}^-\in \R^{K_m}$, there exist unique $(h_{n,k}^-)_{k\in J_m}\in \R^{K_m}$,
$(h_{n,k}^\Lambda)_{k}\in \R^K$, $(h_{n,k}^{(j)})_{j,k}\in (\R^{5})^K$ such that 
\begin{equation}\label{eq:bi}
\|(h_{n,k}^-)_{k\in J_m}\| \lesssim \|\ba_{in}^-\|, \quad
\|(h_{n,k}^\Lambda)_{k}\| \lesssim T_n^{-3}\|\ba_{in}^-\|, \quad
\|(h_{n,k}^{(j)})_{j,k}\| \lesssim T_n^{-3}\|\ba_{in}^-\|
\end{equation}
and, for the initial data defined in \eqref{eq:id}, it holds
\begin{equation}\label{eq:ii}
\begin{aligned}
& \ba_{J}^-(T_n) = \ba_{in}^-,\\
& \left( u_n (T_n) - \MS(T_n) - A E_\mm(T_n), \Lambda_k W_k \right)_{\dot H^1_{\ell_k}}=0,\\
& \left( u_n (T_n) - \MS(T_n) - A E_\mm(T_n), \partial_j W_k\right)_{\dot H^1_{\ell_k}}=0.
\end{aligned}
\end{equation}
Moreover, the decomposition of $\vec u(T_n)$ given by Section \ref{s:4.1} satisfies
\begin{equation}\label{eq:ci}
\| \vec a(T_n)\|_\cE \lesssim  \|\ba_{in}^-\|,\quad |\ba^+(T_n)|+|\ba_I^-(T_n)|\lesssim T_n^{-3} \|\ba_{in}^-\| ,\quad q(T_n)=0.
\end{equation}
\end{lemma}
\begin{proof}[Sketch of the proof]
The proof of this result follows from standard linear algebra and is omitted (see for instance \cite[Proof of Lemma 3.6]{Co}).
Note that by \eqref{eq:id} and \eqref{eq:ii}, we have $q=0$ and
\[
\vec a(T_n) =u_n (T_n) - \MS(T_n) - A E_\mm(T_n)
= \sum_{k\in J_m} h_{n,k}^- \vec E_k (T_n)
+ \sum_k h_{n,k}^\Lambda \vec Z_{k}^\Lambda + \sum_{j,k} h_{n,k}^{(j)} \vec Z_{k}^{(j)},
\]
which implies \eqref{eq:ci} using the estimates in \eqref{eq:bi}.
\end{proof}

\subsection{Bootstrap estimate}
Let $\ba_{in}^- \in \R^{K_m}$ to be chosen later such that 
\begin{equation}\label{eq:rb}
|\ba_{in}^-|\leq  |A| T_n^{-3 \nu_0} e^{-\beta_m(T_n)}.
\end{equation}
(Recall that $\nu_0\in (0,\frac18)$ is defined in \eqref{eq:n0}.)
We consider the solution $\vec u_n$ of \eqref{wave} corresponding to the initial data at $T_n$ defined in \eqref{eq:id} with the parameters
$(h_{n,k}^-)_{k\in J_m}\in \R^{K_m}$,
$(h_{n,k}^\Lambda)_{k}\in \R^K$, $(h_{n,k}^{(j)})_{j,k}\in (\R^{5})^K$ given by Proposition~\ref{pr:33}.
In particular, by \eqref{eq:bi} and \eqref{eq:ii}, we have
\begin{equation}\label{eq:ci2}
\begin{aligned}
& |\ba^+(T_n)|+|\ba_I^-(T_n)|\lesssim  |A| T_n^{-3-3 \nu_0} e^{-\beta_m(T_n)} ,\\
& \| \vec a(T_n)\|_\cE \lesssim |A| T_n^{-3 \nu_0} e^{-\beta_m(T_n)}, \quad q(T_n)=0, \\
& |\ba_{J}^-(T_n)|\leq |A| T_n^{-3 \nu_0} e^{-\beta_m(T_n)}.
\end{aligned}
\end{equation}
From the initial data $\vec u_n(T_n)$ at $t=T_n$, we solve equation \eqref{wave} backwards in time, and more precisely we prove that the solution $\vec u_n$
exists and satisfies uniform estimates on the interval $[t_0,T_n]$ for some $t_0$ large enough independent of $n$.
We use the decomposition of $\vec u_n$ and bootstrap estimates.

\begin{definition}
Let $T(\ba_{in}^-)$ be the infimum of $T\geq t_0$ such that, for all $t\in [T,T_n]$, the following properties hold
\begin{equation}\label{eq:BS}
\begin{aligned}
|\ba^+|+|\ba_I^-|&\leq  |A| t^{-3 \nu_0} e^{-\beta_m},\\
\|\vec a\|_\cE &\leq  |A| t^{-2 \nu_0} e^{-\beta_m},\\
|q| & \leq |A| t^{-\nu_0} e^{-\beta_m},\\
|\ba_J^-|&\leq  |A| t^{-3 \nu_0} e^{-\beta_m}.
\end{aligned}
\end{equation}
\end{definition}
\begin{remark}
Note that for simplicity, there is no bootstrap constant to be fixed in \eqref{eq:BS}. 
We will improve the bootstrap estimate by taking $T$ sufficiently large instead of adjusting large bootstrap constants.
By \eqref{eq:ci2} and by a continuity argument, we observe that if 
$T(\ba_{in}^-)=T_n$ then $|\ba_{in}^-|=|\ba_J^-(T_n)|= |A| T_n^{-3 \nu_0} e^{-\beta_m(T_n)}$
(see Section \ref{s:5.7}).
\end{remark}

We will close the bootstrap estimates in the same order as they are presented in \eqref{eq:BS}.
In the estimates below, we work on the time interval $[T(\ba_{in}^-),T_n]$,
where $T(\ba_{in}^-)\geq t_0$ is defined above.

\subsection{Control of all the stable directions}
For any $k$, using \eqref{eq:BS} and \eqref{eq:zo} we have, for $t\in [T(\ba_{in}^-),T_n]$, for $t_0$ large enough,
\begin{equation*}
\left| \frac{d}{dt} \alpha_k^+ - \mu_k \alpha_k^+ \right|
 \lesssim  |A| t^{-2-\nu_0} e^{-\beta_m} + |A| t^{-2} e^{-\beta_m} \lesssim |A| t^{-2} e^{-\beta_m}.
\end{equation*}
Thus, $|(e^{-\beta_k} \alpha_k^+)'|\lesssim |A| t^{-2} e^{-(\beta_m+\beta_k)}$. By integration
on $[t,T_n]$, and then using the bound on the initial data \eqref{eq:ci2}, we obtain 
\begin{align*}
|\alpha_k^+(t)|&\lesssim e^{\beta_k - \beta_k(T_n)} |\alpha_k^+(T_n)| + |A| t^{-2} e^{-\beta_m}\\
&\lesssim |A| T_n^{-3-3\nu_0} e^{-\beta_m(T_n)}  + |A| t^{-2} e^{-\beta_m}
\lesssim |A| t^{-2} e^{-\beta_m}.
\end{align*}

\subsection{Control of lower unstable directions}
For $k\in I_m$, using \eqref{eq:zo} and \eqref{eq:BS}, we have
\begin{equation*}
\left| \frac{d}{dt} \alpha_k^- + \mu_k \alpha_k^- \right| \lesssim |A| t^{-2}e^{-\beta_m}.
\end{equation*}
Thus $|(e^{\beta_k} \alpha_k^-)'|\lesssim |A| t^{-2} e^{\beta_k-\beta_m}$, and by integration
on $[t,T_n]$, and then \eqref{eq:ci2}, we obtain 
\[
|\alpha_k^-(t)|\lesssim e^{-\beta_k + \beta_k(T_n)} |\alpha_k^-(T_n)| + |A| t^{-1} e^{-\beta_m}
\lesssim |A| t^{-1} e^{-\beta_m}.
\]

In conclusion for $\ba^+$ and $\ba^-_I$, since $\nu_0\leq \frac 18$, we obtain for $t$ large,
\begin{equation}\label{eq:01}
|\ba^+|+|\ba_I^-|\leq C |A| t^{-1} e^{-\beta_m} \leq \frac12 |A| t^{-3 \nu_0} e^{-\beta_m},
\end{equation}
which strictly improves the first line of \eqref{eq:BS}.

\subsection{Control of the infinite dimensional part}
Using \eqref{eq:di} and \eqref{eq:BS}, we have
\[
- \frac{d}{dt}\left(t^{1+\delta} \cH\right)
\leq C |A|^2 t^\delta t^{-6\nu_0} e^{-2\beta_m}.
\]
Note that \eqref{eq:ci2} implies
\[
|\cH(T_n)| \lesssim \|\vec a(T_n)\|_\cE^2 \lesssim |A|^2 T_n^{-6 \nu_0} e^{-2\beta_m(T_n)}.
\]
Thus, by integration on $[t,T_n]$,
\[
\cH \lesssim t^{-1-\delta}T_n^{1+\delta} |\cH(T_n)| + |A|^2 t^{-1-6\nu_0} e^{-2\beta_m}
\lesssim |A|^2 t^{-6\nu_0} e^{-2\beta_m}.
\]
Using now \eqref{nint}, \eqref{eq:cH} and \eqref{eq:BS}, we obtain
\[
\|\vec a\|_\cE \lesssim |A| t^{-3\nu_0} e^{-\beta_m}.
\]
Thus, for $t$ sufficiently large,
\begin{equation}\label{eq:02}
\|\vec a\|_\cE \leq \frac 12 |A| t^{-2\nu_0} e^{-\beta_m},
\end{equation}
which strictly improves the second line of \eqref{eq:BS}.

\subsection{Control of the kernel directions}
Now, we use \eqref{eq:ke} and \eqref{eq:BS},
\[
\left| \frac{d}{dt} \alpha_k^\Lambda \right| + \left| \frac{d}{dt} \alpha_k^{(j)} \right|
\leq C |A| t^{-2\nu_0} e^{-\beta_m}.
\]
By integration and $q(T_n)=0$ (see \eqref{eq:ci2}),
\[
\left|  \alpha_k^\Lambda \right| + \left|  \alpha_k^{(j)} \right|
\leq C |A| t^{-2\nu_0} e^{-\beta_m},
\]
and so, for $t$ sufficiently large,
\begin{equation}\label{eq:03}
|q|\leq \frac 12 |A|t^{-\nu_0} e^{-\beta_m}.
\end{equation}

By \eqref{eq:01}, \eqref{eq:02} and \eqref{eq:03}, we have strictly improved all the estimates in \eqref{eq:BS},
except the one on $\ba_J^-$. If $J_m$ is  empty then the proof is complete.
Otherwise, we need to control the instability directions contained in $\ba_J^-$ 
using a topological argument.

\subsection{Conclusion by a topological argument}\label{s:5.7}
Assume that $J_m$ is not empty.
Set
\[
N = t^{6\nu_0} e^{2\beta_m} \sum_{k \in J_m}  (\alpha_k^-)^2 .
\]
Differentiating and then using \eqref{eq:mm}, \eqref{eq:zo} and \eqref{eq:BS} we find
\begin{align*}
\frac{d}{dt} N
& = 6\nu_0 t^{-1} N + 2\mu_m N
+2 t^{6\nu_0} e^{2\beta_m} \sum_{k \in J_m} \alpha_k^- \frac{d}{dt}\alpha_k^- \\
& \leq 2 \sum_{k\in J_m} \left(3\nu_0 t^{-1}+ \mu_m^\infty - \nu_m t^{-1} + \mu_k^\infty -  \nu_k t^{-1} + C t^{-\frac74}\right) (\alpha_k^-)^2
 \\
&\quad + C |A| t^{-2+6\nu_0} e^{\beta_m} \sum_{k\in J_m} |\alpha_k^-|.
\end{align*}
Now, we recall \eqref{eq:Jm} which says for $k \in J_m$, 
either $\mu_k^\infty> \mu_m^\infty$ or $\mu_k^\infty=\mu_m^\infty$ and then $\nu_k\leq \nu_m-8\nu_0$.
We also use again \eqref{eq:BS} for the last term to obtain, for $t$ large,
\begin{equation*}
\frac{d}{dt} N \leq - 8\nu_0 t^{-1} N + C A^2 t^{-2+3\nu_0}.
\end{equation*}
Now take $t_0$ large enough so that $C t_0^{-2+3\nu_0} \leq 4 \nu_0 t_0^{-1}$.
As a consequence, if $t\geq t_0$ and $N=|A|^2$ then one has
\begin{equation}\label{eq:at}
\frac{d}{dt} N (t) \leq - 4 \nu_0 t^{-1}  A^2 <0.
\end{equation}
By a standard contradiction argument involving Brouwer's  Fixed-Point Theorem (see for example \cite{CMM} or \cite[\S 3.3.7]{Co}),
we obtain the existence of at least one value of $\ba_{in}^-\in \R^{K_m}$ satisfying \eqref{eq:rb} and such that $T(\ba_{in}^-)=t_0$.

\subsection{Compactness argument}
The construction of an exact multi-soliton using a sequence of approximate solutions 
satisfying the uniform estimates of Proposition \ref{pr:33} follows the same strategy 
by compactness as in \cite[\S 5]{MMwave2}. We omit the details here.

Once the construction is performed in the space $\dot H^1 \times L^2$, the solution
is shown to be a strong multi-soliton by using Proposition \ref{pr:2}.

\section{Classification of multi-solitons}\label{S:6}

In this section, we prove Theorem \ref{th:2}, following \cite{Co} for 
the strategy of the proof and using the general setting of Section \ref{S:4} for the technical ingredients specific to the energy-critical wave equation (largely inspired from \cite{MMwave1,MMwave2} and also \cite{CMkg}).

We consider a strong multi-soliton $\vec S$
and the family of strong multi-solitons $\vec S_{\bf A}$ generated by $\vec S$ as constructed in Section \ref{S:5}.
Let $\vec u$ be a multi-soliton of \eqref{wave} on some time interval $[T,+\infty)$ and satisfying
the estimate \eqref{eq:hy} for some $\delta>0$. 
Then, by Proposition \ref{pr:2}, we know that $u$ is a strong multi-soliton and in particular,
it satisfies
\begin{equation}\label{eq:S4}
\left\|\nabla_{t,x} \left(u -\MS\right) (t)\right\|_{L^2} \lesssim t^{-2+\delta}.
\end{equation}
We use exactly the same notation as in Section \ref{s:4.1}, in particular the decomposition of $\vec u$ in \eqref{eq:dS},
but with $A=0$. Indeed, the parameter $A$ in Section \ref{s:4.1} is relevant for the construction in Section \ref{S:5} but
it is useless in the classification proof, now that the family of multi-solitons $\vec S_{\bf A}$ has been constructed.

\subsection{Summarizing estimates}
For convenience, we gather here the estimates proved in Section \ref{S:4}
(specifically, estimates \eqref{eq:ke}, \eqref{eq:zo}, \eqref{eq:eH}, \eqref{eq:cH}, \eqref{eq:ri}), in the special case where $A=0$.
Assuming \eqref{BPweak} and $A=0$, it holds
\begin{align}\label{eq:ke0}
&\left| \frac{d}{dt} \alpha_k^\Lambda \right| + \left| \frac{d}{dt} \alpha_k^{(j)} \right|
\lesssim \|\vec a\|_{\mathcal E} + t^{-2} q ,\\
\label{eq:zo0}
&\left| \frac{d}{dt} \alpha_k^\pm \mp \mu_k \alpha_k^\pm \right| 
\lesssim \|\vec a\|_{\mathcal E}^2 + q^2 + t^{-2} (\|\vec a\|_{\mathcal E}+q) ,
\\
&| \mathcal{H}| \lesssim \|\vec a\|_\cE^2,\label{eq:HH}
\\
\label{eq:cH0}
&\mathcal{H}\geq \mu \|\vec a\|_\cE^2- C B^2 ,\\
\label{eq:ri0}
&- \frac d{dt} \left( t^{1+\delta} \mathcal H \right) + t^{1+\delta} \sum_l \br_l\\
&\quad \lesssim t^{1+\delta} \left(\|\vec a\|_\cE^2 + q^\frac73+t^{-\frac 23} q^2 +q\|\vec a\|_\cE + t^{-\frac32} \|\vec a\|  +t^{-\frac52} q\right)\|\vec a\|_\cE 
+ t^\delta B^2,\nonumber
\end{align}
where the terms $\br_l$ are defined in \eqref{eq:rr}.

Furthermore, by \eqref{eq:S4}, we have
\begin{equation}\label{eq:im}
\|\vec a\|_\cE + B + q\lesssim t^{-2+\delta}.
\end{equation}
\subsection{Correction terms for the energy variation}
We need to add correction terms to the energy to compensate exactly the sum $\sum_l \br_l$ in \eqref{eq:ri0}
(such terms are of size $t^{-2} q\|\vec a\|_\cE$).
Note that such a correction is not needed in the construction proof, since it is based on exponential bootstrap estimates.

\begin{lemma}\label{le:vG}
Assume \eqref{BPweak}.
It holds
\begin{equation}\label{eq:br}
|\br_1|+|\br_2|+|\br_3|+|\br_4|\lesssim t^{-2} q \|\vec a\|_\cE.
\end{equation}
Moreover, let
\begin{align*}
&\cG_{1,k} = - \frac{d_k}{t^2} \alpha_k^\Lambda \int(b-\ell_k \pun a) \Lambda_k^2 W_k,\quad
\cG_{2,k} = - \frac{d_k}{t^2} \lambda_k \sum_j \alpha_k^{(j)} \int(b-\ell_k \pun a) \Lambda_k\partial_j W_k,\\
&\cG_{3,k} = c_k\alpha_k^\Lambda\int a f''(W_k) v_k \Lambda_k W_k,\quad 
\cG_{4,k} = c_k \lambda_k \sum_j \alpha_k^{(j)}\int a f''(W_k) v_k \partial_j W_k . 
\end{align*}
Then,
\begin{align}
\sum_{j=1}^4 \left|\cG_{j,k}\right| & \lesssim t^{-2} q\|\vec a\|_\cE,\label{eq:eG}\\
\left| \frac d{dt} \sum_{j=1}^4\cG_{j,k} - \sum_{j=1}^4 \br_{j,k} \right| 
& \lesssim t^{-2} \|\vec a\|_\cE^2 +t^{-4} q^2+t^{-2} q^3.\label{eq:50} 
\end{align}
\end{lemma}
\begin{proof}
The proofs of \eqref{eq:br} and \eqref{eq:eG} are straightforward.
Now, we prove
\begin{align}
\left| \frac d{dt} \cG_{1,k} - \br_{1,k}
- \ell_k  c_k \kappa_{\ell_k} \epsilon_k (\lambda_k^\infty)^\frac12  t^{-2} \alpha_k^\Lambda \frac{d}{dt}\alpha_k^{(1)}\int (\partial_1 \Lambda  W_\ell )^2 \right| & \lesssim t^{-2} \|\vec a\|_\cE^2 +t^{-4} q^2, \label{eq:51} \\
\left| \frac d{dt} \cG_{2,k} - \br_{2,k} - \ell_k  c_k \kappa_{\ell_k} \epsilon_k (\lambda_k^\infty)^\frac12  t^{-2}  \alpha_k^{(1)} \frac{d}{dt}\alpha_k^\Lambda\int (\partial_1 \Lambda W_\ell)^2\right| & \lesssim t^{-2} \|\vec a\|_\cE^2 +t^{-4} q^2, \label{eq:52} \\
\left| \frac d{dt} \cG_{3,k} - \br_{3,k} 
+ \ell_k  c_k \kappa_{\ell_k} \epsilon_k(\lambda_k^\infty)^\frac12 t^{-2} \alpha_k^\Lambda \frac{d}{dt}\alpha_k^{(1)}\int (\partial_1 \Lambda  W_\ell )^2\right| & \lesssim t^{-2} \|\vec a\|_\cE^2 +t^{-4} q^2, \label{eq:53}\\
\left| \frac d{dt} \cG_{4,k} - \br_{4,k} 
+ \ell_k   c_k \kappa_{\ell_k} \epsilon_k(\lambda_k^\infty)^\frac12 t^{-2}  \alpha_k^{(1)} \frac{d}{dt}\alpha_k^\Lambda\int (\partial_1 \Lambda W_\ell)^2\right| & \lesssim t^{-2} \|\vec a\|_\cE^2 +t^{-4} q^2 \label{eq:54}. 
\end{align}
Note that summing up \eqref{eq:51}-\eqref{eq:54} implies \eqref{eq:50} after two cancellations.

First, we prove \eqref{eq:51}. We differentiate
\begin{align*}
\frac{d}{dt} \cG_{1,k} & = 
2 \frac{d_k}{t^3} \alpha_k^\Lambda \int(b-\ell_k \pun a) \Lambda_k^2 W_k
- \frac{d_k}{t^2}\frac{d}{dt }\alpha_k^\Lambda \int(b-\ell_k \pun a) \Lambda_k^2 W_k \\
&\quad - \frac{d_k}{t^2} \alpha_k^\Lambda \int (\partial_t b-\ell_k \pun \partial_t a) \Lambda_k^2 W_k
- \frac{d_k}{t^2} \alpha_k^\Lambda \int(b-\ell_k \pun a) \partial_t \Lambda_k^2 W_k.
\end{align*}
The first term is controlled as follows
\[
\left|\frac{d_k}{t^3} \alpha_k^\Lambda \int(b-\ell_k \pun a) \Lambda_k^2 W_k \right|
\lesssim t^{-3} q \|\vec a\|_\cE.
\]
The second term is controlled using \eqref{eq:ke0},
\[
\left| \frac{d_k}{t^2}\frac{d}{dt }\alpha_k^\Lambda \int(b-\ell_k \pun a) \Lambda_k^2 W_k \right|
\lesssim t^{-2}\left( \|\vec a\|_\cE + t^{-2} q \right) \|\vec a\|_\cE.
\]
For the third term, we use \eqref{eq:vv}-\eqref{eq:34} (with $A=0$), which says that
\[
\partial_t b-\ell_k \pun \partial_t a
= \Delta a + f'(S) a 
- \ell_k \pun b + p_1 a + p_2 + p_3 + \ell_k \pun P_1 - P_2 - \ell_k \pun M_Q  + M_R.
\]
A few of those terms are handled as follows (using \eqref{eq:C2})
\begin{equation*}
t^{-2} |\alpha_k^\Lambda| \int (|p_1| |a| + |p_2| + |p_3| + |\pun M_Q| + |M_R|) |\Lambda_k^2 W_k| 
\lesssim t^{-2} q \left(\|\vec a\|_\cE^2 +q^2 +t^{-2} q\right).
\end{equation*}
Moreover, we have by the definition of $\vec P$,
\[
|(\ell_k \pun P_1 - P_2) - 2\ell_k \pun P_1^{\rm (I)} |
\lesssim \sum_l\left(\left| \frac{d}{dt} \alpha_l^\Lambda\right|
+\left| \frac{d}{dt} \alpha_l^{(j)}\right|\right) \left(\sum_{l\neq k} \omega_l^4 + q\sum_{l} \omega_l^4\right)
\]
and thus, using \eqref{eq:C2},
\begin{align*}
& t^{-2} |\alpha_k^\Lambda|\int \left|(\ell_k\pun P_1 - P_2 )- 2\frac{d}{dt}\alpha_k^\Lambda\ell_k \pun \Lambda W_k
- 2 \ell_k \lambda_k{\tsum_j} \frac{d}{dt}\alpha_k^{(j)}\partial_1\partial_j W_k \right| |\Lambda^2_k W_k|\\
&\quad
\lesssim t^{-2} |\alpha_k^\Lambda| 
\left(\left| \frac{d}{dt} \alpha_l^\Lambda\right|+\left| \frac{d}{dt} \alpha_l^{(j)}\right|\right)
\left( \int \sum_{l\neq k} \omega_l^4 \omega_k^3 + q\right)
\lesssim t^{-2}q\left( \|\vec a\|_\cE + t^{-2} q\right)(t^{-2}+q).
\end{align*}
Besides, by \eqref{eq:c2},
\[
t^{-2} |\alpha_k^\Lambda| \int |f'(S)-f'(W_k)| |a| |\Lambda^2_k W_k|\lesssim t^{-4} q \|\vec a \|_\cE.
\]
Therefore,
\begin{align*}
&\biggl| \frac{d_k}{t^2} \alpha_k^\Lambda \int (\partial_t b - \ell_k \pun \partial_t a) \Lambda_k^2 W_k
-\frac{d_k}{t^2} \alpha_k^\Lambda \int (\Delta a + f'(W_k) a - \ell_k \pun b) \Lambda_k^2 W_k\\
&\quad -2\ell_k \frac{d_k}{t^2} \alpha_k^\Lambda \int \left(\frac{d}{dt}\alpha_k^\Lambda \pun \Lambda W_k
+\lambda_k{\tsum_j} \frac{d}{dt}\alpha_k^{(j)}\partial_1\partial_j W_k\right)\Lambda_k^2 W_k \biggr|
\lesssim t^{-2}  \|\vec a\|_\cE^2 + t^{-4} q^2 +t^{-2} q^3.
\end{align*}
Since $\int \pun \Lambda W_k \Lambda_k^2 W_k=0$,
$\int \partial_1^2 W_k \Lambda_k^2 W_k
=  \int (\partial_1 \Lambda_k W_k)^2$, and for $j\neq 1$, 
$\int \partial_1\partial_j W_k \Lambda_k^2 W_k=0$, we obtain
\begin{align*}
&\biggl| \frac{d_k}{t^2} \alpha_k^\Lambda \int (\partial_t b - \ell_k \pun \partial_t a) \Lambda_k^2 W_k
-\frac{d_k}{t^2} \alpha_k^\Lambda \int (\Delta a + f'(W_k) a - \ell_k \pun b) \Lambda_k^2 W_k\\
&\quad - 2\ell_k \frac{\lambda_k d_k}{t^2} \alpha_k^\Lambda \frac{d}{dt}\alpha_k^{(1)}\int (\partial_1 \Lambda_k W_k)^2 \biggr|
\lesssim t^{-2}  \|\vec a\|_\cE^2 + t^{-4} q^2 +t^{-2} q^3.
\end{align*}
Moreover, by \eqref{eq:dk} and \eqref{eq:c3},
\[
\left| \lambda_k d_k - \tfrac 12c_k (\lambda_k^\infty)^\frac12 \kappa_{\ell_k} \epsilon_k \right|
\lesssim t^{-1},
\]
and so
\begin{align*}
&\biggl| \frac{d_k}{t^2} \alpha_k^\Lambda \int (\partial_t b - \ell_k \pun \partial_t a) \Lambda_k^2 W_k
-\frac{d_k}{t^2} \alpha_k^\Lambda \int (\Delta a + f'(W_k) a - \ell_k \pun b) \Lambda_k^2 W_k\\
&\quad -  \ell_k  c_k \kappa_{\ell_k} \epsilon_k (\lambda_k^\infty)^\frac12  t^{-2} \alpha_k^\Lambda \frac{d}{dt}\alpha_k^{(1)}\int (\partial_1 \Lambda W_{\ell_k})^2 \biggr|
\lesssim t^{-2}  \|\vec a\|_\cE^2 + t^{-4} q^2 +t^{-2} q^3.
\end{align*}
For the fourth term, we use \eqref{eq:tG} 
\[
\partial_t \Lambda_k^2 W_k = -\ell_k \pun \Lambda_k^2 W_k - \frac{\dot \lambda_k}{\lambda_k} \Lambda_k^3 W_k
-\dot \by_k \cdot \nabla \Lambda_k W_k
\]
and then \eqref{eq:c3}, leading by integration by parts to
\begin{align*}
\left|\frac{d_k}{t^2} \alpha_k^\Lambda \int(b-\ell_k \pun a) \partial_t \Lambda_k^2 W_k 
-\ell_k  \frac{d_k}{t^2} \alpha_k^\Lambda \int \pun(b-\ell_k \pun a)  \Lambda_k^2 W_k \right| \lesssim 
t^{-4} q \|\vec a\|_\cE.
\end{align*}
Thus, gathering the above estimates and using \eqref{eq:im}, \eqref{eq:dk} we have proved \eqref{eq:51}.

The proof of \eqref{eq:52} is similar and omitted.

Now, we prove \eqref{eq:53}. Differentiate
\begin{align*}
\frac{d}{dt} \cG_{3,k} & = 
 c_k \frac{d}{dt} \alpha_k^\Lambda\int a f''(W_k) v_k \Lambda_k W_k 
+  c_k\alpha_k^\Lambda\int \partial_t a f''(W_k) v_k \Lambda_k W_k \\
&\quad +   c_k\alpha_k^\Lambda\int a \partial_t (f''(W_k) v_k \Lambda_k W_k ).
\end{align*}
The first term is controlled using \eqref{eq:vk} and \eqref{eq:ke}
\[
\left| \frac{d}{dt} \alpha_k^\Lambda\int a f''(W_k) v_k \Lambda_k W_k \right| 
\lesssim t^{-2} \left( \|\vec a\|_\cE + t^{-2} q \right) \|\vec a\|_\cE.
\]
For the second term, we use the first line of \eqref{eq:vv} (with $A=0$) so that
\begin{align*}
  c_k\alpha_k^\Lambda\int \partial_t a f''(W_k) v_k \Lambda_k W_k
& =   c_k\alpha_k^\Lambda\int b  f''(W_k) v_k \Lambda_k W_k \\
& \quad +   c_k\alpha_k^\Lambda\int (-P_1+M_Q)  f''(W_k) v_k \Lambda_k W_k 
\end{align*}
where using \eqref{eq:vk} and \eqref{eq:c3},
\[
\left| \alpha_k^\Lambda\int  M_Q   f''(W_k) v_k \Lambda_k W_k \right| 
\lesssim t^{-4} q^2 .
\]
Using \eqref{eq:tG}, \eqref{eq:vk} and \eqref{eq:c3}, we have for the last term
\[
\left| \alpha_k^\Lambda\int a \partial_t (f''(W_k) v_k \Lambda_k W_k )
+ \alpha_k^\Lambda\int a  \ell_k \pun (f''(W_k) v_k \Lambda_k W_k ) \right|
\lesssim t^{-4} q \|\vec a\|_\cE.
\]
Thus,
\begin{equation*}
\left| \frac{d}{dt} \cG_{3,k} -\br_{3,k}
+  c_k\alpha_k^\Lambda\int P_1 f''(W_k) v_k \Lambda_k W_k\right|
 \lesssim t^{-2} \left( \|\vec a\|_\cE + t^{-2} q \right) \|\vec a\|_\cE.
\end{equation*}
Lastly, we treat the term $c_k\alpha_k^\Lambda\int P_1 f''(W_k) v_k \Lambda_k W_k$ above.
By the expression of $P_1=P_1^{\rm(I)}+P_1^{\rm(II)}$, we have
\begin{equation*}
\left| c_k\alpha_k^\Lambda\int P_1 f''(W_k) v_k \Lambda_k W_k
- c_k\alpha_k^\Lambda\int P_1^{\rm (I)} f''(W_k) v_k \Lambda_k W_k\right|
\lesssim t^{-2} q^2 (\|\vec a\|_\cE+t^{-2} q).
\end{equation*}
Thus, we are reduced to estimate the term
\begin{align*}
 c_k\alpha_k^\Lambda\int P_1^{\rm (I)} f''(W_k) v_k \Lambda_k W_k
&  = c_k \alpha_k^\Lambda \frac{d}{dt} \alpha_k^\Lambda \int \Lambda_k W_k f''(W_k) v_k \Lambda_k W_k \\
& \quad + c_k \alpha_k^\Lambda \lambda_k \frac{d}{dt} \alpha_k^{(j)} \int \partial_1 W_k f''(W_k) v_k \Lambda_k W_k.
\end{align*}
We estimate first the second term in the second member of \eqref{eq:tr},
By change of variable and 
setting $V_{\ell}(t,x)=v_{\ell}(t,x+\ell e_1 t)$ (see Remark \ref{rk:pr}), we have 
\begin{equation}\label{eq:tr}
\lambda_k \int (\partial_1 W_k f''(W_k) v_k \Lambda_k W_k)(t)
=\epsilon_k \lambda_k^{-\frac 32}
\int (\partial_1 W_{\ell_k} f''(W_{\ell_k}) V_{\ell_k} \Lambda W_{\ell_k})\left(\frac{t}{\lambda_k}\right).
\end{equation}
We recall from \eqref{H2} that
\[
-\Delta_\ell \pun \Lambda W_\ell - f'(W_\ell) \pun \Lambda W_\ell
=f''(W_\ell) \pun W_\ell \Lambda W_\ell.
\]
Thus, from \eqref{eq:Vl}
\begin{align*}
\int \partial_1 W_{\ell_k} f''(W_{\ell_k}) V_{\ell_k}\left(\frac{t}{\lambda_k}\right) \Lambda W_{\ell_k}
&=\int ( -\Delta_{\ell_k}\pun \Lambda W_{\ell_k} - f'(W_{\ell_k}) \pun \Lambda W_{\ell_k}) V_{\ell_k}\left(\frac{t}{\lambda_k}\right)\\
&=\left(\frac{t}{\lambda_k}\right)^{-2}  \kappa_{\ell_k} \ell_k \int (\partial_1 \Lambda W_{\ell_k})^2
+O(t^{-3}).
\end{align*}
Similarly, since
\begin{align*}
\int \Lambda W_{\ell_k} f''(W_{\ell_k}) V_{\ell_k}\left(\frac{t}{\lambda_k}\right) \Lambda W_{\ell_k}
&=\int ( -\Delta_{\ell_k}  \Lambda^2 W_{\ell_k} - f'(W_{\ell_k}) \Lambda^2 W_{\ell_k}) V_{\ell_k}\left(\frac{t}{\lambda_k}\right) \\
&=\left(\frac{t}{\lambda_k}\right)^{-2}  \kappa_{\ell_k} {\ell_k} \int \Lambda^2 W_{\ell_k} \partial_1 \Lambda W_{\ell_k}
+O(t^{-3}) = O(t^{-3}),
\end{align*}
by symmetry, we have for the first term in the second member of \eqref{eq:tr},
\begin{equation*}
\left| \alpha_k^\Lambda \frac{d}{dt} \alpha_k^\Lambda \int \Lambda_k W_k f''(W_k) v_k \Lambda_k W_k\right|
\lesssim t^{-3} q (\|\vec a\|_\cE+t^{-2} q).
\end{equation*}
We pass from $\lambda_k$ to $\lambda_k^\infty$ using \eqref{eq:c3} .
This finishes the proof of \eqref{eq:53}.
The proof of \eqref{eq:54} is similar.
\end{proof}

Let
\begin{equation}\label{eq:KK}
\cK = \cH - \cG, \quad \cG = \sum_{j,k} \cG_{j,k}.
\end{equation}
As a consequence of \eqref{eq:ri0} and Lemma \ref{le:vG}, we obtain the following properties for the modified energy $\mathcal K$.
\begin{lemma}
Assuming \eqref{eq:im}, it holds
\begin{equation}\label{eq:rj}
- \frac d{dt} \left( t^{1+\delta} \mathcal K \right)  
\lesssim t^{1+\delta}  \left(t^{-\frac32} \|\vec a\|_\cE^2
+ t^{-\frac 72} q^2\right)
+ t^\delta \sum_k B_k^2.
\end{equation}
Moreover, for constants $C_0,C_1>0$,
\begin{equation}\label{eq:CO}
\|\vec a\|_\cE^2 \leq C_1 \mathcal K + C_0 t^{-4} q^2+ C_0 t^{-2} q^3 + C_0 B^2.
\end{equation}
\end{lemma}
\begin{proof}
Assuming \ref{BPweak} for now, and
combining \eqref{eq:ri0} and \eqref{eq:eG}, \eqref{eq:50} in Lemma \ref{le:vG}, we obtain
\begin{align*}
- \frac d{dt} \left( t^{1+\delta} \mathcal K \right)
& = - \frac d{dt} \left( t^{1+\delta} \mathcal H \right)
+t^{1+\delta} \frac d{dt}  \mathcal G  + (1+\delta) t^\delta \mathcal G\\
& \lesssim t^{1+\delta} \left[\left(\|\vec a\|_\cE^2    + q^\frac73+t^{-\frac 23} q^2 +q\|\vec a\|_\cE + t^{-\frac32} \|\vec a\|   +t^{-\frac52} q\right)\|\vec a\|_\cE 
+ t^{-4} q^2+t^{-2} q^3\right] \\ & \quad + t^\delta B^2.
\end{align*}
Now, we use the stronger estimate \eqref{eq:im}.
In particular, we get
\[
\|\vec a\|_\cE^2   + q^\frac73+t^{-\frac 23} q^2 +q\|\vec a\|_\cE + t^{-\frac32} \|\vec a\|   +t^{-\frac52} q
+t^{-2} q^3
\lesssim t^{-\frac32} \|\vec a\|   +t^{-\frac52} q,
\]
and thus
\begin{align*}
- \frac d{dt} \left( t^{1+\delta} \mathcal K \right)  
& \lesssim t^{1+\delta} \left( t^{-\frac32} \|\vec a\|_\cE^2 + t^{-\frac 52} q \|\vec a\|_\cE
+ t^{-\frac72} q^2 \right) + t^\delta B^2\\
& \lesssim t^{1+\delta} \left( t^{-\frac32} \|\vec a\|_\cE^2
+ t^{-\frac 72} q^2 \right) + t^\delta B^2,
\end{align*}
which proves the estimate \eqref{eq:rj}.

By \eqref{eq:cH0} and then \eqref{eq:eG}, we have
\[
\|\vec a\|_\cE^2 \leq C_1 \cH + C_2 B^2
\leq C_1 \cK + C_1 \cG + C_2 B^2
\leq C_1 \cK + C_3 t^{-2} q \|\vec a\|_\cE + C_2B^2.
\]
The estimate \eqref{eq:CO} follows.
\end{proof}

\subsection{Technical lemmas}
The next two lemmas are independent of what precedes, except that the notation is similar to facilitate comprehension.
Moreover, we consider given functions $\mu_k$ as in Section \ref{S:4.1} satisfying \eqref{eq:er},
and we set $\beta_k(t)=\int_T^t \mu_k(s) ds$.
\begin{lemma}\label{le:td}
Let $K\geq 1$.
Assume that there exist constants $\delta\in(0,\frac 14)$, $C_0,C_1>0$, $T\geq 1$ and differentiable functions 
\begin{align*}
N : [T,+\infty) & \to [0,+\infty),\\
\cK : [T,+\infty) & \to \R,\\
\alpha_k^\pm:[T,+\infty) & \to \R, \ k\in \{1,\ldots,K\},\\
\alpha_k^\Lambda:[T,+\infty) & \to \R, \ k\in \{1,\ldots,K\},\\
\alpha_k^{(j)}:[T,+\infty) & \to \R^j, \ k\in \{1,\ldots,K\},
\end{align*}
such that for $k\in \{1,\ldots,K\}$, for all $t\geq T$,
\begin{align}
N^2 & \leq C_1 \cK + C_0 t^{-4} q^2 + C_0 B^2  , \label{eq:z2} \\
\cK & \leq C_0 ( N^2 + t^{-2} qN ),\label{eq:z8}\\
\left| \frac d{dt} \alpha_k^{-} + \mu_k \alpha_k^{-} \right|
&\leq C_0 t^{-2} \left(  N +  q  \right),\label{eq:z4} \\
\left| \frac d{dt} \alpha_k^{+} - \mu_k \alpha_k^{+} \right|
&\leq C_0 t^{-2} \left(   N +  q \right),\label{eq:z5} \\
\left| \frac d{dt} \alpha_k^{\Lambda}\right| + \left| \frac d{dt} \alpha_k^{(j)} \right|
&\leq C_0 \left(  N + t^{-2} q \right),\label{eq:z9} \\
- \frac d{dt} \cK & \leq (1+\delta) t^{-1} \cK
+  C_0   t^{-\frac 72} q^2
  + C_0 t^{-1} B^2, \label{eq:z6}
\end{align}
where
\[
q = \sum_k |\alpha_k^\Lambda|+\sum_{j,k} |\alpha_k^{(j)}|,\quad
B= \sum_k |\alpha_k^-| + |\alpha_k^+|.
\]
If, in addition, for all $t\geq T$,
\begin{equation}\label{eq:z1}
N^2 + B^2 + q^2  \leq C_0 t^{-3},
\end{equation}
then, there exists $C>0$ such that for all $t\geq T$,
\begin{equation}\label{eq:z7}
N^2+B^2+q^2\leq C e^{-\beta_1}.
\end{equation}
\end{lemma}
\begin{proof}
Set
\[
v_0(t) = \sup_{s\geq t} \left( s^{1+2\delta} \left(N(s)+s^{-2}q(s)\right)\right).
\]
By integrating \eqref{eq:z9} on $[s,+\infty)$, for $s>T$, and using \eqref{eq:z1}, we obtain
\[
|\alpha_k^\Lambda(s)| + |\alpha_k^{(j)}(s)|
\lesssim \int_s^{+\infty} (N(\tau)+\tau^{-2}q(\tau)) d\tau \leq \int_s^{+\infty} \tau^{-1-2\delta} v_0(\tau)  d\tau
\lesssim s^{-2\delta} v_0(s).
\]
Thus,
\begin{equation}\label{eq:U3}
q(s) \lesssim s^{-2\delta} v_0(s).
\end{equation}
By \eqref{eq:U3}, we also have
\[
s^{-2+4\delta} q^2(s) \lesssim s^{-2} v_0^2(s)
\]
and thus
\begin{equation}\label{eq:qQ}
\sup_{s\geq t} \left(s^{-2+4\delta} q^2(s) \right)
\lesssim t^{-2} v_0^2(t).
\end{equation}
The equation \eqref{eq:z4} of $\alpha_k^-$ implies
\[
\left| \frac d{dt} \left( e^{\beta_k} \alpha_k^- \right) \right| \lesssim t^{-2} e^{\beta_k} (N+q).
\]
By integrating on $[T,s]$, for $s>T$, denoting $C_k(T)=e^{\beta_k(T)} |\alpha_k^-(T)|$, we get
\[
|\alpha_k^-(s)| \lesssim C_k(T) e^{-\beta_k(s)} +  e^{-\beta_k(s)} \int_T^s \tau^{-2} e^{\beta_k(\tau)}
(N(\tau)+q(\tau)) d\tau
\]
and thus, using the definition of $v_0$ and \eqref{eq:U3}, we obtain
\[
|\alpha_k^-(s)| \lesssim C_k(T) e^{-\beta_k(s)} +  e^{-\beta_k(s)} \int_T^s \tau^{-2-2\delta} e^{\beta_k(\tau)}
v_0(\tau) d\tau.
\]
Let
\[
f_0(t) = \sup_{T\leq \tau \leq t} \left(\tau^{-2-4\delta} e^{\beta_1(\tau)} v_0^2 (\tau)\right).
\]
Let $T\leq t\leq s$. Using the Chasles relation
\begin{align*}
(\alpha_k^-(s))^2
& \lesssim C_k^2(T) e^{-2\beta_k(s)} +  e^{-2\beta_k(s)} \left(\int_T^t \tau^{-2-2\delta} e^{\beta_k(\tau)}
v_0(\tau) d\tau\right)^2\\
&\quad +e^{-2\beta_k(s)} \left(\int_t^s \tau^{-2-2\delta} e^{\beta_k(\tau)}v_0(\tau) d\tau\right)^2\\
& \lesssim C_k^2(T) e^{-2\beta_k(s)} 
+ e^{-2\beta_k(s)} f_0(t) \left(\int_T^t \tau^{-1} e^{\beta_k(\tau)-\frac 12 \beta_1(\tau)}d\tau\right)^2
+ s^{-4-4\delta} v_0^2(t).
\end{align*}
Thus, for $T\leq t\leq s$,
\begin{equation}\label{eq:U1}
(\alpha_k^-(s))^2
 \lesssim C_k^2(T) e^{-2\beta_k(s)}
 + T^{-2} e^{-\beta_1(s)} f_0(t)+ s^{-4-4\delta}  v_0^2(t) .
\end{equation}
Using the equation of $\alpha_k^+$
\[
\left| \frac d{dt} \left( e^{-\beta_k} \alpha_k^+ \right) \right| \leq t^{-2} e^{-\beta_k} (N+q).
\]
By integration on $[s,+\infty)$ for $s>T$, using \eqref{eq:z1}, we obtain
\[
|\alpha_k^+(s)| \leq  e^{\beta_k(s)} \int_s^{+\infty} \tau^{-2} e^{-\beta_k(\tau)}
(N(\tau)+q(\tau)) d\tau
\]
and, using the definition of $v_0$ and \eqref{eq:U3}
\begin{equation}\label{eq:U2}
|\alpha_k^+(s)| \leq  e^{\beta_k(s)} \int_s^{+\infty} \tau^{-2-2\delta} e^{-\beta_k(\tau)}
v_0(\tau) d\tau
\lesssim s^{-2-2\delta} v_0(s).
\end{equation}

Let $T\leq t\leq s$.
By \eqref{eq:U1} and \eqref{eq:U2}, one has
\begin{equation}\label{eq:U5}
B^2(s)\lesssim 
\sum_k C_k^2(T) e^{-2\beta_k(s)}
 + T^{-2}  e^{-\beta_1(s)} f_0(t) + s^{-4-4\delta}  v_0^2(t) .
\end{equation}
We insert estimate \eqref{eq:U3}  into \eqref{eq:z6}, to obtain
\[
- \frac d{dt}(t^{1+\delta} \cK) 
\lesssim t^{-\frac 52+\delta} q^2 + t^{\delta} B^2 
\lesssim t^{-\frac52-3 \delta} v_0^2 + t^{\delta} B^2.
\]
By integration on $[t,+\infty)$, for $t\geq T$, using \eqref{eq:z1}, we find
\[
t^{1+\delta}\cK(t)
\lesssim t^{-\frac32-3\delta} v_0^2(t) +\int_t^{+\infty} s^{\delta} B^2(s) ds.
\]
Multiplying by $t^{1+3\delta}$, we find
\[
t^{2+4\delta}\cK(t) \lesssim t^{-\frac12} v_0^2(t) +\int_t^{+\infty} s^{1+4\delta} B^2(s) ds,
\]
and taking the sup over $[t,+\infty)$, we find
\[
\sup_{s\geq t} \left( s^{2+4\delta}\cK(s)\right)
\lesssim t^{-\frac12} v_0^2(t)
+\int_t^{+\infty} s^{1+4\delta} B^2(s) ds.
\]
Using \eqref{eq:U5}, we obtain
\begin{align*}
\sup_{s\geq t} \left( s^{2+4\delta}\cK(s)\right)
&\lesssim t^{-\frac12} v_0^2(t) +  \sum_k C_k^2(T) \int_t^{+\infty} s^{1+4\delta} e^{-2\beta_k(s)} ds \\
&\quad 
+ T^{-2} f_0(t)\int_t^{+\infty} s^{1+4\delta} e^{-\beta_1(s)} ds 
+ v_0^2(t) \int_t^{+\infty} s^{-3}  ds.
\end{align*}
We simplify
\begin{equation}\label{eq:U4}
\sup_{s\geq t} \left( s^{2+4\delta}\cK(s)\right)
\lesssim t^{-\frac12} v_0^2(t) + t^{1+4\delta} \sum_k C_k^2(T)e^{-2\beta_k(t)}
+ T^{-2} t^{1+4\delta}  f_0(t) e^{-\beta_1(s)}.
\end{equation}
Now, by \eqref{eq:z2} and \eqref{eq:qQ}, \eqref{eq:U5}, \eqref{eq:U4},
\[
\sup_{s\geq t} \left(s^{2+4\delta} N^2(s) \right)
\lesssim t^{-\frac12} v_0^2(t) + t^{2+4\delta} \sum_k C_k^2(T)e^{-2\beta_k(t)}
+T^{-2} t^{2+4\delta}  f_0(t) e^{-\beta_1(t)}.
\]
Therefore, by the definition of $v_0$ and \eqref{eq:qQ}, we obtain
\[
v_0^2(t) \lesssim 
t^{-\frac12} v_0^2(t) + t^{2+4\delta} \sum_k C_k^2(T)e^{-2\beta_k(t)}
+ T^{-2}t^{2+4\delta}  f_0(t) e^{-\beta_1(t)},
\]
and so, for $T$ sufficiently large
\[
v_0^2(t) \lesssim t^{2+4\delta} \sum_k C_k^2(T)e^{-2\beta_k(t)}
+ T^{-2} t^{2+4\delta}  f_0(t) e^{-\beta_1(t)}.
\]
Recalling now the definition of $f_0$, we obtain
\[
f_0(t) \lesssim \sup_{T\leq \tau \leq t} \left(\sum_k C_k^2(T)e^{-\beta_k(\tau)}
+T^{-2} f_0(\tau)\right)
\lesssim 1 + T^{-2}f_0(t).
\]
Therefore, for $T>0$ large enough, we have $f_0(t)\lesssim 1$.
It follows that
\[
  N^2(t) +  t^{-4} q^2(t) \lesssim 
t^{-2-4\delta} v_0^2(t) \lesssim e^{-\beta_1(t)},
\]
which is sufficient to finish the proof of \eqref{eq:z7} using \eqref{eq:U5}
and integrating \eqref{eq:z9} again.
\end{proof}
 
For convenience, we set
\[
\beta_0(t) = \frac 12 \beta_1(t).
\]

\begin{lemma}\label{le:te}
Let $m\in \{1,\ldots,K+1\}$.
Assume the same conditions as in Lemma \ref{le:td}, except that instead of \eqref{eq:z1},
suppose
\begin{equation}\label{eq:*1}
N^2 + B^2 + q^2 \leq C_0 e^{-2\beta_{m-1}(t)}.
\end{equation}
If $m\geq 2$, assume further that for all $k\in \{1,\ldots,m-1\}$,
\begin{equation}\label{eq:*9}
\lim_{t\to+\infty} e^{\beta_k(t)} \alpha_k^-(t)=0.
\end{equation}
Then, if $m\in \{1,\ldots,K\}$,
\begin{equation}\label{eq:*2}
N^2 + B^2 + q^2 \lesssim e^{-2\beta_{m}(t)},
\end{equation}
 and if $m=K+1$, $N= B= q=0$.
\end{lemma}

\begin{proof}

Set
\[
\tilde v(t) = \sup_{s\geq t} \left( e^{\beta_{m-1}(s)} \left(N(s)+s^{-\frac 12} q(s)\right)\right).
\]
By integration of \eqref{eq:z9} on $[s,+\infty)$, for $s>T$, we obtain
\begin{equation*}
|\alpha_k^\Lambda(s)| + |\alpha_k^{(j)}(s)|
\lesssim \int_s^{+\infty} (N(\tau)+\tau^{-2}q) d\tau  \lesssim \int_s^{+\infty} e^{-\beta_{m-1}(\tau)} \tilde v(\tau)  d\tau\ \lesssim e^{-\beta_{m-1}(s)} \tilde v(s).
\end{equation*}
Thus,
\begin{equation}\label{eq:V3}
q \lesssim  e^{-\beta_{m-1} } \tilde v .
\end{equation}
In particular,
\begin{equation}\label{eq:vi}
\sup_{s\geq t} \left(s^{-1}  e^{ 2 \beta_{m-1}(s)}q^2(s) \right)
\lesssim  t^{-1}\tilde v^2(t).
\end{equation}

Let $k\in \{1,\ldots,K\}$ with $k\leq m-1$ (this case is void for $m=1$).
From \eqref{eq:z4},
\[
\frac{d}{dt}\left( e^{\beta_k} \alpha_k^- \right)(s)
\lesssim s^{-2} e^{\beta_k(s)} \left( N + q \right)(s).
\]
Using \eqref{eq:*9}, we integrate this differential inequality on $[s,+\infty)$ and obtain
\[
|\alpha_k^-(s)| \lesssim e^{-\beta_k(s)}
\int_s^\infty \tau^{-2} e^{\beta_k(\tau)} \left( N + q \right)(\tau) d\tau.
\]
By the definition of $\tilde v$ and \eqref{eq:vi},
\begin{equation}\label{eq:V0}
|\alpha_k^-(s)| \leq  e^{-\beta_k(s)} \int_s^{+\infty} \tau^{-2} e^{\beta_k(\tau)}e^{-\beta_{m-1}(\tau)} \tilde v(\tau) d\tau
\lesssim s^{-1} e^{-\beta_{m-1}(s)}\tilde v(s),\quad k\leq  m-1.
\end{equation}

Let $k\in \{1,\ldots,K\}$ with $k\geq m$ (this case is void for $m=K+1$).
By \eqref{eq:z4}, the definition of $\tilde v$ and \eqref{eq:V3}, we obtain after integration on $[T,s]$,
for $s\geq T$,
\[
|\alpha_k^-(s)| \lesssim C_k(T) e^{-\beta_k(s)}
+  e^{-\beta_k(s)} \int_T^s \tau^{-2} e^{\beta_k(\tau)}e^{-\beta_{m-1}(\tau)} \tilde v(\tau) d\tau,
\]
where $C_k(T)=e^{\beta_k(T)} |\alpha_k^{-}(T)|$.
Let $T\leq t\leq s$,
using the Chasles relation
\begin{align*}
(\alpha_k^-(s))^2
& \lesssim C_k^2(T) e^{-2\beta_k(s)} 
+  e^{-2\beta_k(s)} \left(\int_T^t \tau^{-2} e^{\beta_k(\tau)}
e^{-\beta_{m-1}(\tau)}
\tilde v(\tau) d\tau\right)^2\\
&\quad +e^{-2\beta_k(s)} \left(\int_t^s \tau^{-2} e^{\beta_k(\tau)}
e^{-\beta_{m-1}(\tau)} \tilde v(\tau)d\tau\right)^2.
\end{align*}
Let
\[
\tilde f(t) = \sup_{T\leq \tau \leq t} \left( e^{2\beta_m(\tau)-2\beta_{m-1}(\tau)} \tilde v^2 (\tau)\right).
\]
Thus, for $T\leq t\leq s$,
\begin{equation}\label{eq:V1}
(\alpha_k^-(s))^2
 \lesssim C_k^2(T) e^{-2\beta_k(s)}
 + T^{-2} e^{-2\beta_m(s)} \tilde f(t)+ t^{-2} e^{-2 \beta_{m-1}(s)} \tilde v^2(t)
,\quad k\geq m.
\end{equation}
Moreover, using \eqref{eq:z5}, \eqref{eq:*1} and \eqref{eq:V3}, for any $k\in\{1,\ldots,K\}$, we have
\begin{equation}\label{eq:V2}
|\alpha_k^+(s)| \leq  e^{\beta_k(s)} \int_s^{+\infty} \tau^{-2} e^{-\beta_k(\tau)}
e^{-\beta_{m-1}(\tau)} \tilde v(\tau) d\tau
\lesssim s^{-2} e^{- \beta_{m-1}(s)} \tilde v(s).
\end{equation}
Let $T\leq t\leq s$.
By \eqref{eq:V0}, \eqref{eq:V1} and \eqref{eq:V2}, one has
\begin{equation}\label{eq:V5}
B^2(s)\lesssim 
\sum_{k\geq m} C_k^2(T) e^{-2\beta_k(s)}
 + T^{-2} e^{-2\beta_m(s)} \tilde f(t)+ t^{-2} e^{- 2 \beta_{m-1}(s)} \tilde v^2(t).
\end{equation}
Note that above, and hereafter, the term $\sum_{k\geq m} C_k^2(T) e^{-2\beta_k(s)}$ is zero
in the special case where $m=K+1$.

Now, we insert the estimate \eqref{eq:V3}  into \eqref{eq:z6}, to obtain
\[
- \frac d{dt}(t^{1+\delta} \cK) 
\lesssim t^{-\frac 52+\delta} q^2 + t^{\delta} B^2 
\lesssim t^{-\frac 52+\delta} e^{-2\beta_{m-1}(t)} \tilde v^2(t) + t^{\delta} B^2.
\]
By integrating 
on $[t,+\infty)$, for $t\geq T$ and using \eqref{eq:z8}, \eqref{eq:*1}, we find
\[
t^{1+\delta}\cK(t)
\lesssim t^{-\frac52+\delta} e^{-2 \beta_{m-1}(t)}\tilde v^2(t)
+\int_t^{+\infty} s^{\delta} B^2(s) ds .
\]
Multiplying by $t^{-1-\delta} e^{2\beta_{m-1}(t)}$, we find
\begin{align*}
e^{2 \beta_{m-1}(t)}\cK(t)
&\lesssim t^{-\frac72} \tilde v^2(t)
+t^{-1-\delta} e^{2\beta_{m-1}(t)} \int_t^{+\infty} s^{\delta} B^2(s) ds
\end{align*}
Inserting the estimate on $B$ from \eqref{eq:V5}, we obtain
\begin{align*}
e^{2 \beta_{m-1}(s)} \cK(s) 
&\lesssim t^{-\frac72} \tilde v^2(t)
 +   t^{-1} e^{2 \beta_{m-1}(t)} \sum_{k\geq m} C_k^2(T) e^{-2\beta_k(t)} \\
&\quad 
+ T^{-1} t^{-1} e^{-2\beta_m(t)+2 \beta_{m-1}(t)}\tilde f(t) + t^{-2} \tilde v^2(t) .
\end{align*}
Simplifying and taking the supremum over $[t,+\infty)$,
\begin{multline}\label{eq:V4}
\sup_{s\geq t} \left(  e^{2 \beta_{m-1}(s)} \cK(s)\right)
\lesssim t^{-3} \tilde v^2(t) + t^{-1} e^{ 2 \beta_{m-1}(t)}\sum_{k\geq m} C_k^2(T)e^{-2\beta_k(t)}\\
+ T^{-2} t^{-1} \tilde f(t) e^{-2\beta_m(t)+2 \beta_{m-1}(t)}  .
\end{multline}
By \eqref{eq:z2} and \eqref{eq:vi}, \eqref{eq:V5}, \eqref{eq:V4},
\begin{multline*}
\sup_{s\geq t} \left( e^{2 \beta_{m-1}(s)}  N^2(s) \right)
\lesssim t^{-1} \tilde v^2(t) +  e^{2 \beta_{m-1}(t)}\sum_{k\geq m} C_k^2(T)e^{-2\beta_k(t)}
\\
+ T^{-2}  \tilde f(t) e^{-2\beta_m(t)+2 \beta_{m-1}(t)}.
\end{multline*}
Therefore, by the definition of $\tilde v$ and \eqref{eq:vi}, we obtain
\[
\tilde v^2(t) \lesssim 
t^{-1} \tilde v^2(t) +  e^{2\beta_{m-1}(t)}\sum_{k\geq m} C_k^2(T)e^{-2\beta_k(t)}
+ T^{-2} \tilde f(t) e^{-2\beta_m(t)+2 \beta_{m-1}(t)},
\]
and so for $T$ sufficiently large
\[
\tilde v^2(t) \lesssim  e^{ 2 \beta_{m-1}(t)} \sum_{k\geq m} C_k^2(T)e^{-2\beta_k(t)}
+ T^{-2}  \tilde f(t) e^{-2\beta_m(t)+ 2 \beta_{m-1}(t)}.
\]
Recalling now the definition of $f$, 
we obtain
\[
\tilde f(t) \lesssim \sup_{T\leq \tau \leq t} \left(  \sum_{k\geq m} C_k^2(T) 
+T^{-1} \tilde f(\tau)\right)
\lesssim \sum_{k\geq m} C_k^2(T) + T^{-2}\tilde f(t).
\]
Therefore, for $T>0$ large enough, we have $\tilde f(t)\lesssim \sum_{k\geq m} C_k^2(T)\lesssim 1$.
If $m=K+1$, this implies that $\tilde f=0$.
If $m\leq K$, then by \eqref{eq:V3},
\[
 N^2(t) + t^{-1} q^2(t) \lesssim 
 e^{-2 \beta_{m-1}(t)} \tilde v^2(t) \lesssim  e^{-2\beta_m(t)}.
\]
The estimate on $B^2$ follows from \eqref{eq:V5}.
\end{proof}

\subsection{Classification proof}
Since we will use decomposition with respect to other strong
multi-solitons later, we clarify the notation of Section \ref{S:4} by setting
\[
\vec a = \vec a[S],\quad \alpha_k^\Lambda = \alpha_k^\Lambda [S],
\quad \alpha_k^{(j)} = \alpha_k^{(j)} [S],
\quad \alpha_k^\pm = \alpha_k^\pm [S],\quad
q=q[S],\quad B = B[S]
\]
to denote the elements of the decomposition of the multi-soliton $\vec u$ with respect to the strong multi-soliton $\vec S$.
\begin{lemma}\label{le:ex}
There exist $C_0>0$ such that for all $t>T$,
\begin{equation}\label{eq:fe}
\|\vec a[S]\|_{\mathcal E} + B[S](t) + q[S](t)\lesssim e^{-\frac 12\beta_1(t)}.
\end{equation}
\end{lemma}
\begin{proof}
We apply Lemma \ref{le:td} with
$N(t) = \|\vec a[S](t) \|_\cE$
and the same notation otherwise.
We verify all the assumptions of Lemma \ref{le:td}:
Estimate \eqref{eq:z2} follows from \eqref{eq:CO}.
Estimate \eqref{eq:z8} follows from \eqref{eq:HH} and \eqref{eq:eG}.
Then, estimates \eqref{eq:z4}, \eqref{eq:z5} and \eqref{eq:z9} follow from \eqref{eq:ke0} and \eqref{eq:zo0}.
Lastly, \eqref{eq:z6} is a consequence of \eqref{eq:rj}.
Moreover, estimate \eqref{eq:im} implies \eqref{eq:z1}.
Applying Lemma \ref{le:td}, we deduce \eqref{eq:fe}.
\end{proof}

After applying Lemma \ref{le:td} in the proof of Lemma \ref{le:ex}, we apply Lemma \ref{le:te} with $m=1$.
Indeed, \eqref{eq:fe} and the definition of $\beta_0$ implies \eqref{eq:*1} with $m=1$.
Using Lemma \ref{le:te} for $m=1$, we obtain
\begin{equation}\label{eq:P1}
\|a[S](t)\|_\cE+B[S](t)+q[S](t)\lesssim e^{-\beta_1(t)},
\end{equation}
thus strictly improving the exponential decay obtained in \eqref{eq:fe}.

At this stage, an important observation is that by \eqref{eq:zo0} and \eqref{eq:P1}, one has
\[
\left| \frac d{dt} \left( e^{\beta_1} \alpha_1^-[S] \right) \right|
\lesssim t^{-2}.
\]
Thus, there exists a real $A_1$ such that
\begin{equation}\label{eq:a1}
A_1 = \lim_{t\to+\infty} e^{\beta_1(t)} \alpha_1^-[S](t)
\end{equation}
and
\begin{equation}\label{eq:B1}
\left| e^{\beta_1(t)} \alpha_1^-(t) -A_1  \right| \lesssim t^{-1}.
\end{equation}
Now, we consider the strong multi-soliton $S_{A_1}$ constructed in the proof of Theorem \ref{th:1}
and the decomposition of the multi-soliton $u$ around the multi-soliton
$S_{A_1}$ as given in Section \ref{S:4}.
We denote by
\[
a[S_{A_1}], \quad  \alpha_k^\Lambda [S_{A_1}],
\quad \alpha_k^{(j)} [S_{A_1}],
\quad \alpha_k^\pm [S_{A_1}],\quad q[S_{A_1}],\quad B[S_{A_1}]
\]
the corresponding quantities.
By \eqref{eq:A1} and \eqref{eq:B1}
\[
\lim_{t\to+\infty} e^{\beta_1(t)} \alpha_{1}^-[S_{A_1}](t)=0.
\]
Thus, we apply Lemma \ref{le:te} with $m=2$ to get
\[
\|a[S_{A_1}](t)\|_\cE+B[S_{A_1}](t)+q[S_{A_1}](t)\lesssim e^{-\beta_2(t)}.
\]
As before, we define
\begin{equation}\label{eq:a2}
A_2 = \lim_{t\to+\infty} e^{\beta_2(t)} \alpha_2^-[S_{A_1}](t)
\end{equation}
and the strong multi-soliton $S_{A_1,A_2}$.
Using this, we apply Lemma \ref{le:te} with $m=2$ again.
to get 
\[
\|a[S_{A_1,A_2}](t)\|_\cE+B[S_{A_1,A_2}](t)+q[S_{A_1,A_2}](t)\lesssim e^{-\beta_3(t)}.
\]

Iterating this argument, defining recursively $A_m$ for any $m\in \{3,\ldots,K\}$ and applying
Lemma~\ref{le:te} with $m\in\{3,\ldots,K+1\}$ ($m=K+1$ is a special case),
we construct a strong multi-soliton 
$S_{\bf A}=S_{A_1,\ldots,A_K}$ such that the decomposition of the solution $u(t)$ in terms
of $S_{\bf A}$ satisfies $a[S_{\bf A}]=0$ and $B[S_{\bf A}]=q[S_{\bf A}]=0$ which implies
$u(t)=S_{\bf A}$.

\end{document}